\documentclass[10pt]{amsart}
\usepackage{amsmath, amssymb, amsthm, latexsym, amscd, enumerate, MnSymbol}
\usepackage{nicefrac,xspace,tikz,multicol}
\usepackage[all]{xy}
\usepackage{graphicx}
\usepackage{caption}
\usepackage[labelfont={sf},font={small},
  labelsep=space]{caption}
 \newlength{\baseunit}               
 \newcount{\numlines}                
\newtheorem{theorem}{Theorem}

\newtheorem{lemma}[theorem]{Lemma}
\newtheorem{remark}[theorem]{Remark}
\newtheorem{prop}[theorem]{Proposition}
\newtheorem{corollary}[theorem]{Corollary}
\newtheorem{ex}[theorem]{Example}

\newtheorem{conjecture}[theorem]{Conjecture}
\newtheorem{definition}[theorem]{Definition}
\newtheorem{define}[theorem]{Definition}
\newtheorem{THEOREM}{Theorem}

\newcommand{\cA}{{\mathcal A}}

\newcommand{\cL}{{\mathcal L}}

\newcommand{\cO}{{\mathcal O}}

\newcommand{\cU}{{\mathcal U}}

\def\down{\vee}
\def\up{\wedge}

\newcommand{\mg}{\mathfrak{g}}
\newcommand{\mh}{\mathfrak{h}}
\newcommand{\mb}{\mathfrak{b}}

\newcommand{\mC}{\mathbb{C}}

\newcommand{\mZ}{\mathbb{Z}}

\newcommand{\la}{\lambda}

\newcommand{\HOM}{\operatorname{Hom}}

\newcommand{\op}{\operatorname}

\newcommand{\Kar}{\operatorname{Kar}}
\newcommand{\wt}{\operatorname{wt}}

\newcommand{\DEG}{\operatorname{deg}}

\newcommand{\surj}{\mbox{$\rightarrow\!\!\!\!\!\rightarrow$}}

\DeclareMathOperator{\Hom}{Hom}   
\DeclareMathOperator{\End}{End}   \DeclareMathOperator{\Id}{Id}
   \DeclareMathOperator{\Mod}{mod}
 \DeclareMathOperator{\Ext}{Ext}  
\DeclareMathOperator{\gmod}{gmod}   

\DeclareMathOperator{\Hstar}{H^*}
\DeclareMathOperator{\Gr}{Gr}

\newcommand{\C} {\mathbb{C}}

\begin{document}
\pagestyle{plain}
\title{A Lie theoretic categorification of the coloured Jones polynomial}
\author{Catharina Stroppel}
\author{Joshua Sussan}
\begin{abstract}
We use the machinery of categorified Jones-Wenzl projectors to construct a categorification of a type A Reshetikhin-Turaev invariant of oriented framed tangles where each strand is labeled by an arbitrary finite-dimensional representation.   As a special case, we obtain a categorification of the coloured Jones polynomial of links.
\end{abstract}

\maketitle \tableofcontents

\section{Introduction} 
The discovery of the quantum group in the 1980's led to remarkable applications in topology, giving rise to invariants of tangles and 3-manifolds.  While this is still an active area of study, a different direction in the application of quantum groups to topology was proposed by Crane and Frenkel \cite{CF}.  The philosophy of categorifcation is that algebraic structures giving rise to topological invariants should be replaced by structures of a higher categorical level and these higher structures should give rise to topological invariants of one dimension higher.

The first concrete realization of this philosophy was constructed by Khovanov  who assigned a complex of graded vector spaces to an oriented link embedded in $ S^3. $  As a result, one gets a homological invariant of links whose graded Euler characteristic is the Jones polynomial. It was later proved by Khovanov \cite{Khov3} and independently by Jacobsson \cite{Jac} and then Lie theoretically in \cite{StrTQFT} that a surface bounded by two links induces a chain map of the Khovanov complexes which in turn becomes a topological invariant of the surface, up to signs.  The sign inconsistency was subsequently fixed in various ways: Blanchet \cite{Bla} via singular foams, Clark-Morrison-Walker \cite{CMW} via cobordisms with disorientation lines, and Ehrig-Stroppel-Tubbenhauer \cite{ESTarc} via a sign modified arc algebra.
Seidel and Smith \cite{SeidelSmith} reconstructed Khovanov homology using symplectic geometry. Using the geometry of the affine Grassmanian \cite{CK}, Cautis and Kamnitzer gave an algebraic-geometric formulation of Khovanov homology.

Another approach to the categorification of the Jones polynomial was developed by Bernstein, Frenkel, and Khovanov \cite{BFK}.
They proposed a categorification of the tangle invariant constructed by Reshetikhin and Turaev where each strand of the tangle is labeled by the two-dimensional irreducible representation $ V_1 $ of $ \mathcal{U}_q(\mathfrak{sl}_2).$
The main construction of their work is a category whose complexified Grothendieck group is isomorphic to a tensor power of the standard two-dimensional representation of $\mathcal{U}(\mathfrak{sl}_2)$ along with functors on this category which induce an action of the enveloping algebra and the Temperley-Lieb algebra on the Grothendieck group. Relations in these algebras are upgraded to isomorphisms of functors.
The categories involved are categories of highest weight representations of the Lie algebra $ \mathfrak{gl}_n$, the so-called category $ \mathcal{O}$. Their conjectures were proved and categorification of the tangle invariant was completed by the first author \cite{StrDuke}.  The Koszul grading that these highest weight categories possess was used to construct a categorification of the action of the quantum group and of the Temperley-Lieb algebra (such that the generic parameter $q$ corresponds to a shift in the grading).  A proof that the functor associated to a given tangle diagram is invariant under the Reidemeister moves was given,
and hence a categorification of the Reshetikhin-Turaev tangle invariant established.
In \cite{StrSpringer} a functor to Khovanov's construction was established. 


The next development in the categorification of the representation theory of $ \mathcal{U}_q(\mathfrak{sl}_2) $ was a categorification of tensor products of arbitrary finite-dimensional representations of $ \mathcal{U}_q(\mathfrak{sl}_2)$ \cite{FKS}. This construction uses categories of Harish-Chandra bimodules.  This program was continued in \cite{FSS1} where a categorification of the Jones-Wenzl projector and $3j$-symbols was given.

In the present paper we use the machinery developed in \cite{FSS1} to construct a categorification of the Reshetikhin-Turaev invariant of oriented framed tangles where each strand is labeled by an arbitrary finite-dimensional representation.  In the special case of a $ (0,0)$-tangle, we get a categorification of the the coloured Jones polynomial.
The proof that we get a tangle invariant relies on the main theorem of \cite{StrDuke}, a study of twisting functors on subcategories of $\mathcal{O}$ given by projectively presented subcategories, along with a categorical characterization of the Jones-Wenzl projector provided in \cite{FSS1}. We assign functors $\hat {\Phi}_{col}(D)$ to each oriented framed tangle diagram $D$ and establish the following functor valued invariant $\hat {\Phi}_{col}(_-)\langle 3\gamma(\op{cab}(_-)) \rangle$:

\begin{THEOREM}
\label{main}
Let $ D_1 $ and $ D_2 $ be two diagrams for an oriented, framed, coloured tangle $ T $ from points coloured by $ \bf d
$ to points coloured by $ \bf e $.  Then
$$ \hat {\Phi}_{col}(D_1)\langle 3\gamma(\op{cab}(D_1)) \rangle \cong
\hat{\Phi}_{col}(D_2)\langle 3\gamma(\op{cab}(D_2)) \rangle. $$
The induced morphism on the Grothendieck group is the morphism of modules for the quantum group $\mathcal{U}_q(\mathfrak{sl}_2)$ associated to $T$ by  Reshetikhin and Turaev.
\end{THEOREM}

The first categorification of the coloured Jones polynomial was constructed by Khovanov \cite{Khovcolor}.  His construction is based upon his earlier categorification of the Jones polynomial and a certain cabling procedure which differs from the one given here.
His invariant is different from ours since coloured Khovanov homology is non-zero in finitely many degrees whereas our construction lies in an unbounded derived category and a computation for the unknot shows that there is in general non-zero homology in infinitely many degrees.  On the level of graded Euler characteristics, the constructions of course coincide.
Beliakova and Wehrli \cite{BW}, extended \cite{Khovcolor} using Bar-Natan's formulation, and also constructed a Lee deformation \cite{Lee}.

 Webster \cite{Webcombined} constructed a categorification of the Reshetikhin-Turaev invariant for quantum groups attached to arbitrary simple complex Lie algebras and in particular for $\mathcal{U}_q(\mathfrak{sl}_2)$. It appears via a chain of non-trivial equivalences to be equivalent to the construction here. Webster's work is based upon diagrammatically defined algebras generalizing cyclotomic quotients of algebras constructed by Khovanov and Lauda, and independently Rouquier \cite{KL}, \cite{L}, \cite{Rou2}. For the connection between KLR-algebras, Khovanov's algebra and category $\mathcal{O}$ see \cite{BK}, \cite{BS3}, \cite{StrSpringer}.

Cooper and Krushkal \cite{CoKr} categorified the Jones-Wenzl projector and the coloured Jones polynomial using  Bar-Natan's formulation of Khovanov homology \cite{BN}. Their categorification of the Jones-Wenzl projector agrees with ours and the earlier construction in \cite{FKS} up to Koszul duality, \cite{SS}. Rozansky also constructed a categorification of the projector coming from the knot homology of torus braids \cite{Roz}.  This construction was generalized by Cautis \cite{Ca}.

If we allow only fundamental representations as colours then our construction gives an invariant of oriented tangles \cite{StrDuke}. The proof of Theorem~\ref{main} and its straightforward $\mathfrak{sl}_k$ generalisation rely substantially on two ideas: i) Every irreducible representation is the quotient of some tensor product of fundamental representation. ii) The projection operator onto this summand slides through the braiding maps and through evaluation and coevaluation maps (corresponding to cup and cap tangles). We categorify these properties. They reduce the check of Reidemeister moves to the case of fundamental colours, except for the first Reidemeiser move. 
Here, only the weaker version for framed \emph{framed} oriented tangles can be shown. 

The coloured Jones polynomial plays an important role in the construction of the Reshetikhin-Turaev 3-manifold invariant \cite{RT}.  One fixes a 3-manifold and a framed link such that surgery on the link gives the 3-manifold. A summation over all colours assigned to the link components of the corresponding coloured Jones polynomials is invariant under the Kirby moves, so it is in fact a 3-manifold invariant, \cite{RT}.  In order to avoid an infinite summation, Reshetikhin and Turaev consider however a quantum group at a root of unity, where there are only finitely many finite-dimensional irreducible representations.   Thus a categorification of this invariant, as well as the Turaev-Viro invariant, requires a categorification of the Jones-Wenzl projector at a root of unity.  Here, together with \cite{FSS1}, we establish the basics for a categorification of the representation theory for generic $q$.

\subsection*{Overview of the paper}

In Section ~\ref{U2basics}, we review the representation theory of $ \mathcal{U}_q(\mathfrak{sl}_2) $ and show how one can construct tangle invariants and the coloured Jones polynomial from it.  In Section ~\ref{catproducts}, we summarize the Bernstein-Frenkel-Khovanov-Stroppel categorification of tensor products of arbitrary finite-dimensional representation of $ \mathcal{U}_q(\mathfrak{sl}_2) $.
The Bernstein-Gelfand functors between category $ \mathcal{O} $ and categories of Harish-Chandra bimodules are reviewed and used to categorify the inclusion and projection morphisms of tensor products of representations.
We give a new description of the categorical Jones-Wenzl projector in Section \ref{secJWcomplex}.  This formulation utilizes Lauda's $\mathfrak{sl}_2$ category as well as some facts about the cohomology of Grassmannians.
In Section ~\ref{catuncoloredrt} we review the categorification of the Jones polynomial conjectured in \cite{BFK} and proved in \cite{StrDuke}.
A study of twisting functors and their adjoints acting on subcategories of $ \mathcal{O} $ is given in Section ~\ref{twistprojpres}.  The main result of this section is that certain compositions of twisting functors map a derived category of projectively presented objects to another such category and consequently a categorification of the action of the $R$-matrix.
The main theorem and its proof are included in Section ~\ref{reidemeistersection}.   We use Sections ~\ref{catuncoloredrt} and ~\ref{twistprojpres} to show that if two diagrams are related by a Reidemeister move, then the assigned functors are isomorphic.

\subsection*{Acknowledgements}
The authors are very grateful to Sabin Cautis, Mikhail Khovanov, and You Qi for helpful conversations.

J.S. is partially supported by the NSF grant DMS-1807161 and PSC CUNY Award 64012-00 52. 
J.S. gratefully acknowledges support from the Simons Center for Geometry and Physics, Stony Brook University at which some of the research for this paper was performed.
He would also like to thank the Max Planck Institute for Mathematics in Bonn and the Hausdorff Center of Mathematics for its hospitality during the early stages of this project.

\section{Representation theory of $ \mathcal{U}_q(\mathfrak{sl}_2) $}
\label{U2basics}
In this section we recall basic structures on the representation theory of $\mathcal{U}_q(\mathfrak{sl}_2)$ from \cite{Kassel} and \cite{KaLi} and review the corresponding Reshetikhin-Turaev invariant for oriented framed tangles. All this will later be categorified. 

\subsection{Representations}
Let $\mC(q)$ be the field of complex rational functions in an indeterminate $q$. Later on we will also work with the ring of integral formal power series of finite order in $q$ which we denote by $\mZ((q))$. An element of $\mZ((q))$ is a formal Laurent series, $\sum_{i\in\mathbb{Z}}a_iq^i$, in $q$ with coefficients $a_i\in\mathbb{Z}$ with $a_i=0$ for almost all $i<0$.
\begin{define} \label{defofsl2}
Let $\mathcal{U}_q=\mathcal{U}_q(\mathfrak{sl}_2) $ be the associative algebra over $\mathbb{C}(q)$ generated by $ E, F, K, K^{-1} $ satisfying the
relations:
\vspace{-3mm}
\begin{multicols}{2}
\begin{enumerate}[(i)]
\item $ KK^{-1} = K^{-1}K=1$, \item $ KE = q^2 EK $, \item $ KF = q^{-2} FK $, \item $ EF-FE = \frac{K-K^{-1}}{q-q^{-1}} $.
\end{enumerate}
\end{multicols}
\end{define}

Let $ [k]=\sum_{j=0}^{k-1} q^{k-2j-1}$ and $ {n \brack {k}} = \frac{[n]!}{[k]![n-k]!}. $ Let $ \bar{V}_n $ be the unique (up to isomorphism) irreducible module for $\mathfrak{sl}_2$ of dimension $ n+1$. Denote by $V_n $ its quantum analogue (of type I), that is
the irreducible $ \mathcal{U}_q(\mathfrak{sl}_2)$-module with basis $ \lbrace v_0, v_1, \ldots, v_{n} \rbrace $ such that
\begin{equation}
\label{irreddef}
K^{\pm 1} v_i= q^{\pm (2i-n)} v_i\quad\quad Ev_i =[i+1] v_{i+1}\quad\quad F v_i = [n-i+1] v_{i-1}.
\end{equation}

There is a unique bilinear form $\langle\quad,\quad\rangle' \;\colon\; V_n \times V_n \rightarrow \mathbb{C}(q)$
which satisfies
\begin{equation}
\label{scalarprod} \langle v_k, v_l \rangle' = \delta_{k,l}q^{k(n-k)}{n \brack k}.
\end{equation}
The vectors $\lbrace v^0, \ldots, v^n \rbrace $ where  $v^i =\frac{1}{{n \brack i}} v_i$ form {\it the dual standard basis} characterised by $ \langle v_i, v^i \rangle' = q^{i(n-i)}$. Recall that $\mathcal{U}_q$ is a Hopf algebra with comultiplication
\begin{equation}
\label{comult}
\triangle(E)=1\otimes E + E\otimes K^{-1},\quad\quad\quad\triangle(F)=K\otimes F+F\otimes 1,\quad\quad
\triangle(K^{\mp 1})=K^{\mp1}\otimes K^{\mp1}.
\end{equation}
and antipode $S$ defined as $S(K)=K^{-1}$, $S(E)=-EK$ and $S(F)=-K^{-1}F$. Thus, the tensor product $V_{\bf d}:=V_{d_1} \otimes \cdots\otimes V_{d_r} $ has the structure of a $\mathcal{U}_q$-module with {\it standard basis} $\lbrace v_{\bf a}=v_{a_1} \otimes \cdots \otimes v_{a_r} \rbrace $ where $ 0 \leq a_j \leq d_j $ for $ 1 \leq j \leq r$. Denote by $v^{\bf a}=v^{a_1} \otimes \cdots \otimes v^{a_r}$ the corresponding tensor products of dual standard basis elements.

\begin{remark}{\rm 
We made a specific convenient choice of comultiplication.  Others might be taken as well.  For instance, if we set $\widetilde{E}=KE$, 
$\widetilde{F}=FK$, 
$\widetilde{K}=K$ and $\widetilde{v}_0=v_0$,  $\widetilde{v}_1 = qv_1$ in $V_1$, then $\widetilde{E}, \widetilde{F}, \widetilde{K}$ satisfy the relations from Definition \ref{defofsl2} and the structure coefficients from \eqref{irreddef} hold for the tilde-version.  Formulas for the comultiplication of these new generators are given by:
$\triangle(\widetilde{E})=\widetilde{E} \otimes 1 + K \otimes \widetilde{E} $ and
$\triangle(\widetilde{F})=1 \otimes \widetilde{F} + \widetilde{F} \otimes K^{-1}$. 
When we categorify these structures later, the functors corresponding to $\widetilde{E}$ and $\widetilde{F}$ agree with those for $E$ and $F$ up to a shift (encoded by $K$).
}
\end{remark}

\begin{figure}[t]
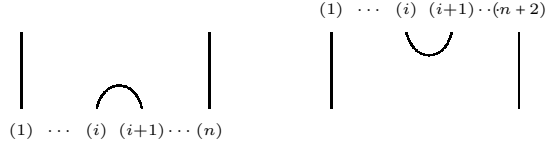

\begin{equation*}
 \xy (0,0);(0,10)*{}**\dir{-};(0,10); (10,0)*{};(16,0)*{}**\crv{(11,4)&(15,4)}; (25,0)*{};(25,10)*{}**\dir{-};(25,10)*{};
(0,-3)*{\mbox{\tiny ($1$)}};(5,-3)*{\mbox{\tiny $\ldots$}};(10,-3)*{\mbox{\tiny ($i$)}};(16,-3)*{\mbox{\tiny
($i$+$1$)}};(21,-3)*{\mbox{\tiny$\ldots$}};(25,-3)*{\mbox{\tiny ($n$)}};
\endxy
\hspace{0.5in} \xy (0,0)*{};(0,10)*{}**\dir{-};(0,10)*{}; (10,10)*{};(16,10)*{}**\crv{(11,6)&(15,6)};(10,10)*{};(16,10)*{};
(25,0)*{};(25,10)*{}**\dir{-};(25,10)*{};
(0,13)*{\mbox{\tiny ($1$)}};(5,13)*{\mbox{\tiny $\ldots$}};(10,13)*{\mbox{\tiny
($i$)}};(16,13)*{\mbox{\tiny ($i$+$1$)}};(21,13)*{\mbox{\tiny$\ldots$}};(25,13)*{\mbox{\tiny ($n+2$)}};
\endxy
\end{equation*}
\vspace{-2mm}
\caption{{\small The intertwiners $\cap_{i,n}$ and $\cup_{i,n}$.}}
\label{fig:capcup}
\vspace{-2mm}
\end{figure}
\vspace{-2mm}
\subsection{Jones-Wenzl projector and intertwiners}
Next we define morphisms between various tensor powers of $V_1$ which intertwine the action of the quantum group, namely $\cup\colon \mathbb{C}(q) \rightarrow V_1^{\otimes 2}$ and $ \cap \colon V_1^{\otimes 2} \rightarrow \mathbb{C}(q)$ which are given on the standard basis by
\begin{eqnarray}
\label{defcupcap}
\cup(1)=v_1\otimes v_0 -qv_0 \otimes v_1,
&&
 \cap(v_i \otimes v_j)=
 \begin{cases}
 0 &\text{if $i=j$}\\
 1 &\text{if $i=0, j=1$}\\
 -q^{-1}&\text{if $i=1, j=0$.}
 \end{cases}
\end{eqnarray}

We define $\cap_{i,n}=\Id^{\otimes (i-1)} \otimes \cap \otimes \Id^{\otimes
    (n-i-1)}$ and $\cup_{i,n}=\Id^{\otimes (i-1)} \otimes \cup \otimes \Id^{\otimes (n-i+1)}$  as $\cU_q$-morphisms from $V_1^{\otimes n}$ to $V_1^{\otimes (n-2)}$, respectively  $V_1^{\otimes (n+2)}$.
Let $C:=\cup\circ\cap$ be their composition and $C_i:=C_{i,n}:= \cup_{i,n-2}\circ\cap_{i,n}$.
We depict the cap and cup intertwiners graphically in Figure \ref{fig:capcup} (reading the diagram from bottom to top), so that $\cap\circ\cup$ is just a circle. In fact, finite compositions of these elementary morphisms generate the $\mathbb{C}(q)$-vector space of all intertwiners, see e.g. \cite[Section 2]{FK}.

If we encode a basis vector $v_{\bf d}$ of $V_1^{\otimes n}$ as a sequence of $\up$'s and $\down$'s according to
the entries of ${\bf d}$, where $0$ is turned into $\down$ and $1$ is turned into $\up$, then the formulas in \eqref{defcupcap} can be symbolized by\\

\begin{picture}(45,28)
\put(9.8,0){$\scriptstyle\up$}
\put(32.7,0){$\scriptstyle\up$}
\put(35,14){\line(0,-1){11.5}}
\put(12,14){\line(0,-1){11.5}}
\put(23.5,14){\oval(23,23)[t]}
\put(41,12){$=0=$}\put(101,12){$,$}
\put(69.6,2){$\scriptstyle\down$}
\put(92.6,2){$\scriptstyle\down$}
\put(95,14){\line(0,-1){11.5}}
\put(72,14){\line(0,-1){11.5}}
\put(83.5,14){\oval(23,23)[t]}
\put(109.8,2){$\scriptstyle\down$}
\put(132.8,0){$\scriptstyle\up$}
\put(135,14){\line(0,-1){11.5}}
\put(112,14){\line(0,-1){11.5}}
\put(123.5,14){\oval(23,23)[t]}
\put(141,12){$=1$,\quad}

\put(169.7,0){$\scriptstyle\up$}
\put(192.8,2){$\scriptstyle\down$}
\put(195.0,14){\line(0,-1){11.5}}
\put(171.9,14){\line(0,-1){11.5}}
\put(183.5,14){\oval(23,23)[t]}
\put(201,12){$=-q^{-1}$,\quad}

\put(243.5,14){$\bigcup(1)=$}
\put(279.8,22){$\scriptstyle\up$}
\put(292.6,25){$\scriptstyle\down$}
\put(295,6){\line(0,1){19.5}}
\put(282,6){\line(0,1){19.5}}
\put(301,12){$-\;q$}
\put(319.6,25){$\scriptstyle\down$}
\put(332.8,22){$\scriptstyle\up$}
\put(335,6){\line(0,1){19.5}}
\put(322,6){\line(0,1){19.5}} 
\put(342,12){$.$}
\end{picture} 
\vspace{4mm}

The symmetric group $\mathbb{S}_n$ acts transitively on the set of $n$-tuples with $i$ ones and $n-i$ zeroes. The stabilizer of ${\bf d}_{\op{dom}}:=(\underbrace{1, \ldots, 1}_{i},
\underbrace{0, \ldots, 0}_{n-i})$ is $\mathbb{S}_{i}\times \mathbb{S}_{n-i}$. By sending the identity element $e$ to ${\bf d}_{\op{dom}}$ we fix for the rest of the paper a bijection between shortest coset representatives in $\mathbb{S}_n/\mathbb{S}_{i}\times \mathbb{S}_{n-i}$ and these tuples ${\bf d}$. We denote by $w_0$ the longest element of $\mathbb{S}_n$ and by $w_0^i$ the longest element in $\mathbb{S}_{i}\times \mathbb{S}_{n-i}$. By abuse of language we denote by $l(\bf{d})$ the (Coxeter) {\it length} of $\bf{d}$ meaning the Coxeter length of the corresponding element in $\mathbb{S}_n$. We denote by $|{\bf d}|$ the numbers of ones in ${\bf d}$.
\begin{define}{\rm
For $ {\bf a}=(a_1, \ldots, a_n)\in\{0,1\}^n$ let $ v_{\bf a}=v_{a_1} \otimes \cdots \otimes v_{a_n}\in V_1^{\otimes n}$ be the
corresponding basis vector.\nopagebreak
\begin{itemize}
\item Let $ \pi_n \colon
V_1^{\otimes n} \rightarrow V_n $ be given by the formula
\begin{eqnarray}
\label{pn} \pi_n(v_{\bf a}) = q^{-l({\bf a})}\frac{1}{{n \brack |{\bf a}|}} v_{|{\bf a}|}=q^{-l({\bf a})}v^{|{\bf a}|}
\end{eqnarray}
where $l({\bf a})$ is equal to the number of pairs $(i,j)$ with $i<j$ and $a_i<a_j.$ 
This gives the {\it projection} $ \pi_{i_1} \otimes
\cdots \otimes \pi_{i_r} \colon V_1^{\otimes(i_1+\cdots+i_r)} \rightarrow V_{i_1} \otimes \cdots \otimes V_{i_r}. $
\item We denote by $ \iota_n \colon V_n \rightarrow V_1^{\otimes n} $ the intertwining map
\begin{eqnarray}
\label{in}
 v_k \mapsto \sum_{|{\bf a}|=k} q^{b({\bf a})} v_{\bf a}
\end{eqnarray}
where $b({\bf a})=|{\bf a}|(n-|{\bf a}|)-l({\bf a})$, i.e. the number of pairs $ (i,j) $ with $ i < j $ and $ a_i > a_j$.
Define the {\it inclusion} $ \iota_{i_1} \otimes \cdots \otimes
\iota_{i_r} \colon V_{i_1} \otimes \cdots \otimes V_{i_r} \rightarrow V_1^{\otimes (i_1+\cdots+i_r)}.
$
\end{itemize}
}
\end{define}
The composite $p_n=\iota_n\circ\pi_n$ is the {\it Jones-Wenzl projector}. We symbolize the projection, inclusion and the projector as follows
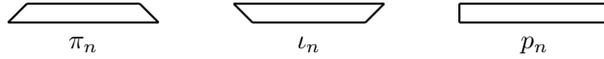
\begin{equation*}
  \begin{aligned}
\begin{tikzpicture}
\draw (0,0) -- (2,0)[thick];
\draw (0,0) -- (.25,.25)[thick];
\draw (2,0) -- (1.75, .25)[thick];
\draw (.25, .25) -- (1.75, .25)[thick];

\draw (3.25,0) -- (4.75,0)[thick];
\draw (3.25,0) -- (3,.25)[thick];
\draw (4.75,0) -- (5,.25)[thick];
\draw (3,.25) -- (5,.25)[thick];

\draw (6,0) -- (8,0)[thick];
\draw (6,0) -- (6,.25)[thick];
\draw (8,0) -- (8, .25)[thick];
\draw (6,.25) -- (8,.25)[thick];

\draw (1, -.3) node{$\pi_n$};
\draw (4, -.3) node{$ \iota_n$};
\draw (7, -.3) node{$ p_n$};
\end{tikzpicture}
\end{aligned}
\vspace{-3mm}
\end{equation*}

\begin{remark}
\label{powerseries}
{\rm
Note that the expression
$\frac{1}{{n \brack |{\bf a}|}}$ can be written in a unique way as $q^{-m}H$, where $m\in\mZ_{\geq 0}$ and $H$
is a formal power series in $q$ with integral coefficients, hence an element in $\mZ((q))$. For instance,
\begin{equation}
\label{quantum2}
\frac{1}{{2 \brack 1}}=\frac{1}{q^{-1}+q}=q^{-1}\frac{1}{1+q^{2}}=q^{-1}(1-q^2+q^4-q^6+\quad\cdots\quad)\in\mZ((q)).
\end{equation}
}
\end{remark}

\begin{ex}
\label{JWex}
{\rm For $n=2$ we have $\iota_2(v_0)=v_0\otimes v_0$,
$\iota_2(v_1)=qv_1\otimes v_0+v_0\otimes v_1$, $\iota_2(v_2)=v_1\otimes v_1$,
and $\pi_2(v_0\otimes v_0)=v_0$, $\pi_2(v_1\otimes v_0)=v^1=[2]^{-1}v_1$,
$\pi_2(v_0\otimes v_1)=q^{-1}v^1=q^{-1}[2]^{-1}v_1$,  $\pi_2(v_1\otimes
v_1)=v_2$. Using formula \eqref{quantum2} we may view these as morphisms of representations defined over $\mC((q))$.
}
\end{ex}

\begin{prop}
\label{charJW}
The endomorphism $p_n$ of $V_1^{\otimes n}$ is the unique $\cU_q(\mathfrak{sl}_2)$-morphism which satisfies {\rm (for $1 \leq i \leq n-1 $):}
\vspace{-3mm}
\begin{multicols}{3}
\begin{enumerate}[(i)]
\item $ p_n \circ p_n = p_n $ \item $C_{i,n} \circ p_n = 0 $ \item $ p_n \circ C_{i,n} = 0$.
\end{enumerate}
\end{multicols}
\end{prop}

\subsection{Reshetikhin-Turaev invariant}
In this section we recall the Resh\-etik\-hin-Turaev-Jones (coloured) tangle invariants. First note that the universal $R$-matrix induces a representation of the braid group with $n$ strands on $V_1^{\otimes n}$. We first define the morphisms $ \Pi_i $ and $ \Omega_i $ corresponding to unoriented elementary braids.
\begin{definition}
Define $ \Pi \colon V_1^{\otimes 2} \rightarrow V_1^{\otimes 2} $ by $ \Pi = -q^{-1} C - q^{-2} \Id=-q^{-1}(C+q^{-1} \Id)$ and then $ \Pi_i =\Id^{\otimes i-1} \otimes \Pi \otimes \Id^{\otimes n-i-1}$.
Define $ \Omega \colon V_1^{\otimes 2} \rightarrow V_1^{\otimes 2} $ by $ \Omega = -q^{} C - q^{2} \Id= -q^{}(C+q^{})\Id$ and then $\Omega_i = \Id^{\otimes i-1} \otimes \Omega \otimes \Id^{\otimes n-i-1}.$

\begin{equation}
\label{crosses} 
\xy (0,0);(0,10)*{}**\dir{-};(0,10); (10,0)*{};(16,10)*{}**\dir{-}; (16,0)*{};(13.6,4)*{}**\dir{-};
(12.4,6)*{};(10,10)*{}**\dir{-}; (25,0)*{};(25,10)*{}**\dir{-};(25,10)*{}; (0,-3)*{\mbox{\tiny (1)}};(5,-3)*{\mbox{\tiny
$\ldots$}};(10,-3)*{\mbox{\tiny $(i)$}};(16,-3)*{\mbox{\tiny ($i$+$1$)}};(21,-3)*{\mbox{\tiny$\ldots$}};(25,-3)*{\mbox{\tiny$(n)$}};
(-8,5)*{\mbox{$\Pi_{i}:$}};
\endxy
\hspace{0.5in} \xy (0,0)*{};(0,10)*{}**\dir{-};(0,10)*{}; (10,10)*{};(16,0)*{}**\dir{-}; (10,0)*{};(12.4,4)*{}**\dir{-};
(13.6,6)*{};(16,10)*{}**\dir{-}; (25,0)*{};(25,10)*{}**\dir{-}; (0,-3)*{\mbox{\tiny (1)}};(5,-3)*{\mbox{\tiny
$\ldots$}};(10,-3)*{\mbox{\tiny ($i$)}};(16,-3)*{\mbox{\tiny ($i$+$1$)}};(21,-3)*{\mbox{\tiny$\ldots$}};(25,-3)*{\mbox{\tiny$(n)$}};
(-8,5)*{\mbox{$\Omega_{i}:$}};
\endxy
\end{equation}
\end{definition}
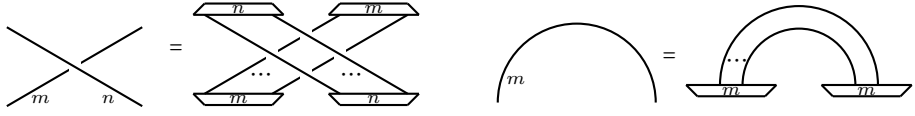
\begin{figure}[t]
\begin{minipage}{.9\textwidth}
\begin{minipage}{0.5\textwidth}
\begin{equation*}
\hspace{-.4cm}
\begin{tikzpicture} [scale=.60]
\draw (.25,-1.75) -- (3.25,0)[thick];
  \path [fill=white] (1.625,-1) rectangle (1.875,-.75);
  \draw  [shift={(-3,-2)}] (4, .4) node{$\scriptstyle{m}$};
  \draw  [shift={(0,-2)}] (2.5, .4) node{$\scriptstyle{n}$};
\draw (3.25,-1.75) -- (.25,0)[thick];
  \draw  [shift={(1.5,-1)}] (2.5, .5) node{$ =$};
\end{tikzpicture}
\begin{tikzpicture} [scale=.60]
\draw (0,0) -- (2,0)[thick];
\draw (0,0) -- (.25,.25)[thick];
\draw (2,0) -- (1.75, .25)[thick];
\draw (.25, .25) -- (1.75, .25)[thick];
\draw (1, .125) node{$\scriptstyle{n}$};
\draw [shift={(-3,-2)}]  (3.25,0) -- (4.75,0)[thick];
\draw [shift={(-3,-2)}] (3.25,0) -- (3,.25)[thick];
\draw [shift={(-3,-2)}](4.75,0) -- (5,.25)[thick];
\draw [shift={(-3,-2)}] (3,.25) -- (5,.25)[thick];
\draw  [shift={(-3,-2)}] (4, .125) node{$\scriptstyle{m}$};
\draw  [shift={(-2.5,-1.5)}] (4, .125) node{$ \cdots$};
\draw  [shift={(-.5,-1.5)}] (4, .125) node{$ \cdots$};
\draw [shift={(0,-2)}]  (3.25,0) -- (4.75,0)[thick];
\draw [shift={(0,-2)}] (3.25,0) -- (3,.25)[thick];
\draw [shift={(0,-2)}](4.75,0) -- (5,.25)[thick];
\draw [shift={(0,-2)}] (3,.25) -- (5,.25)[thick];
\draw  [shift={(0,-2)}] (4, .125) node{$\scriptstyle{n}$};
\draw [shift={(3,0)}](0,0) -- (2,0)[thick];
\draw [shift={(3,0)}](0,0) -- (.25,.25)[thick];
\draw [shift={(3,0)}](2,0) -- (1.75, .25)[thick];
\draw [shift={(3,0)}](.25, .25) -- (1.75, .25)[thick];
\draw [shift={(3,0)}] (1, .125) node{$\scriptstyle{m}$};
\draw (.25,-1.75) -- (3.25,0)[thick];
\path [fill=white] (1.625,-1) rectangle (1.875,-.75);
\draw (1.75,-1.75) -- (4.75,0)[thick];
\path [fill=white] (3.125,-1) rectangle (3.375,-.75);
\path [fill=white] (2.375,-.3125) rectangle (2.625,-.5625);
\path [fill=white] (2.375,-1.1875) rectangle (2.625,-1.4375);
\draw (4.75,-1.75) -- (1.75,0)[thick];
\draw (3.25,-1.75) -- (.25,0)[thick];
\end{tikzpicture}
\end{equation*}
\end{minipage}
\begin{minipage}{0.6\textwidth}
\begin{equation*}
\begin{tikzpicture} [scale=.60]
\draw (-.5,-1) arc (180:0:1.75) [thick];
\draw (-.1,-.5) node{$\scriptstyle{m}$};
\!
 \draw  [shift={(1.5,-1)}] (1.8, 1) node{$ =$};
 \!
\end{tikzpicture}
\begin{tikzpicture} [scale=.60] 
\draw [shift={(-.75,0)}](0,.25) -- (2,.25)[thick];
\draw [shift={(-.75,0)}](0,.25) -- (.25,0)[thick];
\draw [shift={(-.75,0)}](2,.25) -- (1.75, 0)[thick];
\draw [shift={(-.75,0)}](.25, 0) -- (1.75, 0)[thick];
\draw [shift={(-.75,0)}][shift={(3,0)}](0,.25) -- (2,.25)[thick];
\draw [shift={(-.75,0)}][shift={(3,0)}](0,.25) -- (.25,0)[thick];
\draw [shift={(-.75,0)}][shift={(3,0)}](2,.25) -- (1.75, 0)[thick];
\draw [shift={(-.75,0)}][shift={(3,0)}](.25, 0) -- (1.75, 0)[thick];
\draw [shift={(-.75,0)}](1, .125) node{$\scriptstyle{m}$};
\draw [shift={(-.75,0)}][shift={(3,0)}] (1, .125) node{$\scriptstyle{m}$};
\draw [shift={(-.75,0)}](.75,.25) arc (180:0:1.75) [thick];
\draw [shift={(-.75,0)}](1.25,.25) arc (180:0:1.25) [thick];
\draw [shift={(-.75,0)}][shift={(0,0)}] (1.125, .75) node{$\cdots$};
\end{tikzpicture}
\end{equation*}
\end{minipage}
\end{minipage}
\caption{{\small Coloured crossings and caps (similarly cups) via cabling.}}
\label{coloured}
\end{figure}

For $D$ an oriented tangle diagram, let $ \gamma(D) $ be the number of crossings of the types shown in \eqref{orientedcrossings1} minus the number of crossings of the types shown in \eqref{orientedcrossings2}.
\begin{minipage}{7cm}
\begin{equation}
\label{orientedcrossings1} \xy {\ar (10,0)*{};(16,10)*{}**\dir{-}}; {(16,0)*{};(13.6,4)*{}**\dir{-}}; {\ar
(12.4,6)*{};(10,10)*{}**\dir{-}};
\endxy
\hspace{0.5in} \xy {\ar (16,10)*{};(10,0)*{}**\dir{-}}; {\ar (13.6,4)*{};(16,0)*{}**\dir{-}}; (12.4,6)*{};(10,10)*{}**\dir{-};
\endxy
\end{equation}
\end{minipage}
\begin{minipage}{5cm}
\begin{equation}
\label{orientedcrossings2} \xy {(10,0)*{};(12.4,4)*{}**\dir{-}}; {\ar (13.6,6)*{};(16,10)*{}**\dir{-}}; {\ar
(16,0)*{};(10,10)*{}**\dir{-}};
\endxy
\hspace{0.5in} \xy {\ar (12.4,4)*{};(10,0)*{}**\dir{-}}; {(13.6,6)*{};(16,10)*{}**\dir{-}}; {\ar (10,10)*{};(16,0)*{}**\dir{-}};
\endxy
\end{equation}
\end{minipage}

Given a (generic) tangle diagram $D(T)$ of a tangle $T$ from $n$ points to $m$ points, one can write $D(T)$ as a finite horizontal and vertical composition of
elementary cup, cap, and braid tangle diagrams from above. Taking the corresponding tensor products and compositions of intertwiners, we may assign an
intertwiner $\Phi(D(T))$ from $ V_1^{\otimes n} $ to $ V_1^{\otimes m}.$ In case $T$ is additionally oriented we can associate the intertwiner $q^{3\gamma(D(T))} \Phi(D(T))$ where  $\Phi(D(T)):=\Phi(D(T)')$ and $D(T)'$ is obtained from $D(T)$ by forgetting the orientation.

\begin{theorem}
\label{jonesinvariant} \rm{(\cite[Theorem 5.1]{RT1})} Let $ T $ be an oriented tangle from $ n $ points to $ m $ points.  Let $ D_1 $ and
$ D_2 $ be two of its planar projections, written as a composition of elementary tangle diagrams. Then $ q^{3\gamma(D_1)} \Phi(D_1) =
q^{3\gamma(D_2)}\Phi(D_2) \colon V_1^{\otimes n} \rightarrow V_1^{\otimes m}.$ In particular, $T\mapsto q^{3\gamma(D(T))}\Phi(D(T))$ is
independent of the choice and presentation of $D(T)$, hence defines an invariant of tangles.
\end{theorem}

This invariant is the well-known {\it Reshetikhin-Turaev-Witten invariant} of tangles, \cite{RT1}. In the special case of a tangle $L$ from 0 points
to 0 points, the tangle becomes a link and $q^{3\gamma(L)}\Phi(L) \colon\mathbb{C}(q) \rightarrow \mathbb{C}(q) $ is the  (depending on the definition,
renormalised) Jones polynomial. For a proof of the theorem it suffices to check that the morphisms satisfy the so-called Reidemeister
moves (see e.g. \cite[Lemma 5.3]{RT1}) depicted in \eqref{cr1}-\eqref{cr6} with all labels removed.
Let $ E $ be an elementary, oriented, framed tangle from $r$ (ordered) points to $s$ (ordered) points such that each strand is labeled by a natural
number.  This naturally induces a colouring $(d_1, \ldots, d_r)$ and $(e_1, \ldots, e_s)$ on the $r$ respectively $s$ points. 
\vspace{-2mm}
\begin{multicols}{2}
Given a coloured, oriented tangle diagram $ D, $ one may form its {\it cable} $\op{cab}(D) $ which is obtained from $ D $ by drawing $ l $
parallel copies of an uncoloured strand for each strand coloured by $l$, \emph{e.g.}
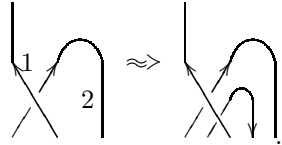
\begin{center}
$\xy {(10,0)*{};(12.4,4)*{}**\dir{-}}; {\ar (13.6,6)*{};(16,10)*{}**\dir{-}}; {\ar (16,0)*{};(10,10)*{}**\dir{-}};
{(22,0)*{};(22,10)*{}**\dir{-}}; {(10,10)*{};(10,18)*{}**\dir{-}}; (16,10)*{};(22,10)*{}**\crv{(17,14)&(21,14)}; {(12,10)*{1}};
{(20,5)*{2}};
\endxy
\hspace{.1in} \xy (10,0)*{}; {\ar@2{~>} (10,10)*{}; (15,10)*{}};
\endxy
\hspace{0.1in} \xy {(10,0)*{};(12.4,4)*{}**\dir{-}}; {\ar (13.6,6)*{};(16,10)*{}**\dir{-}}; {\ar (16,0)*{};(10,10)*{}**\dir{-}};
{(22,0)*{};(22,10)*{}**\dir{-}}; {(10,10)*{};(10,18)*{}**\dir{-}}; {(14.2,2)*{};(13,0)*{}**\dir{-}}; {(14.8,3)*{};(16,5)*{}**\dir{-}};
{(16,10)*{};(22,10)*{}**\crv{(17,14)&(21,14)}}; {\ar (19,5)*{};(19,0)*{}**\dir{-}}; {(16,5)*{};(19,5)*{}**\crv{(16.5,7)&(18.5,7)}};
\endxy$.
\end{center}
\end{multicols}
\vspace{-2mm}


To $E$ one associates the morphism $\Phi_{\op{col}}(E)\colon V_{d_1} \otimes \cdots \otimes V_{d_r}\longrightarrow V_{e_1} \otimes \cdots
\otimes V_{e_s}$ defined as $\Phi_{\op{col}}(E)=(\pi_{e_1} \otimes \cdots \otimes \pi_{e_s}) \circ (\Phi(\op{cab}(E))) \circ (\iota_{d_1} \otimes
\cdots \otimes \iota_{d_r})$. For an arbitrary tangle $ T $ with diagram $D(T) = E_{\alpha_n} \circ \cdots \circ E_{\alpha_1}, $ we define
the intertwiner $\Phi_{\op{col}}(D) = \Phi_{\op{col}}(E_{\alpha_n}) \circ \cdots \circ \Phi_{\op{col}}(E_{\alpha_1})$. Up to some normalisation, this is
well-defined.

\begin{theorem}{\rm (\cite{RT1} Theorem 5.1)}
\label{coloredRTthm}
Let $ T $ be an oriented, coloured, framed tangle from the coloured points $ (d_1, \ldots, d_r) $ to the coloured points $ (e_1, \ldots,
e_s). $  Let $ D_1 $ and $ D_2 $ be two of its planar projections.  Then$$ q^{3\gamma(\op{cab}(D_1))} \Phi_{col}(D_1) =
q^{3\gamma(\op{cab}(D_2))}\Phi_{\op{col}}(D_2) \colon \quad V_{d_1} \otimes \cdots \otimes V_{d_r} \rightarrow V_{e_1} \otimes \cdots \otimes V_{e_s}. $$
In particular, $T\mapsto q^{3\gamma(\op{cab}(D))} \Phi_{col}(D(T))$ is independent of the choice and presentation of $D(T)$ and defines an
invariant of oriented coloured framed tangles.
\end{theorem}


For the proof it suffices to check that the morphisms satisfy the Reidemeister moves given in ~\eqref{cr1}, ~\eqref{cr2}, ~\eqref{cr3},
~\eqref{cr5}, ~\eqref{cr4} and \eqref{cr6}. This can be done easily using Theorem ~\ref{jonesinvariant} and \cite{FK}.
See \cite{RT1} for more details.  In the special case of a framed coloured link, the invariant above becomes the {\it coloured Jones polynomial}.

Using Remark \ref{powerseries}, all the constructions and representations and intertwiners considered so far make sense and can be defined over the base ring $\mC((q))$ instead of $\mC(q)$ which we will do from now on. \\
Hence from now on all the representations of the quantum group as well as the quantum group itself is considered over the base ring $\mC((q))$.

\subsection{An alternate form of Jones-Wenzl projectors}
In this section we give a description of the Jones-Wenzl projector purely in terms of the quantum group.
This formulation was used in \cite{QS2} to investigate the projector on a certain weight space when the quantum parameter is a root of unity.
Define the divided power elements of $\mathcal{U}_q(\mathfrak{sl}_2)$ by
\begin{equation*}
E^{(k)}=\frac{E^k}{[k]!} \hspace{.5in} F^{(k)}=\frac{F^k}{[k]!}.
\end{equation*}
The following identities are well known and can easily proved by induction:

$$\Delta(E^{(k)}) = \sum_{i=0}^k q^{-i(k-i)} E^{(k-i)} \otimes E^{(i)} K^{-k+i},\quad
\Delta(F^{(k)}) = \sum_{i=0}^k q^{i(k-i)} K^i F^{(k-i)} \otimes F^{(i)}.$$

Let $1_{-n+2k} \colon V_1^{\otimes n} \rightarrow V_1^{\otimes n} $
denote the map projecting $V_1^{\otimes n}$ onto the weight space where $K$ acts by the scalar $q^{-n+2k}$.

\begin{theorem}
\label{JWaltformthm}
Restricting to the weight space of $V_1^{\otimes n}$ where $K$ acts as the scalar $q^{2k-n}$, the Jones-Wenzl projector may be expressed as
\begin{equation}
\label{notea}
p_n = \frac{E^{(k)}F^{(k)}}{{n \brack  k}}:=p_{k,n}.
\end{equation}

\end{theorem}

\begin{proof}
We will show that the operator is an intertwiner and satisfies the following properties. By the uniqueness result from Proposition ~\ref{charJW}, \eqref{notea} will then follow.
\vspace{-2mm}
\begin{multicols}{2}
\begin{enumerate}[(i)]
\item $ \frac{E^{(k)}F^{(k)}}{{n \brack  k}} \circ \frac{E^{(k)}F^{(k)}}{{n \brack  k}} = \frac{E^{(k)}F^{(k)}}{{n \brack  k}} $, \item $C_{i,n} \circ \frac{E^{(k)}F^{(k)}}{{n \brack  k}} = 0 $, \item $ \frac{E^{(k)}F^{(k)}}{{n \brack  k}} \circ C_{i,n} = 0$.
\end{enumerate}
\end{multicols}
First we check that $p_{k,n}$ is an intertwiner.
We must show that $E p_{k,n}= p_{k+1,n} E$.
It is easy to see that $E p_{k,n}=\frac{[k+1]E^{(k+1)} F^{(k)}}{{n \brack k}}$.
We calculate
$$p_{k+1,n}E =  \frac{1}{{n \brack k+1}} E^{(k+1)}F^{(k+1)}E 
= \frac{[n-k] }{{n \brack k+1}} E^{(k+1)} F^{(k)} 
= \frac{[k+1]E^{(k+1)} F^{(k)}}{{n \brack k}} 
= E p_{k,n},$$
where we used \cite[(2.11)]{L} for the second equality.  
Note that while there are two terms in formula \cite[(2.11)]{L}, one of them vanishes due to weight space considerations.
Similarly one can show $F p_{k,n} = p_{k-1,n} F$.
It is obvious that $p_{k,n}$ commutes with $K^{\pm 1}$. To prove that the operator $\frac{E^{(k)}F^{(k)}}{{n \brack  k}} $ is idempotent, we calculate
\begin{equation*}
E^{(k)}F^{(k)} E^{(k)}F^{(k)} = {n \brack k} E^{(k)} F^{(k)}.
\end{equation*}
using \cite[(2.11)]{L}. Note that all but one term in that formula vanish due to weight space considerations.
It now follows immediately that $\frac{E^{(k)}F^{(k)}}{{n \brack  k}} $ is an idempotent.

By construction the operator $F^{(k)}$ maps the weight space which consists of eigenvectors for the eigenvalue $q^{2k-n}$ to the lowest weight space.  On the lowest weight space $C_{i,n}$ acts as zero.   Thus
\begin{equation*}
C_{i,n} \circ \frac{E^{(k)}F^{(k)}}{{n \brack  k}} = 0 = \frac{E^{(k)}F^{(k)}}{{n \brack  k}} \circ C_{i,n}.\qedhere
\end{equation*}
\end{proof}

Thus we may write the Jones-Wenzl projector as
\begin{equation*}
p_n = \sum_{k=0}^n \frac{E^{(k)} F^{(k)}}{{n \brack k}} 1_{-n+2k}.
\end{equation*}

We now give an independent proof that the Jones-Wenzl projectors defined in terms of the quantum group slide along cups.  While this fact is already guaranteed by Theorem ~\ref{coloredRTthm} and Theorem ~\ref{JWaltformthm}, this proof of projector sliding will be useful later when trying to prove the categorified statement.

Let $C_n \colon \mathbb{C} \rightarrow V_1^{\otimes n} \otimes V_1^{\otimes n}$ be the intertwiner associated to $n$ nested cups.  

\begin{lemma}
\label{Fslidelemma}
For $ k \leq n$ there is an equality
\begin{equation*}
(F^{(k)} \otimes 1) \circ C_n = (-1)^k q^{k(k-1)} (K^k \otimes F^{(k)}) \circ C_n.
\end{equation*}
\end{lemma}

\begin{proof}
This is proved by induction on $k$.
For the base case $k=1$ first recall that
$\Delta(F)=K \otimes F +  F \otimes 1$.
Then $
(F \otimes 1) \circ C_n = F \circ C_n - (K \otimes F) \circ C_n 
= -(K \otimes F) \circ C_n$,
where the second equality follows from the fact that $C_n$ maps $\mathbb{C}$ into the invariant subspace of $V_1^{\otimes (2n)}$, which is annihilated by $F$.

Now let us assume that the statement is proved for $j < k$.  Then we have
\begin{align*}
(F^{(k)} \otimes 1) \circ C_n &= \frac{(F \otimes 1)}{[k]}(F^{(k-1)} \otimes 1) \circ C_n \\
&= \frac{(F \otimes 1)}{[k]} (-1)^{k-1} q^{(k-1)(k-2)} (K^{k-1} \otimes F^{(k-1)}) \circ C_n \\
&= \frac{(-1)^{k-1}}{[k]} q^{(k-1)(k-2)} q^{2k-2} (K^{k-1} \otimes F^{(k-1)})(F \otimes 1) \circ C_n \\
&= (-1)^k q^{k(k-1)} (K^k \otimes F^{(k)}) \circ C_n
\end{align*}
where the second equality holds by induction, the third equality follows from the fact that $F$ and $K^{k-1}$ commute up to a power of $q$, and the fourth equality follows from the base case of the induction.
\end{proof}

\begin{lemma}
\label{Eslidelemma}
For $k \leq n$ there is an equality
\begin{equation*}
(E^{(k)} \otimes 1) \circ C_n = (-1)^k q^{-k(k-1)} (1 \otimes K^k)(1 \otimes E^{(k)}) \circ C_n.
\end{equation*}
\end{lemma}

\begin{proof}
We will prove by induction on $k$
\begin{equation}
\label{Eslideaux}
(E^{(k)} \otimes K^{-k}) \circ C_n = (-1)^k q^{-k(k-1)}(1 \otimes E^{(k)}) \circ C_n.
\end{equation}
The lemma then follows by multiplying ~\eqref{Eslideaux} on the left by 
$ (1 \otimes K^k)$. In order to prove the base case of ~\eqref{Eslideaux} recall that
$ \Delta(E) = 1 \otimes E + E \otimes K^{-1}$.
Then
$ (E \otimes K^{-1}) \circ C_n  = E \circ C_n -  (E \otimes K^{-1}) = -(1 \otimes E) \circ C_n$,
where the second equality follows from the fact that $C_n$ maps $\mathbb{C}$ into the invariant subspace of $V_1^{\otimes (2n)}$ which is annihilated by $E$.

Now let us assume that the statement is proved for $j < k$.  Then we have
\begin{align*}
(E^{(k)} \otimes K^{-k}) \otimes C_n &= \frac{(E \otimes K^{-1})}{[k]}(E^{(k-1)} \otimes K^{1-k}) \circ C_n \\
&= \frac{(E \otimes K^{-1})}{[k]}(-1)^{k-1} q^{-(k-1)(k-2)} (1 \otimes E^{(k-1)}) \circ C_n \\
&= \frac{(-1)^{k-1} q^{-(k-1)(k-2)}q^{-2k+2}}{[k]}(1 \otimes E^{(k-1)})(E \otimes K^{-1}) \circ C_n \\
&= (-1)^k q^{-k(k-1)}(1 \otimes E^{(k)}) \circ C_n
\end{align*}
where the second equality is the induction hypothesis, the third equality follows from the fact that $K^{-1}$ and $E^{(k-1)}$ commute up to a power of $q$, and the fourth equality follows from the base case.
\end{proof}

\begin{prop}
\label{decatJWcupslideaux} 
We have the equality
\begin{equation*}
(E^{(n)} F^{(n)} \otimes 1) \circ C_n =  (1 \otimes F^{(n)} E^{(n)})(K^n \otimes K^n) \circ C_n.
\end{equation*}
\end{prop}

\begin{proof}
We calculate that $(E^{(n)} F^{(n)} \otimes 1) \circ C_n$ equals
\begin{align*}
(E^{(n)} \otimes 1)(F^{(n)} \otimes 1) \circ C_n 
&= (-1)^n  q^{n(n-1)} (E^{(n)} \otimes 1)(K^n \otimes F^{(n)}) \otimes C_n \\
&= (-1)^n q^{n(n-1)} q^{(-2n^2)} (K^n \otimes F^{(n)})(E^{(n)} \otimes 1) \circ C_n \\
&= q^{(-2n^2)} (K^n \otimes F^{(n)})(1 \otimes K^n)(1 \otimes E^{(n)}) \circ C_n  \\
&= (1 \otimes F^{(n)} E^{(n)})(K^n \otimes K^n) \circ C_n 
\end{align*}
using Lemmas \ref{Fslidelemma} and \ref{Eslidelemma} in the second last equalities respectively.
\end{proof}

The next proposition says that Jones-Wenzl projectors slide along cup diagrams.  This is crucial in proving that the Reshetikhin-Turaev invariant is indeed an invariant of coloured framed tangles.  The important ingredients in proving this proposition are the formulas for comultiplication of divided powers along with the fact that nested cups map the trivial representation into an invariant subspace of $V_1^{\otimes (2n)}$.

\begin{prop}
There is an equality of intertwiners
\begin{equation*}
(\sum_{k=0}^n \frac{E^{(k)} F^{(k)}}{{n \brack k}} 1_{-n+2k} \otimes 1) \circ C_n
=
(1 \otimes \sum_{k=0}^n \frac{E^{(k)} F^{(k)}}{{n \brack k}} 1_{-n+2k}) \circ C_n.
\end{equation*}
\end{prop}

\begin{proof}
The intertwiner $C_n$ maps $\mathbb{C}$ into the zero weight space of $V_1^{\otimes (2n)}$, thus
\begin{equation}
\label{projectionslide}
(1_{-n+2k} \otimes 1) \circ C_n = (1 \otimes 1_{n-2k}) \circ C_n.
\end{equation}
Then by 
Proposition ~\ref{decatJWcupslideaux} and \cite[(2.11)]{L}, we have
\begin{align*}
(\sum_{k=0}^n \frac{E^{(k)} F^{(k)}}{{n \brack k}} 1_{-n+2k} \otimes 1) \circ C_n &=
(1 \otimes \sum_{k=0}^n \frac{F^{(k)} E^{(k)}}{{n \brack k}} 1_{-n+2k}) \circ C_n \\
&= (1 \otimes \sum_{k=0}^n \frac{E^{(k)} F^{(k)}}{{n \brack k}} 1_{-n+2k}) \circ C_n
\end{align*}
noting again that all terms from \cite[(2.11)]{L}  vanish except one.
\end{proof}

\section{Categorification of tensor products}
\label{catproducts}
For an abelian (or triangulated) category $\cA$ we denote by $[\cA]$ the {\it Grothendieck group} of $\cA$ which is by definition the free abelian group generated by the isomorphism classes $[M]$ of objects $M$ in $\cA$  modulo the relation $[C]=[A]+[B]$ whenever there is a
short exact sequence (or distinguished triangle) of the form $A\rightarrow C\rightarrow B$. When $ \cA $ is a triangulated category,
denote $n$ compositions of the shift functor by $ \lsem n\rsem$. In the following $\cA$ will always be a (derived) category of $\mZ$-graded modules over some finite-dimensional algebra $A$. Hence the category has an internal $\mZ$-grading (and a homological grading) and so $[\cA]$ has a natural $\mZ[q,q^{-1}]$-module structure where $q$ acts by shifting the internal grading up by $1$. We denote by $\langle i\rangle$ the functor which shifts the internal grading up by $i$, so that if $ M $ is concentrated in internal degree zero, then $ M \langle i \rangle $ is concentrated in internal degree $ i$.
Similarly, if $ M $ is concentrated in homological degree zero, then $ M \lsem i \rsem $ is concentrated in homological degree $ i$.
In the following we will mostly work with $\mC((q))\otimes_{\mZ[q,q^{-1}]}[\cA]$ and call it the {\it Grothendieck space}.

The categorification of both, the $3j$-symbol \cite{FSS1} and the coloured Jones polynomial is based on a categorification of the representation ${V}_1^{\otimes n}$ and the Jones-Wenzl projector. By this we roughly mean that we want to upgrade each weight space into a $\mZ$-graded abelian category with the action of $E$, $F$ and $K$, $K^{-1}$ via exact functors (see below for more precise statements). Such categorifications were first constructed in \cite{BFK, FKS} via graded versions of the category $ \mathcal{O}$ for $ \mathfrak{gl}_n $ and various functors acting on this category. We refer to \cite[Section 6]{FSS1} for the example for $n=2$.

\subsection{Categorification of $ V_1^{\otimes n}$}
We start by recalling the Lie theoretic categorification of $\overline{V}_1^{\otimes n}$.
Let $n$ be a non-negative integer. Let $\mathfrak{g}=\mathfrak{gl}_n$ be the Lie algebra of complex $n\times n$-matrices. We fix the standard Cartan subalgebra $ \mathfrak{h} \subset \mathfrak{g} $ with its basis given by diagonal matrices: $\lbrace E_{1,1}, \ldots, E_{n,n} \rbrace $ for $ i=1,\ldots, n. $  The dual space $ \mathfrak{h}^* $ comes with the dual basis $ \lbrace e_i \mid i=1,
\ldots, n \rbrace $ with $e_i(E_{j,j})=\delta_{i,j}. $ The nilpotent subalgebra of strictly upper diagonal matrices spanned by $ \lbrace E_{i,j} \mid i<j \rbrace$ is denoted $\mathfrak{n}^{+}$. Similarly, let $ \mathfrak{n}^{-} $ be the subalgebra consisting of lower triangular matrices.  We fix the standard Borel subalgebra $\mathfrak{b}=\mathfrak{h} \bigoplus \mathfrak{n}^{+}$.  For any Lie algebra $L$ we denote by $\mathcal{U}(L)$ its universal enveloping algebra, so $L$-modules are the same as (ordinary) modules over the ring $\mathcal{U}(L)$.

Let $W=\mathbb{S}_n$ denote the Weyl group of $ \mathfrak{gl}_n $ generated by simple reflections (=simple transpositions) $ \lbrace s_i, 1 \leq i \leq n-1 \rbrace$. For $w\in W $ and $ \lambda \in \mathfrak{h}^*$, let $w\cdot\lambda = w(\lambda+\rho_n)-\rho_n$, where $ \rho_n= \frac{n-1}{2}e_1 + \cdots + \frac{1-n}{2}e_n$. In the following we will always consider this action. For $\la\in\mh^*$ we denote by  $W_{\lambda} $ the stabilizer of $\lambda \in \mathfrak{h}^*$.

\begin{define}
\label{defO}
{\rm
We denote by $ \mathcal{O} = \mathcal{O}(\mathfrak{gl}_n) $ the full subcategory of finitely-generated $\mathcal{U}(\mathfrak{g})$-modules $M$ which decompose into finite-dimensional weight spaces $M_\la$ for $\mathfrak{h}$ and are $\mathcal{U}(\mathfrak{b})$-finite (i.e. each vector lies in a finite-dimensional subspace stable under $\mb$). For $n=0$ we denote by $\cO$ the category of finite-dimensional complex vector spaces.
}
\end{define}
This category was introduced in \cite{BGG}. 
In addition to all finite-dimensional modules, the category also contains
the {\it Verma modules} $M(\la)=\mathcal{U}\otimes_{\mathcal{U}(\mb)}\mC_\la$ for any $1$-dimensional $\mh$-module $\mC_\la$.
(Here $\la\in\mh^*$ and  $\mC_\la$  is the module where $h\in\mh$ acts by multiplication with $\la(h)$ and $\mathfrak{n}^+$ acts trivially). We call $\la$ the {\it highest weight} of $M(\la)$. The category is closed under taking submodules and quotients, and under tensoring with finite-dimensional $\mg$-modules. In fact it is the smallest abelian category containing all Verma modules and closed under tensoring with finite-dimensional modules. For details and standard facts we refer to \cite{Hu}.

Every Verma module $M(\la)$ has a unique simple quotient which we denote by $L(\la)$. The latter form precisely the isomorphism classes of simple objects in $\cO$. Moreover, the category has enough projectives. Let $P(\la)$ be the projective cover of $L(\la)$ in $\cO$.
The category $\cO$ decomposes into indecomposable summands $\cO_\la$, called {\it blocks}, under the action of the centre of $\cU(\mg)$. $L(\la)$ and $L(\mu)$ are then in the same block if their highest weights are in the same $W$-orbit (as defined above). Hence the blocks are labeled by the maximal representatives of the $W$-orbits, maximal with respect to the Bruhat ordering on $\mh^*$. We call these representatives $\la$ {\it dominant weights} since they are in the dominant chamber with respect to our $W$-action. Then the $L(w\cdot\la)$, $w\in W/W_\la$ are precisely the simple objects in $\cO_\la$.
Weight spaces of $\overline{V}_1^{\otimes n}$ will be categorified using the blocks $\mathcal{O}_{\la} (\mathfrak{gl}_n)$, which by abuse of notation we will call $\mathcal{O}_k(\mathfrak{gl}_n)$, corresponding to the integral dominant weights $e_1+\cdots+e_k-\rho_n$ for $0\leq k\leq n$. When $ k=0$, the corresponding $ \la $ is $ -\rho$. To make calculations easier denote also by $ M(a_1, \ldots, a_n) $ the Verma module with highest weight $ a_1 e_1 + \cdots + a_n e_n - \rho_n$ with simple quotient $L(a_1, \ldots, a_n) $ and projective cover $ P(a_1, \ldots, a_n) $ in $ \mathcal{O}(\mathfrak{gl}_n).$  They are all in the same block and belong to $ \mathcal{O}_k(\mathfrak{gl}_n) $ if and only if $k$ of the $a_j$'s are $1$ and $n-k$ of them are $0$.

Each block of $\cO$ is equivalent to a category of modules over a finite-dimensional algebra, although this algebra is very difficult to describe in general, see \cite{Strquiv} for small examples. Therefore, our arguments will be Lie theoretic in general, but keeping in mind that in principle everything could be formulated in terms of modules over certain graded finite-dimensional algebras.


In the following we will focus on the blocks $\cO_k$ and formulate the statements for them only, although most of them are true in general. Each block can be enriched (in a non-trivial way) with the structure of a $\mZ$-grading as follows.  Let $A_{k,1^n} $ denote the endomorphism algebra of a minimal projective generator $P_k$ of $ \mathcal{O}_k$ (for explicit examples see again \cite{Strquiv}). We identify $\cO_k$ with the category of right $A_{k,1^n}$-modules via the functor $\HOM_\mg(P_k,_-)$. Then it is a non-trivial fact that the algebra $A_{k,1^n}$ can be equipped with a non-negative grading, actually a Koszul grading, see \cite{BGS}.

Hence we may consider a graded version of the category $\cO_k$, that is the category $\gmod-A_{k,1^n}$ of finite-dimensional, graded right modules over $A_{k,1^n}$ and try to lift all the representation theory to the graded setting. This has been worked out in \cite{Strquiv}, so we just recall the results. Each object $M$ specifically mentioned earlier in category $\mathcal{O}$ has a graded lift $\hat{M}$ in the category $\gmod-A_{k,1^n}$. These lifts are in fact unique up to isomorphism and shift in the grading. We pick such lifts $\hat{P}(a_1, \ldots, a_n), \hat{M}(a_1, \ldots, a_n) $ and $ \hat{L}(a_1, \ldots, a_n) $ by requiring that their heads are concentrated in degree zero. (The choice is then unique up to isomorphism).

Category $\cO$ has a duality $ \op{d} $ which just amounts to an identification of  $A_{k,1^n}$ with its opposite algebra by taking the usual vector space dual. In particular we can also work with the category $A_{k,1^n}-\gmod$ of graded left modules. Let $\hat{\op{d}}=\Hom_{A_{k,1^n}}(_-,\mC)$ denote the graded lift of this duality normalized such that it preserves the $ \hat{L}(a_1, \ldots, a_n)$'s. Let $\hat{\nabla}(a_1, \ldots, a_n)=\hat{\op{d}}(\hat{M}(a_1, \ldots, a_n))$ and  then let ${\nabla}(a_1, \ldots, a_n)$ be the module after forgetting the grading. Similarly $\hat{I}(a_1, \ldots, a_n)=\hat{\op{d}}(\hat{P}(a_1, \ldots, a_n))$ is the injective hull of $\hat{L}(a_1, \ldots, a_n)$
and ${I}(a_1, \ldots, a_n)$ denotes the module when forgetting the grading. One can work out these modules for the explicit small examples in \cite{Strquiv} in terms of representations of quivers.

Let $ \langle r \rangle \:\colon\: \gmod-A_{k,1^n} \rightarrow \gmod-A_{k,1^n} $ be the functor of shifting the grading up by $r$ (such that if $M=M_i$ is concentrated in degree $i$ then $M\langle r\rangle=M\langle k\rangle_{i+r}$ is concentrated in degree $i+r$). The
additional grading turns the Grothendieck group into a $\mZ[q,q^{-1}]$-module, the shift functor $\langle r\rangle$ induces the multiplication by $q^r$. 

\begin{prop}{\rm \cite[Theorem 4.1, Theorem 5.3]{FKS}}
\label{Grothgraded} There is an isomorphism 
\begin{eqnarray}
\label{iso26}
\Phi_n:\; \mathbb{C}((q))\otimes_{\mathbb{Z}[q,q^{-1}]} [\bigoplus_{k=0}^{n} \gmod-A_{k,1^n}] &\cong&V_1^{\otimes n}.\\
\left[\hat{M}(a_1, \ldots, a_n)\right]&\mapsto&v_{a_1}\otimes v_{a_2}\otimes \cdots\otimes v_{a_n}\nonumber
\end{eqnarray}
of $\mathbb{C}((q))$-modules. The images 
of the isomorphism classes $[\hat{L}(a_1, \ldots, a_n)]$ of simple modules
are Lusztig's {\it dual canonical basis} elements $v^{a_1}\heartsuit v^{a_2}\heartsuit \cdots \heartsuit v^{a_n}$.
\end{prop}

The following theorem categorifies the $\cU_q$-action. The generator $E$ acts just by a graded lift $\hat{\mathcal{E}}$ of tensoring with
the natural representation and then projecting onto the corresponding block (and $F$ with the adjoint), whereas $K$ acts by an appropriate grading shift on each block. We state the theorem abstractly below and refer to \cite{FKS} for details.

\begin{theorem}[{\rm \cite[Theorem 4.1]{FKS}}]
\label{projfunc} There are exact functors of graded categories
\begin{eqnarray*}
&\hat{\mathcal{E}}_k \colon\quad \gmod-A_{k,1^n} \rightarrow \gmod-A_{k+1,1^n},\quad
\hat{\mathcal{F}}_k \colon\quad \gmod-A_{k,1^n} \rightarrow \gmod-A_{k-1,1^n},&\\
&\hat{\mathcal{K}}_k^{\pm 1} \colon\quad \gmod-A_{k,1^n}
\rightarrow \gmod-A_{k,1^n}&
\end{eqnarray*}
such that
\begin{align*}
&{\hat{\mathcal{K}}_{i+1}} \hat{\mathcal{E}}_i \cong {\hat{\mathcal{E}}_i} {\hat{\mathcal{K}}}_i \langle 2
    \rangle, \quad
    {\hat{\mathcal{K}}}_{i-1}\hat{\mathcal{F}}_i \cong {\hat{\mathcal{F}}_i} {\hat{\mathcal{K}}}_i \langle -2 \rangle, \quad
    {\hat{\mathcal{K}}}_i {\hat{\mathcal{K}}}_i^{-1} \cong {\Id}
    \cong {\hat{\mathcal{K}}}_i^{-1} {\hat{\mathcal{K}}}_i^1, \\
    &\displaystyle \hat{\mathcal{E}}_{i-1}
    \hat{\mathcal{F}}_i \bigoplus
    \bigoplus_{j=0}^{n-i-1} {\Id} \langle n-2i-1-2j \rangle \cong \hat{\mathcal{F}}_{i+1} \hat{\mathcal{E}}_i
    \bigoplus
    \bigoplus_{j=0}^{i-1}
    {\Id} \langle 2i-n-1-2j \rangle,
\end{align*}
and the isomorphism \eqref{iso26} becomes an isomorphism of $\mathcal{U}_q(\mathfrak{sl}_2)$-modules.
\end{theorem}

\subsection{Categorification of the Jones-Wenzl projector}
\label{sec:JW}
In order to categorify more general tensor products, we introduce the category $ {}_{k} \mathcal{H}_{{\bf d}}^{1}(\mathfrak{gl}_n)$ of certain Harish-Chandra bimodules. For details on these specific categories we refer to \cite{MS2}. For basics on the category of (generalised) Harish-Chandra bimodules see \cite{BG}, \cite[Kapitel 6]{Ja}, and for its description in terms of Soergel bimodules see \cite{Soergel-HC}, \cite{StGolod}.

\begin{define}{\rm
Let $\mg=\mathfrak{gl}_n$ and define for $\mu$, $\la$ dominant integral $ {}_{\lambda} \mathcal{H}_{\mu}^1(\mg) $ to be the full subcategory of $\mathcal{U}(\mg)$-bimodules of finite length with objects $M$ satisfying the following conditions
\begin{enumerate}[(i)]
\item $M$ is finitely-generated as a $ \mathcal{U}(\mg)$-bimodule,
\item every element $m\in  M$ is contained in a finite-dimensional vector space stable under the adjoint action $x.m=xm-mx$ of $\mg$ (where $x\in\mg$, $m\in M$),
\item for any $m\in M$ we have $m I_\mu=0$ and there is some $n\in\mZ_{>0}$ such that $(I_\lambda)^nm=0$, where $I_\mu$, respectively $I_\la$ is the maximal ideal of the centre of $\cU(\mg)$ corresponding to $\mu$ and $\la$ under the Harish-Chandra isomorphism. (One usually says {\it $M$ has generalised central character $\chi_\la$ from the left and central character $\chi_\mu$ from the right}).
\end{enumerate}
}
\end{define}
We call the objects in these categories {\it Harish-Chandra bimodules}. Here is the construction of the simple objects of these categories. Given two $\mg$-modules $M$ and $N$ we can form the space $\HOM_\mC(M,N)$ which is naturally a $\mg$-bimodule, but very large. We denote by $\cL(M,N)$ the ad-finite part, that is the subspace of all vectors lying in a finite-dimensional vector space invariant under the adjoint action $X.f:=Xf-fX$ for $X\in \mg$ and $f\in\HOM_\mC(M,N)$. This is a Harish-Chandra bimodule and the simple objects in ${}_{\la} \mathcal{H}_{\mu}^1(\mg)$ are precisely the bimodules of the form $\cL(M(\mu), L(w.\la))$, where $w$ is a longest element representative in the double coset space $\mathbb{S}_{\mu}\backslash \mathbb{S}_n / \mathbb{S}_{\lambda}$ where $ \mathbb{S}_{\mu}, \mathbb{S}_{\lambda}$ are the stabilizers of $\mu$ respectively $\la$ (see \cite[6.23]{Ja} for details). The Bernstein-Gelfand-Joseph-Zelevinsky inclusion functor realizes the category of Harish-Chandra bimodules as a subcategory of $\mathcal{O}$, see \cite{BG}, \cite{Ja}.

\begin{define}{\rm
Let $\la, \mu$ be integral dominant weights. Let $ \mathcal{O}_{\lambda, \mu}(\mathfrak{gl}_n) $ be the full subcategory of $ \mathcal{O}_{\lambda}(\mathfrak{gl}_n) $ consisting of
modules $ M $ with projective presentations $ P_2 \rightarrow P_1 \rightarrow M \rightarrow 0 $ where $ P_1 $ and $ P_2 $ are direct sums
of projective objects of the form $ P(x.\lambda) $ where $ x $ is a longest element representative in the double coset space $ \mathbb{S}_{\mu}
\backslash \mathbb{S}_n / \mathbb{S}_{\lambda}. $
Then there are functors
\begin{eqnarray}
{}_{\lambda} \bar{\pi}_{\mu} \colon \mathcal{O}_{\lambda}(\mathfrak{gl}_n) \rightarrow {}_{\lambda}
    \mathcal{H}_{\mu}^1(\mathfrak{gl}_n),&& {}_{\lambda} \bar{\pi}_{\mu}(X) =\cL(M(\mu),X),\label{Lietheoretic1}\\
{}_{\lambda} \bar{\iota}_{\mu} \colon {}_{\lambda}\mathcal{H}_{\mu}^1(\mathfrak{gl}_n) \rightarrow \mathcal{O}_{\lambda}(\mathfrak{gl}_n), &&{}_{\lambda} \bar{\iota}_{\mu}(M) = M \otimes_{\mathcal{U}(\mathfrak{gl}_n)} M(\mu).\label{Lietheoretic2}
\end{eqnarray}
}
\end{define}

\begin{theorem}{\rm (\cite{BG})}
\label{BG} The functors $ {}_{\lambda} \bar{\iota}_{\mu} $ and $ {}_{\lambda} \bar{\pi}_{\mu} $ provide inverse equivalences of categories
between $ \mathcal{O}_{\lambda, \mu}(\mathfrak{gl}_n) $ and $ {}_{\lambda} \mathcal{H}_{\mu}^1(\mathfrak{gl}_n). $
\end{theorem}

If the stabilizer of $ \lambda $ is equal $ \mathbb{S}_k \times \mathbb{S}_{n-k} $ and the stabilizer of $ \mu $ is equal to $ \mathbb{S}_{\bf d}=\mathbb{S}_{d_1} \times \cdots \times \mathbb{S}_{d_r} $, then we denote the categories in Theorem ~\ref{BG} by
$ \mathcal{O}_{k, {\bf d}}(\mathfrak{gl}_n) $ and $ {}_{k} \mathcal{H}_{\bf d}^1(\mathfrak{gl}_n) $ respectively. The following might be viewed, together with Proposition \ref{Grothgraded}, as a categorical version of \cite[Theorem 1.11]{FK}.
\begin{prop}
\label{simples}
\begin{enumerate}[(i)]
\item$ {}_{\lambda} \bar{\pi}_{\mu} $ maps simple objects to simple objects or zero. All simple Harish-Chandra bimodules are obtained in this way.
\item The simple objects in ${}_{\lambda} \mathcal{H}_{\mu}^1(\mathfrak{gl}_n)={}_{k} \mathcal{H}_{\bf d}^1(\mathfrak{gl}_n)$  are precisely the $\cL(M(\mu), L(x.\lambda))$ where $ x $ is a longest element representative in the double coset space $ \mathbb{S}_{\mu}
\backslash \mathbb{S}_n / \mathbb{S}_{\lambda}$.
Denote the simple modules in ${}_{k} \mathcal{H}_{\bf {d}}^1(\mathfrak{gl}_n)$ by:
$$L\left(k_1,d_1|k_2,d_2|\cdots |k_r,d_r\right):=\cL(M(\mu), L(\underbrace{0, \ldots, 0,}_{d_1-k_1}
\underbrace{1, \ldots, 1}_{k_1},\cdots,
\underbrace{0, \ldots, 0}_{d_r-k_r}, \underbrace{1, \ldots, 1}_{k_r})).$$
\end{enumerate}
\end{prop}

\begin{proof}
This is clear from the definition of the functors and the classification of simple Harish-Chandra bimodules, see \cite{Ja}.
\end{proof}

\begin{definition}
\label{propstandard}
{\rm The {\it proper standard module} labeled by $(k_1,d_1|k_2,d_2|\cdots |k_r,d_r)$ is defined to be
\begin{eqnarray*}
&\blacktriangle\left(k_1,d_1|k_2,d_2|\cdots |k_r,d_r\right)&\\
&:=\cL(M(\mu), M(\underbrace{0, \ldots, 0,}_{d_1-k_1}
\underbrace{1, \ldots, 1,}_{k_1}\cdots
\underbrace{0, \ldots, 0,}_{d_r-k_r} \underbrace{1, \ldots, 1}_{k_r}))\in {}_{k} \mathcal{H}_{\bf {d}}^1(\mathfrak{gl}_n).&
\end{eqnarray*}
The name comes from the fact that this family of modules form the proper standard objects in a fully stratified structure (see Lemma \ref{propstrat} and \cite{BSsemi, MS2} for details.)
}
\end{definition}

\begin{definition}
\label{standardsinduced}
The standard objects are defined as parabolically induced `big projectives' in category $\cO$ for the Lie algebra $\mathfrak{gl}_{\bf d}:=\mathfrak{gl}_{d_1}\oplus\mathfrak{gl}_{d_2}\oplus\cdots\oplus\mathfrak{gl}_{d_r}$. In formulas:
\begin{eqnarray*}
{\Delta}\left(k_1,d_1|k_2,d_2|\cdots |k_r,d_r\right):=\cU(\mathfrak{gl_n})\otimes_{\cU(\mathfrak{p})}\left(P(0^{d_1-k_1}1^{k_1})\boxtimes \cdots\boxtimes P(0^{d_r-k_r}1^{k_r})\right),
\end{eqnarray*}
where the $\mathfrak{gl}_{\bf d}$-action is extended by zero to $\mathfrak{p}=\mathfrak{gl}_{\bf d}+\mathfrak{n}^+$,
see \cite[Proposition 2.9]{MS2}.
\end{definition}

Let $ A_{k,{\bf d}} $ be the endomorphism algebra of a minimal projective generator of $ \mathcal{O}_{k,{\bf d}}(\mathfrak{gl}_n) $.  This algebra naturally inherits a grading from $ A_{k,1^n} $ and so via the Bernstein-Gelfand
equivalence, we may consider the graded category of Harish-Chandra right modules $ \gmod-A_{k,{\bf d}} $.
Let
\begin{eqnarray*}
\hat{L}\left(k_1,d_1|k_2,d_2|\cdots |k_r,d_r\right),& \hat{\blacktriangle}\left(k_1,d_1|k_2,d_2|\cdots |k_r,d_r\right), &\hat{\Delta}\left(k_1,d_1|k_2,d_2|\cdots |k_r,d_r\right)
\end{eqnarray*}
be the standard graded lifts with head concentrated in degree zero for the first two and $ -\prod_{i=1}^rk_i(d_i-k_i) $ for the last one of the corresponding Harish-Chandra bimodules.

\begin{lemma}[{\cite[Lemma 44]{FSS1}}].
There are graded lifts of the BG-functors
\begin{eqnarray}
{}_k\hat\pi_{\bf d} \colon\quad \gmod-A_{k, 1^n} \rightarrow \gmod-A_{k, {\bf d}} \label{algebraic1}\\
{}_k\hat\iota_{\bf d} \colon\quad \gmod-A_{k, {\bf d}} \rightarrow \gmod-A_{k, 1^n},\label{algebraic2}
\end{eqnarray}
which naturally commute with the $\mathcal{U}_q(\mathfrak{sl}_2)$-action from Theorem \ref{projfunc}.
\end{lemma}

The $\mathcal{U}_q$-action from Theorem ~\ref{projfunc} by exact functors restricts to the subcategories $ \mathcal{O}_{k, {\bf d}}(\mathfrak{gl}_n) $, hence defines also an action on $ {}_{k} \mathcal{H}_{{\bf d}}^1(\mathfrak{gl}_n)$. It induces the following isomorphism.

\begin{theorem}[Arbitary tensor products and its integral structure]\hfill\\
\label{cattensor}
There is an isomorphism of  $\mathcal{U}_q(\mathfrak{sl}_2)$-modules
\begin{eqnarray}
\Phi_{\bf d}:\; \displaystyle \mC((q)) \otimes_{\mathbb{Z}[q,q^{-1}]} \left[\bigoplus_{k=0}^{n} \gmod-A_{k, {\bf
d}} \right] &\cong& V_{d_1} \otimes \cdots \otimes V_{d_r} \nonumber\\
\left[ \hat{L} \left(k_1,d_1| k_2,d_2 |\cdots | k_r,d_r \right) \right]
&\longmapsto&v^{k_1}\heartsuit v^{k_2}\heartsuit \cdots \heartsuit v^{k_r},\label{basissimples}
\end{eqnarray}
This isomorphism sends proper standard modules to the dual standard basis:
\begin{eqnarray}
\label{basisstandards}
[\hat{\blacktriangle}\left(k_1,d_1|k_2,d_2|\cdots |k_r,d_r\right)]
&\longmapsto&v^{k_1}\otimes v^{k_2}\otimes \cdots \otimes v^{k_r}.
\end{eqnarray}
\end{theorem}

\begin{proof}
See \cite[Theorem 45]{FSS1}.
\end{proof}

While the functor $ {}_{k}{\hat{\pi}}_{\bf d} $ is exact, the functor $ {}_{k}{\hat{\iota}}_{\bf d} $ is only right exact. Let
\begin{eqnarray*}
\mathbb{L}({}_{k}{\hat{\iota}}_{\bf d}):\;D^{<}(\gmod-A_{k, {\bf d}}) \rightarrow D^<(\gmod-A_{k, 1^n})
\end{eqnarray*}
be the left derived functor of ${}_{k}{\hat{\iota}}_{\bf d}$, where we use the symbol $ D^{<}(?)$ to denote the full subcategory of
the derived category $D(?)$ consisting of complexes bounded to the right. In particular, for $?=\mathcal{O}_{k, {\bf d}}(\mathfrak{gl}_n)$, $\gmod-A_{k,1^n}$, ${}_{k}\mathcal{H}_{{\bf d}}^1(\mathfrak{gl}_n)$, or $\gmod-A_{k,{\bf d}}$, any complex in $ D^{<}(?)$ with cohomology only in finitely many places is quasi-isomorphic to a complex of projectives from ?. Note also that $\mathbb{L}({}_{k}{\hat{\pi}}_{\bf d})={}_{k}{\hat{\pi}}_{\bf d}$ and the functors from the $\mathcal{U}_q(\mathfrak{sl}_2)$-action extend uniquely to the corresponding  $D^{<}(?)$. In case $?$ is graded,  let  ${D}^{\triangledown}(?)$ denote the full subcategory of $D^{<}(?)$ consisting of all complexes $K^\bullet$ in  $D^{<}(?)$ such that the graded Euler characteristic $\sum_{i\in \mathbb{Z}}(-1)^{i}[K^i]$ is a well-defined element in  $\mC((q)) \otimes_{\mathbb{Z}[q,q^{-1}]} \left[?\right]$. We call such complexes {\it Euler finite}. In such a case, the Euler characteristic gives a well-defined element in the completed Grothendieck group $\mC((q)) \otimes_{\mathbb{Z}[q,q^{-1}]} \left[?\right]$ of the abelian category which we identify with the Grothendieck group of the triangulated category  ${D}^{\triangledown}(?)$ as in \cite{AchS}.

\begin{theorem}[Categorification of the Jones-Wenzl projector]\hfill
\label{catJW}
\begin{enumerate}
\item The composition ${\hat{p}}_{k,\bf d}:=(\mathbb{L}({}_k \hat{\iota}_{\bf d})){}_k\pi_{\bf d}$ is an idempotent. More precisely ${}_k\pi_{\bf d}\mathbb{L}({}_k \hat{\iota}_{\bf d})$ is isomorphic to the identity functor. 
\item The functors $\mathbb{L}({}_k \hat{\iota}_{\bf d})$ and ${}_k\pi_{\bf d}$ map Euler finite complexes to Euler finite complexes. 
\item The induced $\mZ((q))$-linear morphism between completed Grothendieck groups 
\begin{align*} 
\displaystyle  [\bigoplus_{k=0}^n \mathbb{L}({}_{k} {\hat{\iota}}_{\bf d})] \colon  \quad
    &
    [D^\triangledown(\bigoplus_{k=0}^n \gmod-A_{k,{\bf d}})] \rightarrow  
    [D^\triangledown(\bigoplus_{k=0}^n \gmod-A_{k,1^n})] 
\end{align*}    
    is equal to the tensor
    product of the inclusion maps $\iota_{d_1} \otimes \cdots \otimes \iota_{d_r}$.
\item The induced $\mZ((q))$-linear  morphism between completed Grothendieck groups 
\begin{align*} \displaystyle [\bigoplus_{k=0}^n \mathbb{L}({}_{k} {\hat{\pi}}_{\bf d})] \colon\quad 
&
[D^\triangledown(\bigoplus_{k=0}^n \gmod-A_{k,1^n})]
    \rightarrow 
   [D^\triangledown(\bigoplus_{k=0}^n \gmod-A_{k,{\bf d}})]
\end{align*}
    is equal to the tensor product of the projection maps $
    \pi_{d_1} \otimes \cdots \otimes \pi_{d_r}.$
\end{enumerate}
\end{theorem}
\begin{proof}
See \cite[Theorem 46]{FSS1}.
\end{proof}

\begin{remark}{\rm
\label{stillneeded}
Theorem~\ref{catJW} provides a categorification of the Jones-Wenzl projector. Later on we will introduce cup and cap functors.  By \cite[Theorem 70]{FSS1} the obvious categorified version of Proposition~ \ref{charJW} holds. 
See \cite[Theorem 70]{FSS1} for the properties needed to characterise this functor uniquely.}
\end{remark}

\begin{ex}[Infinite complexes]
{\rm The complexity of the above functors is already transparent in Example \ref{JWex}: namely $\iota_2\circ\pi_2(v_0\otimes v_1)=\iota_2(q^{-1}[2]^{-1}v_1)=[2]^{-1}(v_1\otimes v_0+q^{-1}v_0\otimes v_1)=\frac{1}{1+q^2}(qv_1\otimes v_0+v_0\otimes v_1)$ rewritten using formula \eqref{quantum2} gets categorified by the infinite resolution
$$\cdots\stackrel{f} {\longrightarrow}\hat{P}(01)\langle4\rangle\stackrel{f}\longrightarrow
\hat{P}(01)\langle 2\rangle\stackrel{f}\longrightarrow\hat{P}(01),$$
where $P(01)$ fits into a short exact sequence of the form $\hat{M}(10)\langle 1\rangle\rightarrow\hat{P}(01)\rightarrow \hat{M}(01)$
and $f$ is the unique up to a scalar degree $2$ element in $\End_{A_{1,1^2}}(\hat{P}(01))\cong\mC[x]/[x^2]$. Note that this is a complex which has homology in all degrees!
}
\end{ex}

The first part of Theorem~\ref{catJW} may be refined as follows, (see Figure \ref{projeat} for an illustration).
\begin{theorem} \label{JWgenidempotent}
Let ${\bf d}$ be a composition of $n$ and ${\bf e}$ a refinement. Then $${\hat{p}}_{k,\bf e}{\hat{p}}_{k,\bf d}\cong{\hat{p}}_{k,\bf d}\cong{\hat{p}}_{k,\bf d}{\hat{p}}_{k,\bf e}.$$
Moreover, for any refinements ${\bf e_i}$ of ${\hat{p}}_{k,\bf d}$ we have ${\hat{p}}_{k,\bf d}\cong {\hat{p}}_{k,\bf e_1}{\hat{p}}_{k,\bf e_2}\cdots{\hat{p}}_{k,\bf e_r}$
as long as each part of ${\bf d}$ appears as part in at least one ${\bf e_i}$.
\end{theorem}
\begin{figure}
\begin{equation*}
\begin{tikzpicture}[scale=0.8]

\draw (6,0) -- (8,0)[thick];
\draw (6,0) -- (6,.25)[thick];
\draw (8,0) -- (8, .25)[thick];
\draw (6,.25) -- (8,.25)[thick];

\draw (6.25, .25) -- (6.25,1)[thick];
\draw (7, .25) -- (7,1)[thick];
\draw (7.75, .25) -- (7.75,1)[thick];

\draw (6.25, 0) -- (6.25,-1)[thick];
\draw (7, 0) -- (7,-1)[thick];
\draw (7.75, 0) -- (7.75,-1)[thick];

\draw (9,0)  node{$=$};

\draw (11.5,0.5) -- (11.5, .75)[thick];
\draw (10.5,0.5) -- (10.5, .75)[thick];
\draw (10.5,0.5) -- (11.5,0.5)[thick];
\draw (10.5,0.75) -- (11.5,0.75)[thick];

\draw (10,0) -- (12,0)[thick];
\draw (10,0) -- (10,.25)[thick];
\draw (12,0) -- (12, .25)[thick];
\draw (10,.25) -- (12,.25)[thick];

\draw (10.25, .25) -- (10.25,1)[thick];
\draw (11, .25) -- (11,.5)[thick];
\draw (11, .75) -- (11,1)[thick];
\draw (11.75, .25) -- (11.75,1)[thick];

\draw (10.25, 0) -- (10.25,-1)[thick];
\draw (11, 0) -- (11,-1)[thick];
\draw (11.75, 0) -- (11.75,-1)[thick];

\draw (13,0)  node{$=$};

\draw (15.5,-0.25) -- (15.5, -.5)[thick];
\draw (14.5,-0.5) -- (14.5, -.25)[thick];
\draw (14.5,-0.25) -- (15.5,-0.25)[thick];
\draw (14.5,-0.5) -- (15.5,-0.5)[thick];

\draw (14,0) -- (16,0)[thick];
\draw (14,0) -- (14,.25)[thick];
\draw (16,0) -- (16, .25)[thick];
\draw (14,.25) -- (16,.25)[thick];

\draw (14.25, .25) -- (14.25,1)[thick];
\draw (15, .25) -- (15,1)[thick];
\draw (15.75, .25) -- (15.75,1)[thick];

\draw (14.25, 0) -- (14.25,-1)[thick];
\draw (15, 0) -- (15,-.25)[thick];
\draw (15, -.5) -- (15,-1)[thick];
\draw (15.75, 0) -- (15.75,-1)[thick];
\end{tikzpicture}
\end{equation*}
\caption{{\small Properties of projectors}}
\label{projeat}
\end{figure}
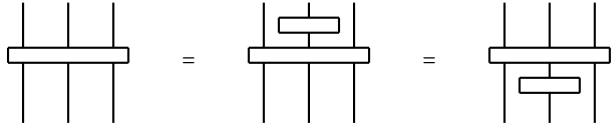

\begin{proof}
By definition of the quotient functor $G:={\hat{\pi}}_{k,\bf e}\mathbb{L}{\hat{\iota}}_{k,\bf d}$ is the identity functor on the additive category given by projectives in $\gmod-A_{k, \bf d}$, and hence also $\mathbb{L}{\hat{\iota}}_{k,\bf e}G\cong \mathbb{L}{\hat{\iota}}_{k,\bf d}$ on this subcategory. Since the derived category is generated by projectives, the first statement of the theorem follows by precomposing the functors by ${\hat{\pi}}_{k,\bf d}$. The second statement follows by similar arguments which are therefore omitted.
\end{proof}

\section{A Jones-Wenzl complex} \label{secJWcomplex}
\subsection{Lauda's $2$-category \cite{L}}
The $2$-category $\mathbf{U}$ is a an additive graded $\mathbb{C}$-linear category.  For each $\lambda \in \mathbb{Z}$, there is an object $\mathsf{1}_{\lambda}$ .  The $1$-morphisms are generated by symbols 
\begin{equation*}
\mathsf{E} \colon \mathsf{1}_{\lambda} \rightarrow \mathsf{1}_{\lambda +2},
\hspace{.5in}
\mathsf{F} \colon \mathsf{1}_{\lambda} \rightarrow \mathsf{1}_{\lambda -2}
\end{equation*}
which we often write as
$ \mathsf{1}_{\lambda+2} \mathsf{E} \mathsf{1}_{\lambda} $ and
$ \mathsf{1}_{\lambda-2} \mathsf{F} \mathsf{1}_{\lambda}$ respectively. The $1$-morphisms may be drawn as oriented strands.
\begin{equation}
\label{1morphisms}
\begin{tikzpicture}
\draw (-1.5,0) -- (4,0)[][very thick];
\draw (-.2, .7) node{Generating};
\draw (-.2, .3) node{$1$-morphisms};
\draw (-.2, -1) node{Diagrams};
\draw (1,1) -- (1,-2)[][very thick];
\draw (1.8,.5) node{$\mathsf{1}_{\lambda+2} \mathsf{E} \mathsf{1}_{\lambda}$};
\draw (2,-.5) -- (2,-1.5)[<-][thick];
\draw (2.25, -1) node{$\lambda$};
\draw (1.5, -1) node{$\lambda+2$};
\draw (2.5,1) -- (2.5,-2)[][very thick];

\draw (3.3,.5) node{$\mathsf{1}_{\lambda-2} \mathsf{F} \mathsf{1}_{\lambda}$};
\draw (3.5,-.5) -- (3.5,-1.5)[<-][thick];
\draw (3.75, -1) node{$\lambda$};
\draw (3, -1) node{$\lambda-2$};
\draw (4,1) -- (4,-2)[][very thick];

\draw (-1.5,1) -- (-1.5,-2)[very thick];
\draw (-1.5,1) -- (4,1)[][very thick];
\draw (-1.5,-2) -- (4,-2)[][very thick];
\end{tikzpicture}
\end{equation}
The $2$-morphisms are generated by the following diagrams.

\begin{equation}
\label{degrees}
\begin{tikzpicture}
[scale=0.85]
\draw (-1.5,0) -- (12,0)[][very thick];
\draw (-.2,.5) node{Degrees};
\draw (-.2, -.8) node{Generating};
\draw (-.2, -1.2) node{$2$-morphisms};
\draw (1,1) -- (1,-2)[][very thick];

\draw (1.5,-.5) -- (1.5,-1.5)[<-][thick];
\draw (3,1) -- (3, -2)[very thick];
\draw (1.5, .5) node{$2$};
\filldraw[black](1.5,-1) circle (2pt);

\draw (2.5,-.5) -- (2.5,-1.5)[->][thick];
\draw (2,1) -- (2, -2)[very thick];
\draw (2.5, .5) node{$2$};
\filldraw[black](2.5,-1) circle (2pt);

\draw (3.5,-1.5) -- (4.5,-.5)[->][thick];
\draw (4.5, -1.5) -- (3.5, -.5)[->][thick];
\draw (4,.5) node{$-2$};
\draw (-1.5,1) -- (-1.5,-2)[very thick];
\draw (5,1) -- (5,-2)[very thick];
\draw (-1.5,1) -- (12,1)[][very thick];
\draw (-1.5,-2) -- (12,-2)[][very thick];
\draw (12,1) -- (12,-2)[][very thick];
\draw (10.5,1) -- (10.5,-2)[][very thick];
\draw (8.75,1) -- (8.75,-2)[][very thick];
\draw (7,1) -- (7,-2)[][very thick];
\draw (6,.5) node{$1+\lambda$};
\draw (6,-1.5) node{$\lambda$};
\draw (7.75,.5) node{$1-\lambda$};
\draw (7.75,-1.5) node{$\lambda$};

\draw (9.5,.5) node{$1-\lambda$};
\draw (10,-.3) node{$\lambda$};

\draw (11.25,.5) node{$1+\lambda$};
\draw (11.75,-.3) node{$\lambda$};

\draw (5.5,-.5) arc (180:360:.5)[->] [thick];
\draw (7.25,-.5) arc (180:360:.5)[<-] [thick];
\draw (9,-.95) arc (180:0:.5)[->] [thick];
\draw (10.75,-.95) arc (180:0:.5)[<-] [thick];

\end{tikzpicture}
\end{equation}
The $2$-morphisms are taken subject to the following relations.

\begin{itemize}
\item The biadjointness relations.
\begin{equation}
\label{biadj1}
\begin{tikzpicture}
[scale=0.85]
\draw[thick] (0,0) to (0,1) arc(180:0:.5) arc(180:360:.5) to (2,2)[->];
\draw (2.5,1) node {$=$};
\draw [thick] (3,0) -- (3,2)[->];
\draw (3.5,1) node {$=$};
\draw[thick] (4,2)[<-] to (4,1) arc(180:360:.5) arc(180:0:.5) to (6,0);
  
\draw [shift={+(-1,0)}] [thick] (8,0) to (8,1)[<-] arc(180:0:.5) arc(180:360:.5) to (10,2);
\draw [shift={+(-1,0)}] (10.5,1) node {$=$};
\draw [shift={+(-1,0)}][thick] (11,0) -- (11,2)[<-];
\draw [shift={+(-1,0)}](11.5,1) node {$=$};
\draw[thick] [shift={+(-1,0)}](12,2) to (12,1) arc(180:360:.5) arc(180:0:.5) to (14,0)[->];
\end{tikzpicture}
\end{equation}

\begin{equation}
\label{biadj2}
\begin{tikzpicture}[>=stealth]
\draw (0,0) -- (0,-.5) [->][thick];
\filldraw [black] (0,-.25) circle (2pt);
\draw (1,0) -- (1,-.5) [thick];
\draw (0,0) arc (180:0:.5) [thick];
\draw (1.5,0) node {=};
\draw (2,0) -- (2,-.5)[->] [thick];
\filldraw [black] (3,-.25) circle (2pt);
\draw (3,0) -- (3,-.5) [thick];
\draw (2,0) arc (180:0:.5) [thick];

\draw (5,0) -- (5,-.5) [thick];
\filldraw [black] (5,-.25) circle (2pt);
\draw (6,0) -- (6,-.5)[->] [thick];
\draw (5,0) arc (180:0:.5) [thick];
\draw (6.5,0) node {=};
\draw (7,0) -- (7,-.5) [thick];
\filldraw [black] (8,-.25) circle (2pt);
\draw (8,0) -- (8,-.5) [->][thick];
\draw (7,0) arc (180:0:.5) [thick];
\end{tikzpicture}
\end{equation}

\begin{equation}
\label{biadj3}
\begin{tikzpicture}[>=stealth]
\draw (0,0) -- (.5,.5)[thick];
\draw (.5,0) -- (0,.5)[thick];
\draw (1,0) -- (1,.5)[->][thick];
\draw (1.5,0) -- (1.5,.5)[->][thick];
\draw (.5,.5) arc (180:0:.25) [thick];
\draw (0,.5) arc (180:0:.75) [thick];
\draw (2,.5) node{=};

\draw (2.5,0) -- (2.5,.5)[<-][thick];
\draw (3,0) -- (3,.5)[<-][thick];
\draw (3.5,0) -- (4,.5)[thick];
\draw (4,0) -- (3.5,.5)[thick];
\draw (3,.5) arc (180:0:.25) [thick];
\draw (2.5,.5) arc (180:0:.75) [thick];

\draw [shift={+(6,0)}] (0,0) -- (.5,.5)[thick];
\draw  [shift={+(6,0)}] (.5,0) -- (0,.5)[thick];
\draw  [shift={+(6,0)}] (1,0) -- (1,.5)[<-][thick];
\draw  [shift={+(6,0)}] (1.5,0) -- (1.5,.5)[<-][thick];
\draw  [shift={+(6,0)}] (.5,.5) arc (180:0:.25) [thick];
\draw  [shift={+(6,0)}](0,.5) arc (180:0:.75) [thick];
\draw  [shift={+(6,0)}] (2,.5) node{=};

\draw  [shift={+(6,0)}](2.5,0) -- (2.5,.5)[->][thick];
\draw  [shift={+(6,0)}](3,0) -- (3,.5)[->][thick];
\draw  [shift={+(6,0)}](3.5,0) -- (4,.5)[thick];
\draw  [shift={+(6,0)}](4,0) -- (3.5,.5)[thick];
\draw  [shift={+(6,0)}](3,.5) arc (180:0:.25) [thick];
\draw  [shift={+(6,0)}](2.5,.5) arc (180:0:.75) [thick];

\end{tikzpicture}
\end{equation}

\item The bubble relations.
\begin{equation}
\label{BUBBLES1}
\begin{tikzpicture}[>=stealth]
\draw (.25,0)[<-] arc (180:360:0.5cm) [thick];
\draw(1.25,0) arc (0:180:0.5cm) [thick];
\draw (1,0) node [anchor=east] {$k$};
\draw (1.5,0) node {$=$};
\draw (1.75,0) node {$0$};
\draw (2,0) node {$=$};
\draw [shift={+(2.25,0)}](0,0)[->] arc (180:360:0.5cm) [thick];
\draw [shift={+(2.25,0)}](1,0) arc (0:180:0.5cm) [thick];
\draw [shift={+(2.25,0)}](.75,0) node [anchor=east] {$k$};
\draw (4.25,0) node{$\text{ if } k<0$,};
\end{tikzpicture}
\quad \quad
\begin{tikzpicture}[>=stealth]
\draw [shift={+(-.75,0)}](0,0)[<-] arc (180:360:0.5cm) [thick];
\draw [shift={+(-.75,0)}](1,0) arc (0:180:0.5cm) [thick];
\draw [shift={+(-.75,0)}](.75,0) node [anchor=east] {$0$};
\draw [shift={+(-1,0)}](1.5,0) node {$=$};
\draw [shift={+(-1.25,0)}](2,0) node {$1$};
\draw [shift={+(-1.5,0)}](2.5,0) node {$=$};
\draw [shift={+(-1.75,0)}][shift={+(3,0)}](0,0)[->] arc (180:360:0.5cm) [thick];
\draw [shift={+(-1.75,0)}][shift={+(3,0)}](1,0) arc (0:180:0.5cm) [thick];
\draw [shift={+(-1.75,0)}][shift={+(3,0)}](.75,0) node [anchor=east] {$0$};
\draw (4.75,0) node{$$};
\end{tikzpicture}
\end{equation}



\item The infinite Grassmannian relations.
\begin{equation}
\label{BUBBLES2}
\begin{tikzpicture}[>=stealth]
\draw [shift={+(-.75,0)}] (-.5,0) node {$\displaystyle \sum_{k \geq 0}$};
\draw [shift={+(-.75,0)}](0,0)[<-] arc (180:360:0.5cm) [thick];
\draw [shift={+(-.75,0)}](1,0) arc (0:180:0.5cm) [thick];
\draw [shift={+(-.75,0)}](.75,0) node [anchor=east] {$k$};
\draw [shift={+(-.75,0)}] (1.5,0) node {$t^k$};
\draw [shift={+(-.75,0)}] (-.8,0) node {$($};
\draw [shift={+(-.75,0)}] (1.8,0) node {$)$};

\draw [shift={+(3,0)}]  [shift={+(-.75,0)}] (-.5,0) node {$\displaystyle \sum_{k \geq 0}$};
\draw [shift={+(3,0)}][shift={+(-.75,0)}](0,0)[->] arc (180:360:0.5cm) [thick];
\draw [shift={+(3,0)}] [shift={+(-.75,0)}](1,0) arc (0:180:0.5cm) [thick];
\draw [shift={+(3,0)}][shift={+(-.75,0)}](.75,0) node [anchor=east] {$k$};
\draw [shift={+(3,0)}] [shift={+(-.75,0)}] (1.5,0) node {$t^k$};
\draw [shift={+(3,0)}][shift={+(-.75,0)}] (-.8,0) node {$($};
\draw [shift={+(3,0)}] [shift={+(-.75,0)}] (1.8,0) node {$)$};

\draw [shift={+(2,0)}](2.5,0) node {$=$};
\draw [shift={+(2,0)}](3,0) node {$1$};

\end{tikzpicture}
\end{equation}

\item The Nil-Hecke relations.
\begin{equation}
\label{nilhecke1}
\begin{tikzpicture}[>=stealth]
\draw (0,0) -- (1,1)[->][thick];
\draw (1,0) -- (0,1)[->][thick];
\filldraw [black] (.25,.25) circle (2pt);
\draw (1.5,.5) node{$-$};
\draw (2,0) -- (3,1)[->][thick];
\draw (3,0) -- (2,1)[->][thick];

\filldraw [black] (2.75, .75) circle (2pt);
\draw (3.5,.5) node{=};
\draw (4,0) -- (4,1)[->][thick];
\draw (5,0) -- (5,1)[->][thick];

\draw (5.5,.5) node{=};

\draw (6,0) -- (7,1)[->][thick];
\draw (7,0) -- (6,1)[->][thick];
\filldraw [black] (6.25,.75) circle (2pt);
\draw (7.5,.5) node{$-$};
\draw (8,0) -- (9,1)[->][thick];
\draw (9,0) -- (8,1)[->][thick];

\filldraw [black] (8.75, .25) circle (2pt);
\end{tikzpicture}
\end{equation}

\begin{equation}
\label{nilhecke2}
\begin{tikzpicture}
\draw (0,0) .. controls (1,1) .. (0,2)[->][thick];
\draw (1,0) .. controls (0,1) .. (1,2)[->][thick];
\draw (1.5,1) node{$=$};
\draw (2,1) node{$0$ }  ;
\end{tikzpicture}
\quad \quad \quad
\begin{tikzpicture}[>=stealth]
\draw (0,0) -- (2,2)[->][thick];\draw (2,0) -- (0,2)[->][thick];
\draw (1,0) .. controls (0,1) .. (1,2)[->][thick];
\draw (2.5,1) node {=};
\draw (3,0) -- (5,2)[->][thick];
\draw (5,0) -- (3,2)[->][thick];
\draw (4,0) .. controls (5,1) .. (4,2)[->][thick];
\end{tikzpicture}
\end{equation}


\item The reduction to bubbles relations.
\begin{equation}
\label{REDBUBBLES1}
\begin{tikzpicture}[>=stealth]
\draw (1.2,.5) node{$\lambda$};
\draw  (2,0) .. controls (2,.5) and (1.3,.5) .. (1.1,0) [thick];
\draw  (2,0) .. controls (2,-.5) and (1.3,-.5) .. (1.1,0) [thick];
\draw  (1,-1) .. controls (1,-.5) .. (1.1,0) [thick];
\draw  (1.1,0) .. controls (1,.5) .. (1,1) [->] [thick];
\draw [shift={+(1,0)}](1.5,0) node {$= -$};
\draw [shift={+(1,0)}] (2,0) node {$\displaystyle \sum_{a+b=-\lambda}$};
\draw [shift={+(4.5,0)}](0,0)[<-] arc (180:360:0.5cm) [thick];
\draw [shift={+(4.5,0)}](1,0) arc (0:180:0.5cm) [thick];
\draw [shift={+(5,0)}](.25,0) node [anchor=east] {$a$};
\draw(4,-1) --(4,1)[->][thick];
\filldraw [black] (4,0) circle (2pt);
\draw (4.25,0) node {$b$};
\draw (4.35,.5) node {$\lambda$ };

\end{tikzpicture}
\quad \quad
\begin{tikzpicture}[>=stealth]
\draw (-.2,.5) node{$\lambda$};
\draw  (0,0) .. controls (0,.5) and (.7,.5) .. (.9,0) [thick];
\draw  (0,0) .. controls (0,-.5) and (.7,-.5) .. (.9,0) [thick];
\draw  (1,-1) .. controls (1,-.5) .. (.9,0) [thick];
\draw  (.9,0) .. controls (1,.5) .. (1,1) [->] [thick];
\draw (1.5,0) node {$=$};
\draw  (2,0) node {$\displaystyle \sum_{a+b=\lambda}$};
\draw [shift={+(2.5,0)}](0,0)[->] arc (180:360:0.5cm) [thick];
\draw [shift={+(2.5,0)}](1,0) arc (0:180:0.5cm) [thick] ;
\draw [shift={+(2.5,0)}](.75,0) node [anchor=east] {$a$};
\draw(4,-1) --(4,1)[->][thick];
\filldraw [black] (4,0) circle (2pt);
\draw (4.25,0) node {$b$};
\draw (3.5,.5) node {$\lambda$};

\end{tikzpicture}
\end{equation}



\item The identity relations.

\begin{equation}
\label{ID1}
\begin{tikzpicture}[>=stealth]
\draw (1.85,1) node{$-$};
\draw (1.95,1.5) node{$\lambda$};

\draw  (2,0) .. controls (3,1) .. (2,2)[->][thick];
\draw (3,0) .. controls (2,1) .. (3,2)[<-] [thick];
\draw (1.5,1) node {=};
\draw (-.5,1) node{$\lambda$};
\draw(0,0) --(0,2)[->][thick];
\draw (1,0) -- (1,2)[<-][thick];
\draw (3.8,1) node{$+ \displaystyle \sum \limits_{a+b+c=\lambda-1}$};
\draw (5.2,1) node{$\lambda$};

\draw [shift={+(1,0)}] (4,1.75) arc (180:360:.5) [thick];
\draw [shift={+(1,0)}] (4,2) -- (4,1.75) [thick];
\draw [shift={+(1,0)}](5,2) -- (5,1.75) [thick][->];
\draw [shift={+(1,0)}] (5,.25) arc (0:180:.5) [thick];
\filldraw [shift={+(1,0)}][black] (4.5,1.25) circle (2pt);
\draw [shift={+(1,0)}] (4.5,1.25) node [anchor=south] {$c$};
\filldraw [shift={+(1,0)}] [black] (4.5,0.75) circle (2pt);
\draw [shift={+(1,0)}] (4.5,.75) node [anchor=north] {$a$};
\draw [shift={+(1,0)}] (5,0) -- (5,.25) [thick];
\draw [shift={+(1,0)}] (4,0) -- (4,.25) [thick][->];

\draw [shift={+(6,1)}](0,0)[->] arc (180:360:0.5cm) [thick];
\draw [shift={+(6,1)}](1,0) arc (0:180:0.5cm) [thick];
\draw [shift={+(6,1)}](.75,0) node [anchor=east] {$b$};

\end{tikzpicture}
\end{equation}

\begin{equation}
\label{ID2}
\begin{tikzpicture}[>=stealth]
\draw (1.85,1) node{$-$};
\draw (1.95,1.5) node{$\lambda$};

\draw  (2,0) .. controls (3,1) .. (2,2)[<-][thick];
\draw (3,0) .. controls (2,1) .. (3,2)[->] [thick];
\draw (1.5,1) node {=};
\draw (-.5,1) node{$\lambda$};
\draw(0,0) --(0,2)[<-][thick];
\draw (1,0) -- (1,2)[->][thick];
\draw (3.8,1) node{$+ \displaystyle \sum \limits_{a+b+c=-\lambda-1}$};
\draw (5.2,1) node{$\lambda$};

\draw [shift={+(1,0)}] (4,1.75) arc (180:360:.5) [thick];
\draw [shift={+(1,0)}] (4,2) -- (4,1.75) [thick];
\draw [shift={+(1,0)}](5,2) -- (5,1.75) [thick][<-];
\draw [shift={+(1,0)}] (5,.25) arc (0:180:.5) [thick];
\filldraw [shift={+(1,0)}][black] (4.5,1.25) circle (2pt);
\draw [shift={+(1,0)}] (4.5,1.25) node [anchor=south] {$c$};
\filldraw [shift={+(1,0)}] [black] (4.5,0.75) circle (2pt);
\draw [shift={+(1,0)}] (4.5,.75) node [anchor=north] {$a$};
\draw [shift={+(1,0)}] (5,0) -- (5,.25) [thick];
\draw [shift={+(1,0)}] (4,0) -- (4,.25) [thick][<-];

\draw [shift={+(6,1)}](0,0) arc (180:360:0.5cm) [thick];
\draw [shift={+(6,1)}][<-](1,0) arc (0:180:0.5cm) [thick];
\draw [shift={+(6,1)}](.75,0) node [anchor=east] {$b$};

\end{tikzpicture}
\end{equation}
\end{itemize}

\begin{definition}
The nil-Hecke algebra $NH_n$ is the $\mathbb{Z}$-graded algebra generated by $y_1, \ldots, y_n$ of degree $2$ and generators $\psi_1, \ldots, \psi_{n-1}$ of degree $-2$ with relations 
\begin{enumerate}
\item $y_i y_j = y_j y_i $ for all $i$ and $j$,
\item $ \psi_i^2 = 0 $ for $i=1, \ldots, n-1$,
\item $ \psi_i \psi_j = \psi_j \psi_i$ if $|i-j|>1$,
\item $ \psi_i \psi_{i+1} \psi_i = \psi_{i+1} \psi_i \psi_{i+1} $ for $i=1, \ldots, n-2$,
\item $y_i \psi_i - \psi_i y_{i+1} = 1 = \psi_i y_i - y_{i+1} \psi_i$ for $i=1,\ldots n-1$.
\end{enumerate}
\end{definition}

Let $w_0$ be the longest element in the symmetric group and $\psi_{w_0}$ the corresponding element in $NH_n$.
Define the indecomposable idempotent in $NH_n$
\begin{equation*}
\epsilon_{w_0} = y_1^{n-1} y_2^{n-2} \cdots y_{n-1} \psi_{w_0}.
\end{equation*}
The idempotent $\epsilon_{w_0}$ is indecomposable since $\End_{NH_n}(NH_n \epsilon_{w_0})$ is a non-negatively graded algebra which is one-dimensional in degree zero.

Since the $2$-morphisms of $ \mathbf{U}$ satisfy nil-Hecke relations, there is an action of $NH_n$ on $\mathsf{E}^n$ and $ \mathsf{F}^n$ respectively. In the Karoubi envelope $\Kar(\mathbf{U})$ of $\mathbf{U}$ define objects
\begin{equation*}
\mathsf{E}^{(n)} = (\mathsf{E}^n,e_{w_0}) \langle \frac{n(n-1)}{2}   \rangle, \quad \quad
\mathsf{F}^{(n)} = (\mathsf{F}^n,e_{w_0}) \langle \frac{n(n-1)}{2}   \rangle.
\end{equation*}

\begin{theorem}
\cite[Theorem 9.13]{L}
There is an isomorphism $[\Kar(\mathbf{U})] \cong \dot{\mathcal{U}}_q$
where $\dot{\mathcal{U}}_q $ is Lusztig's idempotent version of $\mathcal{U}_q$.
\end{theorem}

\subsection{The complex}

For a $k$-tuple ${\bf n}=(n_1, \ldots, n_k) \in \mathbb{N}^k$ define
\begin{equation*}
|{\bf n}|=n_1+\cdots+n_k ,
\hspace{.4in}
\wt({\bf n})=\sum_{i=1}^k i n_i ,
\hspace{.4in}
{\bf e}^{\bf n}=e_1^{n_1} \cdots e_k^{n_k} ,
\hspace{.4in}
c_{\bf n}=\frac{|{\bf n}|!}{n_1! \cdots n_k!}.
\end{equation*}
Then set
\begin{equation*}
r_j = \sum_{\substack{{\bf n} \in \mathbb{N}^k \\ \wt({\bf n})=n-k+j }}
(-1)^{|{\bf n}|} c_{{\bf n}} {\bf e}^{\bf n}.
\end{equation*}

The cohomology of the Grassmannian of $k$-dimensional planes in $\C^n$ plays an important role in what follows.  We will sometimes abbreviate this cohomology by $\Hstar:=\Hstar(\Gr(k,n))$ and we have the following well known classical result.

\begin{prop}
The cohomology of the Grassmannian $\Hstar(\Gr(k,n))$ is given by:
\begin{equation*}
\Hstar(\Gr(k,n)) \cong \mathbb{C}[e_1, \ldots, e_k]/I_{k,n}
\end{equation*}
where $I_{k,n}$ is the ideal generated by $r_1, \ldots, r_k$.
\end{prop}

We follow now the notation and ideas of Wolffhardt ~\cite{Wolf}.  He finds free bimodule resolutions of complete intersection rings 
in general but we focus here on the specific example of cohomology of Grassmannians.
Let $V$ be a $\mathbb{Z}^2$-graded vector space spanned by vectors
$b_i$ of degree 
$(-2,2n-2k+2i)$
and vectors $f_i$ of degree 
$(-1,2i)$
for $i=1,\ldots, k$.  
Let $S(V)$ be the $\mathbb{Z}^2$-graded symmetric algebra of $V$.  It is spanned by elements of the form
\begin{equation*}
b_{j_1} \wedge \cdots \wedge b_{j_m} \wedge f_{i_1} \wedge \cdots \wedge f_{i_n}
\end{equation*}
with $j_1 \leq \cdots \leq j_m $ and
$i_1 < \cdots < i_n$.  In $S(V)$, we have $b_j \wedge x = x \wedge b_j$ for all $x \in S(V)$ while $f_{i} \wedge f_{i'}= - f_{i'} \wedge f_i$.
For a homogeneous element $x \in S(V)$ we refer to the first grading as the $q$-degree and denote it by $\DEG_q(x)$ and the second grading as the homological grading and denote it by $ \DEG_h(x)$. For $i=0,\ldots,k$ define
\begin{equation*}
\tau_i, \partial_i \colon \Hstar(\Gr(k,n)) \rightarrow \Hstar(\Gr(k,n)) \otimes \Hstar(\Gr(k,n))
\end{equation*}
by
\begin{equation*}
\tau_i(e_j) = 
\begin{cases}
e_j \otimes 1 \hspace{.2in} \text{ for } \hspace{.2in} 1 \leq j \leq i, \\
1 \otimes e_j \hspace{.2in} \text{ for } \hspace{.2in} i < j \leq k,
\end{cases}
\hspace{.2in}
\partial_i := \frac{1}{\tau_k (e_i) - \tau_0 (e_i)}(\tau_i - \tau_{i-1}).
\end{equation*}
It is easy to check that the map $\partial_i$ is a well-defined.
Consider the vector space 
\begin{equation*}
\Hstar(\Gr(k,n)) \otimes_{\mathbb{C}} \Hstar(\Gr(k,n)) \otimes S(V)
\end{equation*}
endowed with the obvious algebra structure, explicitly:
\begin{equation*}
(a_1 \otimes b_1 \otimes v_1)(a_2 \otimes b_2 \otimes v_2)=
(a_1 a_2 \otimes b_1 b_2 \otimes v_1 \wedge v_2).
\end{equation*}

Define a complex of $\Hstar(\Gr(k,n)) \otimes_{\mathbb{C}} \Hstar(\Gr(k,n))$-bimodules:
\begin{equation}
\label{grbimodcomplex}
(\Hstar(\Gr(k,n)) \otimes_{\mathbb{C}} \Hstar(\Gr(k,n)) \otimes S(V), \partial)
\end{equation}
with 
\[
\partial(1 \otimes 1 \otimes f_j)=e_j \otimes 1 \otimes 1 - 1 \otimes e_j \otimes 1 \ , \quad \quad
\partial(1 \otimes 1 \otimes b_j) = \sum_{i=1}^k \partial_i(r_j) \otimes f_i  \ ,
\]
\[
\partial(1 \otimes 1 \otimes (v \wedge w)) = 
\partial(1 \otimes 1 \otimes v) (1 \otimes 1 \otimes w) + 
(-1)^{\DEG_h(v)} (1 \otimes 1 \otimes v) \partial(1 \otimes 1 \otimes w) \ ,
\]
and the $(\Hstar(\Gr(k,n)),\Hstar(\Gr(k,n)))$-bimodule action is given naturally on the first two tensor factors.

\begin{prop}
\cite[Theorem 2]{Wolf}
The complex in ~\eqref{grbimodcomplex} is a free bimodule resolution of 
$\Hstar(\Gr(k,n))$.
\end{prop}


\subsection{Properties of the complex}
In order to connect the complex \eqref{grbimodcomplex} to the categorified Jones-Wenzl projector from the previous section, we first recall some important results of Soergel.

\begin{prop} \cite[Endomorphismensatz]{SoergelKatO} \label{EndofAntiDom}
There is an isomorphism of algebras
\begin{equation*}
\End_{\mathfrak{g}}(P(0^{n-k} 1^k)) \cong \Hstar(\Gr(k,n)).
\end{equation*}
\end{prop}

The {\it Soergel functor} 
$\mathbb{V}_{k,n} \colon \gmod-A_{k,(1^n)} \rightarrow \gmod-\Hstar(\Gr(k,n))$ is defined as
\begin{equation*}
\mathbb{V}_{k,n}(M)=\Hom_{A_{k,(1^n)}}(\hat{P}(0^{n-k} 1^k), M)
\cong \hat{P}(0^{n-k} 1^k)   \otimes_{A_{k,(1^n)}} M .
\end{equation*}

Via the equivalence of categories
$\gmod-A_{k,(n)} \cong \gmod-\Hstar(\Gr(k,n))$,
we identify half of the categorified Jones-Wenzl projector $\hat{\pi}_{k,(n)}$ with the Soergel functor $\mathbb{V}_{k,n}$.

The left adjoint functor of the Soergel functor 
$\mathbb{T}_{k,n} \colon D^\triangledown(\gmod-\Hstar(\Gr(k,n))) \rightarrow 
D^\triangledown(\gmod-A_{k,(1^n)})$ can be described by
\begin{equation*}
\mathbb{T}_{k,n}(M)=\hat{P}(0^{n-k} 1^k)  \otimes^{\mathbb{L}}_{\Hstar(\Gr(k,n))} M .
\end{equation*}
Thus we could write the categorified Jones-Wenzl projector as
\begin{equation*}
\hat{p}_{k,n} \colon D^\triangledown(\gmod-A_{k,(1^n)}) \rightarrow D^\triangledown(\gmod-A_{k,(1^n)})
\end{equation*}
\begin{equation*}
\hat{p}_{k,n}(M)=  \hat{P}(0^{n-k} 1^k) \otimes^{\mathbb{L}}_{\Hstar} \hat{P}(0^{n-k} 1^k) \otimes_{A_{k,(1^n)}}  M.
\end{equation*}

There is a $2$-functor from Lauda's $2$-category to $\bigoplus_{k=0}^{n} \gmod-A_{k,1^n} $. By the discussion above, the endomorphism algebras of $\mathsf{E}^{(k)}$ and $\mathsf{F}^{(k)}$ are modules over $\Hstar(\Gr(k,n))$. 
  Thus we may
define the following complex in the homotopy category:
\begin{equation}
(\mathsf{p}_{k,n}, \partial) =  ((\mathsf{E}^{(k)} \mathsf{F}^{(k)}) \otimes_{\Hstar \otimes_{\mathbb{C}} \Hstar} 
(\Hstar \otimes_{\mathbb{C}} \Hstar) \otimes S(V), \partial).
\end{equation}

\begin{theorem}
\label{complexonO}
On $D^\triangledown(\gmod-A_{k,1^n})$, 
the functor $\mathsf{p}_{k,n}$ is isomorphic to 
$ \hat{p}_{k,(n)}$,
and thus the following isomorphisms hold:
\begin{enumerate}
\item $ \mathsf{p}_{k,n} \circ \mathsf{p}_{k,n} \cong \mathsf{p}_{k,n}$,
\item $ \hat{\mathcal{E}}_k \circ \mathsf{p}_{k,n} \cong \mathsf{p}_{k+1,n} \circ \mathcal{\hat{E}}_k$,
\item $ \mathcal{\hat{F}}_k \circ \mathsf{p}_{k,n} \cong \mathsf{p}_{k-1,n} \circ \mathcal{\hat{F}}_k$.
\end{enumerate}
\end{theorem}

\begin{proof}
In order to obtain a complex of $A_{k,1^n}$-bimodules quasi-isomorphic to 
$ \hat{p}_{k,(n)}$, we use the resolution of $\Hstar(\Gr(k,n))$ of free 
$(\Hstar(\Gr(k,n)), \Hstar(\Gr(k,n)))$-\\bimodules from ~\eqref{grbimodcomplex}. Recall that $\hat{p}_{k,n}$ is given by the complex of $(A_{k,(1^n)},A_{k,(1^n)})$-bimodules
\begin{equation*}
\hat{P}(0^{n-k} 1^k) \otimes^{\mathbb{L}}_{\Hstar}  \hat{P}(0^{n-k} 1^k).
\end{equation*}
We know from ~\eqref{grbimodcomplex} that as an 
$(\Hstar, \Hstar)$-bimodule, there is a quasi-isomorphism
\begin{equation}
\label{Grbimodres}
\Hstar \cong \Hstar \otimes_{\mathbb{C}} \Hstar \otimes S(V)
\end{equation}
Tensoring ~\eqref{Grbimodres} over $\Hstar$
by $\hat{P}(0^{n-k} 1^k)$, we get that $\hat{p}_{k,n}$ is quasi-isomorphic to
\begin{equation*}
\hat{P}(0^{n-k} 1^k)  \otimes_{\Hstar}  \Hstar \otimes_{\mathbb{C}} \Hstar \otimes_{\Hstar}  \hat{P}(0^{n-k} 1^k) \otimes S(V)
\end{equation*}
which is then isomorphic to
$\hat{P}(0^{n-k} 1^k)  \otimes_{\mathbb{C}}  \hat{P}(0^{n-k} 1^k) \otimes S(V) \cong \mathsf{E}^{(k)} \mathsf{F}^{(k)} \otimes S(V)$.
Thus $ \mathsf{p}_{k,n} \cong \hat{p}_{k,(n)}$.
The remaining statements of the theorem about $\mathsf{p}_{k,n}$
follow because we know the corresponding statements already for
$ \hat{p}_{k,(n)}$.
\end{proof}

\begin{remark}{\rm 
Webster also connected the categorified Jones-Wenzl projector to the cohomology of the Grassmannian in \cite[Section 4.5]{Webgrass}.

For the case $k=1$ (projective space) and also in the context of categorification at a root of unity, see \cite{QS2}.}
\end{remark}

\begin{remark}{\rm 
Let $\mathcal{C}_{1^n} = \oplus_{k=0}^n \mathcal{C}_{k,1^n}$ be a $2$-representation of $\mathbf{U}$ in the sense of Losev and Webster \cite{LW}, such that $[\mathcal{C}_{1^n}] \cong V_1^{\otimes n}$. Then on the homotopy category of $\mathcal{C}_{k,1^n}$ there are isomorphisms:
\vspace{-2mm}
\begin{multicols}{3}
\begin{enumerate}
\item $ \mathsf{p}_{k,n} \circ \mathsf{p}_{k,n} \cong \mathsf{p}_{k,n}$
\item $ \mathsf{E} \circ \mathsf{p}_{k,n} \cong \mathsf{p}_{k+1,n} \circ \mathsf{E}$
\item $ \mathsf{F} \circ \mathsf{p}_{k,n} \cong \mathsf{p}_{k-1,n} \circ \mathsf{F}$.
\end{enumerate}
\end{multicols}
\vspace{-2mm}
This is true for category $\mathcal{O}$ by Theorem ~\ref{complexonO} and it was shown in \cite{SaStr} that category $\mathcal{O}$ provides a $2$-representation of $\mathbf{U}$. 
The general case follows from the uniqueness result of tensor product categorifications of Losev and Webster \cite{LW} using \cite{SaStr}.}
\end{remark}
\section{Categorification of the uncoloured Reshetikhin-Turaev invariant}
\label{catuncoloredrt}
A categorification of the Reshetikhin-Turaev tangle invariant for the standard representation was first constructed in \cite{StrDuke}. The main result there is the following.

\begin{theorem}{\rm(\cite[Theorem 7.1, Remark 7.2]{StrDuke})}
\label{catjones} Let $ T $ be an oriented tangle from $ n $ points to $ m $ points.  Let $ D_1 $ and $ D_2 $ be two tangle diagrams of $
T. $ Let
$${\hat{\Phi}}(D_1),{\hat{\Phi}}(D_2)\colon \quad D^b\left(\bigoplus_{k=0}^n \gmod-A_{k,1^n} \right) \rightarrow
D^b\left(\bigoplus_{k=0}^m \gmod-A_{k,1^m}\right) $$ be the corresponding functors associated to the oriented tangle. Then there is an isomorphism of functors $
{\hat{\Phi}}(D_1)\langle 3 \gamma(D_1) \rangle \cong {\hat{\Phi}}(D_2)\langle 3 \gamma(D_2) \rangle $.
\end{theorem}

We now briefly explain how to associate a functor to a tangle. This is done by associating to each elementary tangle (cup, cap, braid) a functor. To a braid one associates a certain derived equivalence.
In order to prove that this is a tangle invariant, we must show that if two tangles are related by a Reidemeister move, then the associated functors are the same up to isomorphism.
Note that we work here in a Koszul dual picture of the one developed in \cite{StrDuke}, since we have a better understanding of the categorification of arbitrary tensor products in this context. The translation between these two picture is given by the result in \cite[Theorem 35, Theorem 39]{MOS} relating the corresponding functors via the Koszul duality equivalence of categories.

\subsection{Functors associated to cups and caps}
\label{catcupscaps}
In the following we briefly recall the definition of the functors and the main properties which will be used later. For each $1\leq i < n$ we will define now functors which sends a module to its maximal quotient which has only composition factors from a certain allowed set.

Given such $i$, consider the set $S$ of isomorphism classes of simple right $A_{k,1^n}$-modules  ${L}(a_1,a_2,\ldots, a_n)$ where the sequence ${\bf a}=(a_1,a_2,\ldots, a_n)$ is obtained from the sequence $(1^k0^{n-k})$ by applying an element $w\in \mathbb{S}_n$ which is a shortest coset representative in $\mathbb{S}_n / \mathbb{S}_{k}\times \mathbb{S}_{n-k}$ such that the entries in components $ i $ and $ i+1$ of the sequence resulting from $ w $ applied to
$(1^k0^{n-k})$ are $ 1 $ and $ 0 $ respectively. Let $\Mod-A^i_{k,1^n}$ be the full subcategory of $\Mod-A_{k,1^n}$ containing only modules with simple composition factors from the set $S$. There are the natural functors
\begin{eqnarray}
{\epsilon}_i: \Mod-A^i_{k,1^n} \rightarrow \Mod-A_{k,1^n} && {Z}_i: \Mod-A_{k,1^n} \rightarrow \Mod-A^i_{k,1^n},
\end{eqnarray}
of inclusion of the subcategory, respectively of taking the maximal quotient contained in the subcategory. Note that ${Z}_i$ is left adjoint to ${\epsilon}_i$.

The category $\gmod-A^i_{k,1^n}$ is a graded version of the so-called parabolic category $\cO$ defined as follows: let $\mathfrak{p}_i$ be the {\it $i$-th minimal parabolic} subalgebra of $ \mathfrak{g} $ which has basis the matrix units $E_{r,s}$, where $s\geq r$ or $r,s  \in \lbrace i, i+1 \rbrace$. Now replace locally $\mb$-finiteness in Definition \ref{defO} by locally $ \mathfrak{p}_i$-finiteness and obtain the {\it parabolic category} $ \mathcal{O}_k^i(\mathfrak{gl}_n)$, a full subcategory of $\cO_k$, see \cite[Section 9.3]{Hu}).  We have  $ \mathcal{O}_k^i\cong A^i_{k,1^n}-\op{mod}$. In this context $Z_i$ is the {\it Zuckerman functor} of taking the maximal locally finite quotient with respect to $ \mathfrak{p}_i$. That means we send a module $M\in\cO_k$ to the largest quotient in $\mathcal{O}_k^i(\mathfrak{gl}_n)$.
An important class of objects in these parabolic categories are the parabolic Verma modules $ M^{\mathfrak{p}_i}(a_1, \ldots, a_n) $ where $ {\bf a} = (a_1, \ldots, a_n) $ has the same conditions on it as the labels of the simple object in this category.  Each parabolic Verma module has a graded lift $ \hat{M}^{\mathfrak{p}_i}(a_1, \ldots, a_n) $ such that its head is concentrated in degree zero.
Now fix a graded lift $ \hat{Z}_i $ of $ Z_i $ such that $ \hat{Z}_i \hat{M}(a_1, \ldots, a_n) \cong \hat{M}^{\mathfrak{p}_i}(a_1, \ldots, a_n) \langle -1 \rangle $ (when $ a_i=1 $ and $ a_{i+1}=0$).
A classical result of Enright and Shelton \cite{ES} relates parabolic category $\cO$ with non-parabolic category $\cO$ for a smaller rank algebra. This equivalence was lifted to the graded setup in \cite{Rh}. For a geometric approach see  \cite{SoergelES}. The statement is the following.

\begin{prop}
Let $n\geq 0$. There is an equivalence of categories $\zeta_n \colon
    \mathcal{O}_k
    (\mathfrak{gl}_n) \rightarrow \mathcal{O}_{k+1}^1 (\mathfrak{gl}_{n+2})$ which can be lifted to an equivalence $ \hat{\zeta}_n \colon \gmod-A_{k,1^n} \cong  \\ \gmod-A^1_{k+1,1^{n+2}}$,
such that $ \hat{M}^{\mathfrak{p}_1}(a_1, \ldots, a_n) $ gets mapped to $ \hat{M}(a_3, \ldots, a_n) $.
For $n=0$ the corresponding category is equivalent to the category of graded vector spaces.
\end{prop}

Now there are functors (up to shifts in the internal and homological degree) pairwise adjoint in both directions
\begin{eqnarray*}
&{\hat{\cap}}_{i,n} \colon D^b(\gmod-A_{k,1^n}) \rightarrow D^b(\gmod-A_{k,1^{n-2}})&\\
&{\hat{\cap}}_{i,n}:={\hat{\zeta}_n}^{-1} \circ \mathbb{L} \hat{Z}_1\lsem -1 \rsem \circ \hat\epsilon_2 \circ \mathbb{L} \hat{Z}_2 \lsem -1 \rsem \circ \cdots \circ \hat\epsilon_{i}\circ \mathbb{L} \hat{Z}_i \lsem -1 \rsem&\\
&{\hat{\cup}}_{i,n}\colon D^b(\gmod-A_{k,1^n}) \rightarrow D^b(\gmod-A_{k,1^{n+2}})& \\
&{\hat{\cup}}_{i,n}:=\hat\epsilon_i \circ \mathbb{L} \hat{Z}_i \lsem -1 \rsem \circ \cdots \circ\hat\epsilon_2\circ \mathbb{L} \hat{Z}_2 \lsem -1 \rsem \circ \hat\epsilon_1\circ \hat{\zeta}_n
\end{eqnarray*}
where we denote by $\hat\epsilon_i$, the standard lift of the inclusion functor compatible with \eqref{defcupcap}.\\

The following theorem implies that these functors provide a functorial action of the Temperley-Lieb category.
\begin{theorem}
\label{cupcapiso}
Let $ j \geq k$.  There are isomorphisms
\vspace{-2mm}
\begin{multicols}{2}
\begin{enumerate}
\item $ \hat{\cap}_{i+1,n+2} \hat{\cup}_{i,n} \cong \hat{\Id} $ 
\item $ \hat{\cap}_{i,n+2}
    \hat{\cup}_{i+1,n} \cong \hat{\Id} $ 
    \item $ \hat{\cap}_{j,n} \hat{\cap}_{i,n+2} \cong
    \hat{\cap}_{i,n} \hat{\cap}_{j+2,n+2}$ 
    \item $ \hat{\cap}_{i, n+2} \hat{\cup}_{i,n} \cong
    \hat{\Id}\lsem1\rsem\langle 1 \rangle \bigoplus \hat{\Id}\lsem-1\rsem\langle -1 \rangle $
    \item $\hat{\cup}_{j,n-2} \hat{\cap}_{i,n} \cong
    \hat{\cap}_{i,n+2} \hat{\cup}_{j+2,n} $ \item $ \hat{\cup}_{i,n-2} \hat{\cap}_{j,n} \cong
    \hat{\cap}_{j+2,n+2} \hat{\cup}_{i,n} $ \item $ \hat{\cup}_{i,n+2} \hat{\cup}_{j,n} \cong
    \hat{\cup}_{j+2,n+2} \hat{\cup}_{i,n} $ 
    \\
\end{enumerate}
\end{multicols}
\vspace{-2mm}
of graded endofunctors of $\oplus_{k=0}^n \gmod-A_{k,1^n}$. 
In the Grothendieck group, $ [\hat{\cap}_{i,n}]=\cap_{i,n} $ and $ [\hat{\cup}_{i,n}]=\cup_{i,n} $.
\end{theorem}

\begin{proof}
The first part was proven in the Koszul dual case in \cite[Theorem 6.2]{StrDuke}. The theorem now holds for the functors defined above by \cite[Section 6.4]{MOS}. The second part follows directly from \cite[Proposition 15]{BFK} and Lemma \ref{ZonVermas} below.
\end{proof}

The following result categorifies \eqref{defcupcap}, and is the main tool in computing the cap functors explicitly.
\begin{lemma}
\label{ZonVermas}
Let $\hat M({\bf a})\in \gmod-A_{k,1^n}$ be the standard graded lift of the Verma module $M({\bf a})\in\cO_k$. Let ${\bf b}$ be the sequence ${\bf a}$ with $a_i$ and $a_{i+1}$ removed. Then there are isomorphisms of graded modules
\begin{eqnarray}
\hat{\cap}_i\hat M(\bf a)
&\cong&
\begin{cases}
0&\text{if $a_i=a_{i+1}$}\\
\hat M({\bf b})\langle-1\rangle\lsem-1\rsem\in \gmod-A_{k-1,1^{n-2}} &\text{if $a_i=1$, $a_{i+1}=0$}\\
\hat M({\bf b})\in \gmod-A_{k-1,1^{n-2}} &\text{if $a_i=0$, $a_{i+1}=1$}
\end{cases}
\end{eqnarray}
whereas $\hat\epsilon_i(\hat M({\bf b}))$ is quasi-isomorphic to a complex of the form
$$\cdots\quad0\longrightarrow\hat M({\bf c})\langle 1\rangle\longrightarrow\hat M({\bf d})\longrightarrow0\quad\cdots,$$
where ${\bf c}$ and ${\bf d}$ are obtained from ${\bf b}$ by inserting $01$ respectively $10$ at places $i$, $i+1$.
\end{lemma}
\begin{proof}
This follows directly from \cite[Theorem 8.2, Theorem 5.3]{Strgrad} and Koszul duality \cite[Theorem 35]{MOS}.
\end{proof}

To a coloured cap or cup we assign the functors graphically depicted in Figure~\ref{coloured}.

\subsection{Twisting functors and (coloured) crossings}
Fix an element $ w $ in the Weyl group of $ \mathbb{S}_n. $  Let $ \mathfrak{n}_w = \mathfrak{n}^- \cap w^{-1}(\mathfrak{n}^+). $ Denote by $ N_w $
the universal enveloping algebra $ \mathcal{U}(\mathfrak{n}_w). $ There is a natural $ \mathbb{Z}$-grading on $ \mathfrak{gl}_n $ where
the degree one part is the direct sum of root spaces for simple roots. This induces a $ \mathbb{Z}$-grading on $ N_w. $

Let $ \Gamma_w $ be an automorphism of $ \mathfrak{gl}_n $ corresponding to the element $w$.  Define the semiregular bimodule $ S_w =
\mathcal{U}(\mathfrak{gl}_n)\otimes_{N_w} N_w^*, $ where $ N_w^* $ is the graded dual of $ N_w. $  It is a nontrivial fact that $ S_w $
is a $ (\mathcal{U}(\mathfrak{gl}_n),\mathcal{U}(\mathfrak{gl}_n))$-bimodule and the functor of tensoring with this bimodule does not depend (up to isomorphism) on the choice of $\Gamma_w$. An algebraic proof first appeared in \cite{Ark}, see also \cite{AL}.

\begin{define}
The {\it twisting functor} $ T_w \colon \mathcal{U}(\mathfrak{gl}_n)-\Mod \rightarrow \mathcal{U}(\mathfrak{gl}_n)-\Mod $ is defined by
$ T_w(M)=S_w \otimes_{\mathcal{U}(\mathfrak{gl}_n)} M $ and the action of $ \mathfrak{gl}_n $ is twisted by $ \Gamma_w. $
\end{define}
The functor is right exact. It is well-known (\cite{AL}) that $T_w $ restricts to an endofunctor on $
\mathcal{O}_{\lambda}(\mathfrak{gl}_n)$. For a detailed study of the properties of these functors we also refer to \cite{AS}.
These functors satisfy the braid relations, so we are particularly interested in the functors $ T_w $ for $w=\sigma_i$, the simple
reflections generating $\mathbb{S}_n$.  Denote $T_{s_i}$ by $T_i$. By \cite[Section 5.1]{FKS}, we know that the functor $T_i$ has a graded lift $
\hat{T}_i \colon \gmod-A_{k,n} \rightarrow \gmod-A_{k,n}$  unique up to isomorphism and a shift in the grading. We choose this lift
such that
\begin{eqnarray}
\label{liftofark}
\hat{T}_i \hat{M}(\lambda) \cong
\begin{cases}
 \hat{M}(s_i.\lambda) \langle -1 \rangle &\text{if $s_i \in W^{\lambda}$}\\
 \hat{M}(s_i.\lambda) \langle -2 \rangle  &\text{otherwise.}
 \end{cases}
\end{eqnarray}

\begin{remark}{\rm 
This lift differs from the one in \cite[Section 5]{FKS} by an overall shift.}
\end{remark}
We will need the following important relationship between the twisting functor $ T_i $ and the Zuckerman functor $ Z_i $:
\begin{theorem}{\rm (\cite{KhoMaz}, \cite[Proposition 5.4]{AS})}
\label{KMmap} There is a natural transformation $ \tau_i \colon T_i \rightarrow \Id $ such that for any object $ M, $ the cokernel of the
map $ (\tau_i)_M \colon T_i M \rightarrow M $ is $ Z_i M. $
\end{theorem}

\begin{corollary}{\rm (\cite[Lemma 41]{MOS})}
\label{disttri}
There is a distinguished triangle of derived functors
$$ L \hat{T}_i \rightarrow \hat{\Id}\langle -2 \rangle \rightarrow \hat{\epsilon}_i \hat{Z}_i \langle -1
\rangle. $$
\end{corollary}

\begin{remark}{\rm 
The shifts in the Corollary ~\ref{disttri} differ from those in \cite{MOS} since the graded lifts of the functors differ from the
conventions in \cite{MOS}.
Also note that the internal grading convention in ~\cite{MOS} is opposite to the one used here.}
\end{remark}
A right adjoint to $T_w$ was studied in \cite{AS} and is known to coincide with Joseph's completion functor by \cite{KhoMaz}. Denote this functor by $J_w$ and its graded version by $\hat{J}_w$. Again we write $J_i $ and $\hat{J}_i$ in case $w=s_i$.
By \cite{AS}, the graded lift $ \hat{J}_i $ satisfies
\begin{eqnarray}
\label{liftofjos}
\hat{J}_i \hat{\nabla}(\lambda) \cong
\begin{cases}
 \hat{\nabla}(s_i.\lambda) \langle 1 \rangle &\text{if $s_i \in W^{\lambda}$}\\
 \hat{\nabla}(s_i.\lambda) \langle 2 \rangle  &\text{otherwise}
 \end{cases}
\end{eqnarray}

\begin{lemma}
\label{distrijos}
There is a distinguished triangle relating $ R\hat{J}_i $ and $ L\hat{Z}_i $ as in Corollary ~\ref{disttri}:
$$ \hat{\epsilon}_i \mathbb{L}\hat{Z}_i \lsem -2 \rsem \langle 1 \rangle \rightarrow \hat{\Id} \langle 2 \rangle \rightarrow
\mathbb{R}\hat{J}_i.$$
\end{lemma}

\begin{proof}
By ~\cite[Theorem 35]{MOS}, a graded lift of the translation functor $ \hat{\theta}_i $ and the functor $ \hat{\epsilon}_i \mathbb{L} \hat{Z}_i \lsem -1 \rsem $ are Koszul dual.
There is a distinguished triangle of functors
\begin{eqnarray}
\label{districoshuffling}
\mathbb{R}\hat{D}_i \rightarrow \hat{\theta}_i \rightarrow \Id \langle -1 \rangle
\end{eqnarray}
where $ \hat{D}_i $ is the coshuffling functor, which is defined as the kernel of the canonical morphism $ \hat{\theta}_i \rightarrow \Id \langle -1 \rangle $.  See \cite{MS2} for more details.
By ~\cite[Theorem 39]{MOS}, the Koszul dual of $ \mathbb{R}\hat{D}_i $ is the functor $ \mathbb{R}\hat{J}_i \langle -1 \rangle $.
After applying the Koszul duality functor to the distinguished triangle in ~\eqref{districoshuffling}, rotation and internal grading shift of the resulting triangle gives the lemma.
\end{proof}

These distinguished triangles play an important role in proving the following theorem which appears in \cite{StrDuke} in
the Koszul dual case. It follows from there directly by applying the results from \cite[Sections 6.4 and 6.5]{MOS}.

\begin{theorem} [Reidemeister moves]
\label{functorisos} There are isomorphism of functors:
   \vspace{-2mm}
\begin{multicols}{2}
\begin{enumerate}
\item $ \hat{\cap}_{i,n} \circ \mathbb{L}\hat{T}_i \lsem -1\rsem \cong \hat{\cap}_{i,n} $
 \item $ \hat{\cap}_{i,n} \circ
    \mathbb{R}\hat{J}_i \lsem 1\rsem \cong \hat{\cap}_{i,n} $
    \item $ \mathbb{L} \hat{T}_i \circ \mathbb{R} \hat{J}_i \cong \hat{\Id} \cong
    \mathbb{R} \hat{J}_i \circ \mathbb{L} \hat{T}_i $
      \item $ \mathbb{L} \hat{T}_i \circ \mathbb{L} \hat{T}_j \cong \mathbb{L} \hat{T}_j \circ \mathbb{L}\hat{T}_i, $ if $ |i-j| \geq 2 $
    \item $ \mathbb{R} \hat{J}_i \circ \mathbb{R} \hat{J}_j \cong \mathbb{R} \hat{J}_j \circ \mathbb{R} \hat{J}_i, $ if $ |i-j| \geq 2 $
    \item $ \mathbb{L} \hat{T}_i \circ \mathbb{R} \hat{J}_j \cong \mathbb{R} \hat{J}_j \circ \mathbb{L} \hat{T}_i, $ if $ |i-j| \geq 2 $
   \end{enumerate} 
   \end{multicols} 
   \vspace{-5mm}
    \begin{enumerate}
      \setcounter{enumi}{6}
    \item $ \mathbb{L} \hat{T}_{i} \circ \mathbb{L} \hat{T}_{i+1} \circ \mathbb{L} \hat{T}_{i}
    \cong \mathbb{L} \hat{T}_{i+1} \circ \mathbb{L}\hat{T}_i \circ \mathbb{L}\hat{T}_{i+1} $
    \item $ \mathbb{R} \hat{J}_i \circ
    \mathbb{R} \hat{J}_{i+1} \circ \mathbb{R} \hat{J}_i \cong \mathbb{R} \hat{J}_{i+1} \circ \mathbb{R} \hat{J}_i \circ \mathbb{R}\hat{J}_{i+1} $.
\end{enumerate}
\end{theorem}
\vspace{-2mm}

To the crossing appearing in the left of \eqref{crosses} we assign the functor $ \mathbb{L}\hat{T}_i \lsem -1 \rsem. $ To the crossing appearing in
the right of \eqref{crosses} we assign the functor $ \mathbb{R} \hat{J}_i \lsem 1 \rsem. $ To the cup appearing in \eqref{fig:capcup} we assign the functor
$ \hat{\cup}_{i,n}. $ To the cap appearing in \eqref{fig:capcup} we assign the functor $ \hat{\cap}_{i,n}. $ Thus to any
unoriented tangle diagram D, from $n$ points to $m$ points, we may associate a functor
$$ \hat{\Phi}(D) \colon D^b(\bigoplus_{k=0}^n \gmod-A_{k,1^n}) \rightarrow D^b(\bigoplus_{k=0}^m \gmod-A_{k,1^m}). $$

By counting the number of
positive and negative crossings and including an overall shift into the functors, we get Theorem ~\ref{catjones}, giving a functor valued invariant
of oriented tangle which categorifies the invariant from Theorem ~\ref{jonesinvariant}.  The proof of the Reidemeister moves,  follows
directly from Theorem ~\ref{functorisos}.

\begin{remark}{\rm 
On the Grothendieck group level we have the skein relations:
\begin{eqnarray*}
[\mathbb{L} \hat{T}_k \lsem -1\rsem ]&=&-q^{-1}[\hat{\cup}_{k-1,n-2} \circ \hat{\cap}_{k,n}]-q^{-2}[\hat{\Id}]\\
\left[\mathbb{R} \hat{J}_k \lsem 1\rsem \right]&=&-q^2[\hat{\Id}] -q[\hat{\cup}_{k-1,n-2} \circ \hat{\cap}_{k,n}].
\end{eqnarray*}
Thus we have a functor valued invariant of tangles which categorifies the Jones polynomial. Note that the categorification of the Jones
polynomial by Khovanov \cite{KhovJones} can be deduced by restricting the functors from the previous theorem to a certain subcategory which is
equivalent to Khovanov's category of graded $H_n$-modules (this follows from \cite[Theorem 5.8.1]{StrSpringer}, \cite[Theorem 1.1 and
Theorem 1.2]{BS3}).}
\end{remark}

\section{Slide moves}
\label{twistprojpres} 

\subsection{Height moves of projection, inclusion and projectors}
Let ${\bf d}$ be a composition of $n$ with the corresponding Young subgroup $\mathbb{S}_{\bf d}\subset \mathbb{S}_n$.  Fix $k$ with $0\leq k\leq n$ and let $F={}_k {\pi}_{\bf d}$ . 
On the abelian category, $F$ is given by a Serre quotient functor
\[
F \; \colon \;  \mathcal{O}_{k}(\mathfrak{gl}_n) \rightarrow
\mathcal{O}_{k}(\mathfrak{gl}_n) / \mathcal{S} .
\]
Recall that $\lambda= e_1+\cdots+e_k-\rho_n$.
Let $L(w \cdot \lambda)$ for $w \in J$, be the simple objects in $\mathcal{S}$.  These are the objects annihilated by the functor $F$.
Let $L(x \cdot \lambda)$ for $x \in I$, be the simple objects not annihilated by the functor $F$.
Note that $x $ is a longest coset representative in $  \mathbb{S}_{\bf d} \backslash \mathbb{S}_n / \mathbb{S}_k \times \mathbb{S}_{n-k} $. Fix $i$, $1\leq i\leq n-1$ such that $ss_i=s_is$ for any $s\in\mathbb{S}_{\bf d}$. 
Let $G:=\epsilon_i \circ \mathbb{L} {Z}_i \lsem -1 \rsem$.
\begin{lemma} 
There exists a triangulated functor $\bar{G}$ making the following diagram 
 \begin{equation}
\label{FGG'diag}
\xymatrix{
\mathcal{D}^<(\mathcal{O}_k(\mathfrak{gl}_n)) \ar[r]^{\hspace{0in} G} \ar[d]^{F}
&
\mathcal{D}^<(\mathcal{O}_k(\mathfrak{gl}_n))  \ar[d]^{{F}}\\
\mathcal{D}^<(\mathcal{O}_k(\mathfrak{gl}_n)/\mathcal{S}) \ar[r]^{\hspace{0in} \bar{G}} 
&
\mathcal{D}^<(\mathcal{O}_k(\mathfrak{gl}_n) / \mathcal{S})  \ .
} 
\end{equation}
commute.  Moreover the straightforward graded version also holds, with $F$ replaced by $\hat{F} ={}_k \hat{\pi}_{\bf d}$ and $D^<$ replaced by $D^\triangledown$ of the corresponding graded category. 
\end{lemma}

\begin{proof}
Let $\mathcal{D}^<(\mathcal{O}_k(\mathfrak{gl}_n))_{\mathcal{S}} $ denote the full category of $\mathcal{D}^<(\mathcal{O}_k(\mathfrak{gl}_n)) $
generated by complexes with cohomology in $\mathcal{S}$.
Let $P(w\cdot \lambda)$ denote the projective cover of the corresponding simple module $L(w \cdot \lambda)$ in $\mathcal{S}$.
We would like to show the existence of the dashed arrow in \eqref{FGG'diag2} which we will define to be $\bar{G}$. It is enough to show that if $F(X)=0$, then $FG(X)=0$.  Assume $F(X)=0$.  This means that
\[
\Hom(\oplus_{x \in I} P(x \cdot \lambda), X\lsem j \rsem) =0 
\]
for all $j \in \mathbb{Z}$.  
Note that this implies that $X \in \mathcal{D}^<(\mathcal{O}_k(\mathfrak{gl}_n))_{\mathcal{S}}$.
We would like to show $\Hom(\oplus_{x \in I} P(x \cdot \lambda), GX \lsem j \rsem)=0$.

 \begin{equation}
\label{FGG'diag2}
\xymatrix{
\mathcal{D}^<(\mathcal{O}_k(\mathfrak{gl}_n))_{\mathcal{S}} \ar[d]^{}
&
\mathcal{D}^<(\mathcal{O}_k(\mathfrak{gl}_n))_{\mathcal{S}}   \ar[d]^{}\\
\mathcal{D}^<(\mathcal{O}_k(\mathfrak{gl}_n)) \ar[r]^{\hspace{0in} G} \ar[d]^{F}
&
\mathcal{D}^<(\mathcal{O}_k(\mathfrak{gl}_n))  \ar[d]^{{F}}\\
\mathcal{D}^<(\mathcal{O}_k(\mathfrak{gl}_n)/\mathcal{S}) \ar@{-->}[r]^{\hspace{0in} \bar{G} } 
&
\mathcal{D}^<(\mathcal{O}_k(\mathfrak{gl}_n) / \mathcal{S})   
} 
\end{equation}

Since $G$ is self-adjoint, we have
\begin{equation} \label{Gadjointeq}
\Hom(\oplus_{x \in I} P(x \cdot \lambda), GX \lsem j \rsem) =
\Hom(\oplus_{x \in I} G P(x \cdot \lambda), X \lsem j \rsem) .
\end{equation}
Clearly $G P(x \cdot \lambda)$ is quasi-isomorphic to a complex of projective objects and
it is enough to show that the projectives which appear are of the form $P(y \cdot \lambda)$ where $y \in I$.
Note that $y \in I$, if and only if $\epsilon_j \circ \mathbb{L} {Z}_jP(y \cdot \lambda)=0$ for all $j$ such that $s_j \in \mathbb{S}_{\bf d}$.

By assumption, we have for all $s_j \in \mathbb{S}_{\bf d}$ that  $s_j s_i = s_i s_j$. Thus
$(\epsilon_j \circ \mathbb{L} {Z}_j ) G \cong G (\epsilon_j \circ \mathbb{L} {Z}_j) $. 
In particular,
$(\epsilon_j \circ \mathbb{L} {Z}_j ) G P(x \cdot \lambda) = 0$.
Thus $GP(x \cdot \lambda)$ is quasi-isomorphic to a complex of projectives involving only objects of the form $P(y \cdot \lambda)$ where $y \in I$.
As a consequence, the spaces in \eqref{Gadjointeq} vanish, and thus the functor $\bar{G}$ exists.  The commutativity of \eqref{FGG'diag} follows immediately.
The construction obviously lifts to the graded setting. In this case we can replace $D^<$ by $D^\triangledown$.
\end{proof}

\subsection{Twisting functors on $ \mathcal{O}_{k, {\bf d}}(\mathfrak{gl}_n)$}
For this subsection, fix $ {\bf d}=(d_1, \ldots, d_r) $ with $ d_c = p, d_{c+1}=q $ and $ {\bf d'} = (d_1, \ldots, d_{c+1}, d_c,
\ldots, d_r). $ Let $ b = d_1 + \cdots + d_{c-1}. $ Let
$$ w_{p,q} = (\sigma_{b+q} \sigma_{b+q+1} \cdots \sigma_{b+p+q-1}) \cdots (\sigma_{b+2} \cdots \sigma_{b+p+1})(\sigma_{b+1} \cdots
\sigma_{b+p}). $$

Now we fix several Lie algebras.  Let
\begin{align}
\label{auxilaryalgebras}
\mathfrak{g} &= \mathfrak{gl}_{d_1+ \cdots + d_r}\\
\mathfrak{a}_1 &= \mathfrak{gl}_{d_1} \oplus \cdots \oplus
\mathfrak{gl}_{d_c+d_{c+1}} \oplus \cdots \oplus \mathfrak{gl}_{d_r}\\
\mathfrak{a}_2 &= \mathfrak{gl}_{d_1} \oplus \cdots \oplus \mathfrak{gl}_{d_c} \oplus
\mathfrak{gl}_{d_{c+1}} \oplus \cdots \oplus \mathfrak{gl}_{d_r}\\
\mathfrak{a}_3 &= \mathfrak{gl}_{d_1} \oplus \cdots \oplus \mathfrak{gl}_{d_{c+1}} \oplus
\mathfrak{gl}_{d_c} \oplus \cdots \oplus \mathfrak{gl}_{d_r}.
\end{align}

The goal of this subsection is to prove that $ T_{w_{p,q}} $ is a functor from $ \mathcal{O}_{k, {\bf d}}(\mathfrak{gl}_n) $ to $
\mathcal{O}_{k, {\bf d'}}(\mathfrak{gl}_n). $

\begin{lemma}
\label{injlemma} All injective modules in the block $ \mathcal{O}_k(\mathfrak{gl}_n) $ are objects in the category $ \mathcal{O}_{k, {\bf
d}}(\mathfrak{gl}_n). $
\end{lemma}

\begin{proof}
By \cite[Proposition 1.6]{StrTQFT}, any projective object $ P $ in $ \mathcal{O}_k(\mathfrak{gl}_n) $ has a copresentation $ 0 \rightarrow
P \rightarrow P(0^{n-k} 1^{k})^{\bigoplus Q} \rightarrow P(0^{n-k} 1^{k})^{\bigoplus R}. $  The module $P(0^{n-k} 1^k) $ is clearly in the subcategory $
\mathcal{O}_{k, {\bf d}}(\mathfrak{gl}_n). $  Since $ P(0^{n-k} 1^{k}) $ is self dual, for any injective object $ I $ in $
\mathcal{O}_k(\mathfrak{gl}_n), $ there is a short exact sequence $ P(0^{n-k} 1^{k})^{\bigoplus R} \rightarrow P(0^{n-k} 1^{k})^{
\bigoplus Q} \rightarrow I \rightarrow 0. $ Thus any injective $ I $ is in the projectively presented subcategory $ \mathcal{O}_{k, {\bf
d}} (\mathfrak{gl}_n). $
\end{proof}

As a consequence, we obtain the following result.

\begin{corollary}
Any complex $ A $ in $ D^{<}(\mathcal{O}_{k, {\bf d}}(\mathfrak{gl}_n)), $ (or in $ D^{<}({}_{k}\mathcal{H}_{{\bf d}}^1(\mathfrak{gl}_n))
$), is quasi-isomorphic to a complex $A'$ in $ D^{<}(\mathcal{O}_{k, {\bf d}}(\mathfrak{gl}_n)), $ (or in $
D^{<}({}_{k}\mathcal{H}_{{\bf d}}^1(\mathfrak{gl}_n)) $), where $ A' $ is a (possibly infinite) complex of projectives or a finite complex of injectives.
\end{corollary}

In order to understand how the twisting functor acts on projective objects in $ \mathcal{O}_{k, {\bf d}}(\mathfrak{gl}_n) $, it is easier
to study its action on standard objects.  The key connection between standard and projective objects is the following lemma.

\begin{lemma} \cite[Theorem 2.16]{MS2}\label{propstrat}
Any projective object $ P \in \mathcal{O}_{k,{\bf d}}(\mathfrak{gl}_n) $ has a filtration with subquotients
$ \Delta(k_1, d_1 | \cdots | k_r, d_r) $ where $ k_1 + \cdots + k_r = k$.
\end{lemma}

The twisting functor is a composition of tensoring with a bimodule and twisting the action of the Lie algebra by an automorphism defined earlier.  We describe now such an automorphism explicitly.

\begin{lemma}
\label{automorphismlemma} The automorphism $ \Gamma_{w_{p,q}} $ can be chosen to act on $ \mathfrak{gl}_p \oplus \mathfrak{gl}_q \subset \mathfrak{gl}_{p+q}
$ as follows: $ \Gamma_{w_{p,q}}(e_{i,j}) = e_{i+q,j+q} $ if $ i,j \leq p $ and is equal to $ e_{i-p,j-p} $ if $ i,j \geq p. $
\end{lemma}

\begin{proof}
We choose the automorphism $ \Gamma_w (g) = wgw^{-1} $ where $ w $ is the matrix obtained from the identity by permuting the rows.  The $ i\text{th} $
column of $ w $ has a non-zero entry only at the row $ w(i). $  Thus $ w e_{i,j} = e_{w(i),i} e_{i,j} = e_{w(i),j}. $ The $ j\text{th} $
row of $ w^{-1} $ has a non-zero entry only at the $ r \text{th} $ column such that $ w^{-1}(r)=j. $  Therefore $ r=w(j). $  Thus $
e_{w(i),j} w^{-1} = e_{w(i),j} e_{j,w(j)} = e_{w(i),w(j)}. $ Now the lemma follows when $ w=w_{p,q}. $
\end{proof}

We now aim to give an explicit description of the bimodule which defines the twisting functor.  We recall its definition in terms of localization with respect to a certain set of root vectors which satisfies
the Ore condition.  For the moment, take $ w=s_i $ a simple reflection.  Let $ f_i $ be the basis vector for the one-dimensional subalgebra $ \mathfrak{n}_{s_i}$.
Define $ S_i'' = \mathcal{U}(\mathfrak{gl}_n) \otimes_{\mathbb{C}[f_i]} \mathbb{C}[f_i, f_i^{-1}] $ to be the localization of the enveloping algebra with respect to the set generated by $ f_i$.
This is naturally a $ (\mathcal{U}(\mathfrak{gl}_n), \mathcal{U}(\mathfrak{gl}_n)) $-bimodule which contains $ \mathcal{U}(\mathfrak{gl}_n) $ as a subbimodule.
Set $ S_i' = S_i'' / \mathcal{U}(\mathfrak{gl}_n)$.  This is in fact precisely how Arkhipov endows $ S_i $ with the structure of a bimodule so we record the following non-trivial lemma.

\begin{lemma} {\rm(\cite[Corollary 2.1.4]{Ark})}
\label{arklemma}
There is an isomorphism of $ (\mathcal{U}(\mathfrak{gl}_n), \mathcal{U}(\mathfrak{gl}_n))$-bimodules: $ S_i \cong S_i'$.
\end{lemma}

Let $ B_1 = z_1, \ldots, z_{\beta} $ be a basis for $ \mathfrak{n}_{w_{p,q}} $ and $ B_2= y_1, \ldots, y_{\alpha} $ a basis for $ \mathfrak{g}/\mathfrak{n}_{w_{p,q}} $.
Define $ S_{w_{p,q}}'' $ to be the localization of $ \mathcal{U}(\mathfrak{gl}_n) $ with respect to $ B_1$.
$$ S_{w_{p,q}}'' = \mathcal{U}(\mathfrak{gl}_n) \otimes_{\mathbb{C}[B_1]} \mathbb{C}[B_1, B_1^{-1}]. $$
As before, define the bimodule $ S_{w_{p,q}}' = S_{w_{p,q}}'' / \mathcal{U}(\mathfrak{gl}_n)$.

\begin{lemma}
\label{vor}
Let $ \lbrace k_1, \ldots, k_{\beta} \rbrace \subset \mathbb{Z}_{\geq 0} $ and $ \lbrace l_1, \ldots, l_{\alpha} \rbrace \subset \mathbb{Z}_{>0}$.
Then the monomials $ y_{a_1}^{k_1} \cdots y_{a_{\gamma}}^{k_{\gamma}} \otimes y_{b_1}^{-l_1} \cdots y_{b_{\delta}}^{-l_{\delta}}$ form a basis for $ S_{w_{p,q}}' $ over $ \mathbb{C} $.
\end{lemma}

\begin{proof}
The linear independence of the elements follows directly from ~\cite[Lemma 13]{KhoMaz}.  The fact that these elements also span the bimodule follows from a simplification of the proof of ~\cite[Lemma 13]{KhoMaz}
since $ \mathfrak{n}_{w_{p,q}} $ is commutative.  Thus the rearrangement of terms given in the aforementioned proof is trivial in this case.
\end{proof}

\begin{remark}{\rm 
For a similar statement, see \cite[Theorem 3.1]{Vor}.}
\end{remark}

\begin{corollary}
\label{bimodiso}
Suppose $ w_{p,q}=w_1 \cdots w_l $ is a reduced expression for $ w $ in terms of simple reflections.  Then
$ S_{w_{p,q}}' \cong S_{w_{p,q}} $ as $ (\mathcal{U}(\mathfrak{gl}_n), \mathcal{U}(\mathfrak{gl}_n)) $-bimodules.
\end{corollary}

\begin{proof}
Since $ S_{w_{p,q}} $ is independent of a reduced expression for $ {w_{p,q}} $,
$$ S_{w_{p,q}} \cong S_{w_1} \otimes_{\mathcal{U}(\mathfrak{gl}_n)} \cdots \otimes_{\mathcal{U}(\mathfrak{gl}_n)} S_{w_l} \cong
S_{w_1}'  \otimes_{\mathcal{U}(\mathfrak{gl}_n)} \cdots \otimes_{\mathcal{U}(\mathfrak{gl}_n)} S_{w_l}'  $$
by Lemma ~\ref{arklemma}.  The multiplication map and Lemma ~\ref{vor} give that the latter bimodule is isomorphic to $ S_{w_{p,q}}'$.
\end{proof}

\begin{lemma}
\label{finite} As an $ \mathfrak{a}_2 $-module under the adjoint action, $ S_{w_{p,q}} $ is a direct sum of finite-dimensional
submodules.
\end{lemma}

\begin{proof}
We must prove that if $ X \in \mathfrak{a}_2 $,  then the adjoint action of $ X $ on $ S_{w_{p,q}} $ is locally finite.
By Corollary ~\ref{bimodiso}, we consider the action of $ X $ on $S_{w_{p,q}}' $.
As a vector space, $ S_{w_{p,q}}' $ is filtered by subspaces each of which is spanned  by monomials given in
Lemma ~\ref{vor} of a fixed length .
Let $ m_1 =  y_{a_1}^{k_1} \cdots y_{a_{\gamma}}^{k_{\gamma}} $ and $ m_2 = y_{b_1}^{-l_1} \cdots y_{b_{\delta}}^{-l_{\delta}}$.
There is an obvious embedding $ \mathfrak{gl}_{d_c} \oplus \mathfrak{gl}_{d_{c+1}} \subset \mathfrak{a}_2 $.

If $ X \notin \mathfrak{gl}_{d_c} \oplus \mathfrak{gl}_{d_{c+1}} $, then $ X $ and $ m_2 $ commute.
Then $ Xm_1 \otimes m_2 - m_1 \otimes m_2 X = (Xm_1 - m_1 X) \otimes m_2$.
This is then essentially the standard adjoint action of the Lie algebra on an enveloping algebra so the length of $ m_1 $ does not increase.

Now suppose $ X \in \mathfrak{gl}_{d_c} \oplus \mathfrak{gl}_{d_{c+1}} $.  Then
$ y_{b_j}^{-1} X = X y_{b_j}^{-1} $ or
$ y_{b_j}^{-1} X = X y_{b_j}^{-1} + y_{b_{j'}} y_{b_j}^{-2} $
for some other index $ j'$.
Continuing to commute $ X $ to the left of $ m_2 $, we get
$ m_2 X = X m_2 + J $ for some polynomial $ J $ in the generators $ y_{b_j} $ with now positive and negative exponents.  Since a term with a positive exponent
either cancel a term with a negative exponent or kill the monomial, it is clear that the length of $ J$ is less than or equal to the length of $ m_2 $.
Thus $ Xm_1 \otimes m_2 - m_1 \otimes m_2 X = (Xm_1-m_1X) \otimes m_2 + J $.
Once again, the length of $ Xm_1-m_1X $ is less than or equal to the length of $ m_1 $ so the adjoint action of  $ X $ on $ m_1 \otimes m_2 $ does not increase its length.

Thus for all $ X \in \mathfrak{a}_2 $, $ X $ preserves this filtration so
each vector subspace is stable under this adjoint action and is finite-dimensional.
\end{proof}

\begin{lemma}
\label{twiststandardlemma} Let $ \Delta $ be a standard object in $ \mathcal{O}_{k, {\bf
d}}(\mathfrak{gl}_n). $  Then $ T_{w_{p,q}} \Delta $ is an object in $ \mathcal{O}_{k, {\bf d'}}(\mathfrak{gl}_n).  $
\end{lemma}

\begin{proof}
By \cite[Proposition 2.10]{MS2}, it suffices to verify that as an $ \mathfrak{a}_3$-module, $ T_{w_{p,q}} \Delta $ is a
direct sum of projective objects of $ \mathcal{O}_k(\mathfrak{a}_3). $ Due to Lemma ~\ref{automorphismlemma}, it suffices to show that as
an $ \mathfrak{a}_2$-module, $ S_{w_{p,q}} \otimes_{\mathcal{U}(\mathfrak{gl}_n)} \Delta $ is a direct sum of projective
modules in $ \mathcal{O}_k(\mathfrak{a}_2). $
By the definition of the standard module,
\begin{equation*}
\text{Res}_{\mathfrak{g}}^{\mathfrak{a}_2} S_{w_{p,q}} \otimes_{\mathcal{U}(\mathfrak{g})} \Delta \cong
\text{Res}_{\mathfrak{g}}^{\mathfrak{a}_2} S_{w_{p,q}} \otimes_{\mathcal{U}(\mathfrak{g})} (\mathcal{U}(\mathfrak{g})
\otimes_{\mathcal{U}(\mathfrak{p}_{2})} P^{\mathfrak{a}_2})
\cong \text{Res}_{\mathfrak{g}}^{\mathfrak{a}_2} S_{w_{p,q}}
\otimes_{\mathcal{U}(\mathfrak{p}_{{2}})} P^{\mathfrak{a}_2}
\end{equation*}
where $ \mathfrak{p}_{2} $ is the parabolic subalgebra whose reductive part is $ \mathfrak{a}_2 $ and $ P^{\mathfrak{a}_2} $ is an anti-dominant projective object in $ \mathcal{O}(\mathfrak{a}_2)$.  This is a quotient of $ S_{w_{p,q}}^{}
\otimes_{\mathbb{C}} P^{\mathfrak{a}_2}, $ where $ S_{w_{p,q}}^{} $ is an $ \mathfrak{a}_2$-module under the adjoint action.
By Lemma ~\ref{finite}, this module is locally finite under the adjoint action, so as an $ \mathfrak{a}_3$-module, $ T_{w_{p,q}}
\Delta $ is a direct sum of projective objects of $ \mathcal{O}_k(\mathfrak{a}_3). $
\end{proof}

\begin{prop}
\label{twistprop} Let $ P $ be a projective object of $ \mathcal{O}_{k, {\bf d}}(\mathfrak{gl}_n). $  Then $ T_{w_{p,q}} P $ is an object
of $ \mathcal{O}_{k, {\bf d'}}(\mathfrak{gl}_n). $
\end{prop}

\begin{proof}
By Lemma ~\ref{twiststandardlemma}, this functor sends standard objects to objects of $ \mathcal{O}_{k,{\bf d'}}(\mathfrak{gl}_n). $
Standard objects have Verma flags, thus the twist functor is exact on the subcategory of standard objects by \cite[Theorem 2.2]{AS}.
Since projective objects of $ \mathcal{O}_{k, {\bf d}}(\mathfrak{gl}_n) $ have standard flags the twist functor sends a projective object
of $ \mathcal{O}_{k, {\bf d}}(\mathfrak{gl}_n) $ to an object of $ \mathcal{O}_{k, {\bf d'}}(\mathfrak{gl}_n) $ by induction on the length
of the standard flag.
\end{proof}

\begin{prop}
\label{inversetwistprop} Let $I$ be an injective object of $ \mathcal{O}_{k}(\mathfrak{gl}_n). $ Then $J_{w_{p,q}}I$ is an object of
$\mathcal{O}_{k, {\bf d'}}(\mathfrak{gl}_n). $
\end{prop}

\begin{proof}
Let us first look at the corresponding statement for the principal block. Consider the injective object $N:=I(0)=\op{d}M(0)$ in the principal block of $\cO(\mathfrak{gl}_n)$. Then $N':=J_{w_{p,q}}\nabla(0)=\op{d}M(w_{p,q}\cdot0)$ is again a dual Verma module by \eqref{liftofjos}. The projective cover of a dual Verma module is $P(w_0\cdot0)$, hence we have a short exact sequence of the form
\begin{equation}
\label{ses}
K\hookrightarrow P(w_0\cdot0) \surj N',
\end{equation}
for some module $K$. We claim that the projective cover of $K$ has only indecomposable summands of the form $P(y\cdot0)$ with $sy<y$ for any $s\in \mathbb{S}_{{\bf d}'}$. This claim can be verified by using an alternative definition of the twisting functor $T_s$ as partial coapproximation, \cite{KhoMaz}. Namely $T_s$ of the module $N'$ is isomorphic to the largest quotient of the projective cover $P(N')$ of $N'$ which surjects onto $N'$ and the kernel contains only simple modules of the form $L(x.0)$ where $sx>x$. The claim is then equivalent to the statement that $T_sN'=N'$ for any simple reflection $s\in \mathbb{S}_{{\bf d}'}$. Since $sw_{p,q}>w_{p,q}$ for those $s$, the latter claim follows from \cite[Theorem 2.3]{AS}. Since $J_{w_{p,q}}$ commutes with translation functors and hence with shuffling functors, the same statements hold for any injective object $I(w\cdot0)$ and for any dual Verma module $\op{d}M(w\cdot0)$.
Let now $Q:=\op{d}M(\bf{a})\in\cO_k$. Then $Q$ can be obtained from a dual Verma module $N=M(w\cdot0)$ in the principal block by translating to $\cO_k$, in formulas $Q=\theta_{on}N$. Using \eqref{ses} and the exactness of $\theta_{on}$ we get an exact sequence
$$\theta_{on}P(K)\longrightarrow \theta_{on}P(w_0\cdot0) \surj \theta_{on}N'=\theta_{on}J_{w_{p,q}}N.$$
Now $\theta_{on}P(w_0\cdot0)$ is isomorphic to several copies of $P(0^{n-k}1^k)$ and $\theta_{on}P(K)$ is a direct sum of copies of $P(x\cdot\la)\in\cO_k$, where $sx\cdot\la=x\cdot\la$ or $sx<x$. The first statement here is clear, and for the second note that $sy<y$ if and only if $T_sP(y\cdot0)\cong P(y\cdot0)$, \cite[Proposition 5.3 and Corollary 5.2]{AS}. In this case $T_s\theta_{on}P(y\cdot0) \cong \theta_{on}T_sP(y\cdot0)\cong \theta_{on}P(y\cdot0)$ which, by the arguments in the proof of \cite[Corollary 5.2]{AS}, is only possible if all summands occurring are of the form $P(x\cdot\la)\in\cO_k$, $sx\cdot\la=x\cdot\la$ or $sx<x$. Hence, $J_{w_{p,q}}\op{d}M(\bf{a})$ is an object of $\mathcal{O}_{k, {\bf d'}}(\mathfrak{gl}_n)$. Since dual Verma modules are acyclic for $J_{w_{p,q}}$ (\cite[Theorem 4.1, Theorem 2.2, Lemma 2.1 (4)]{AS}) and injective objects have a dual Verma flag, the claim of the proposition follows.
\end{proof}

We now obtain the following important result needed to define categorified coloured crossings.
\begin{corollary}
\label{maintwistingtheorem}
The functors $ \mathbb{L} \hat{T}_{w_{p,q}} $ and $ \mathbb{R} \hat{J}_{w_{p,q}^{-1}} $ restrict to projectively presented subcategories as follows:
\begin{enumerate}
\item $ \mathbb{L} \hat{T}_{w_{p,q}} \colon D^{\triangledown}(\oplus_{k=0}^n \gmod-A_{k,{\bf d}}) \rightarrow D^{\triangledown}(\oplus_{k=0}^n \gmod-A_{k,{\bf d'}}) $,
\item $ \mathbb{R} \hat{J}_{w_{p,q}^{-1}} \colon D^{\triangledown}(\oplus_{k=0}^n \gmod-A_{k,{\bf d'}}) \rightarrow D^{\triangledown}(\oplus_{k=0}^n \gmod-A_{k,{\bf d}})$.
\end{enumerate}
\end{corollary}

\begin{proof}
The ungraded version follows from Propositions ~\ref{twistprop} and ~\ref{inversetwistprop} since we need only to apply these functors to complexes of projective and injective objects respectively.
The graded version follows immediately.
\end{proof}

To the coloured crossings as displayed on the left hand side of Figure \ref{coloured} we associate the compositions of functors as displayed on the right hand side.

\subsection{Sliding past a crossing}
The next result says that we could slide the categorified Jones-Wenzl projector past a categorified crossing when we restrict to the appropriate category.  This is a key result in proving the main theorem in Section ~\ref{reidemeistersection}.  See Figure ~\ref{crossslidefig} for a graphical interpretation.

\begin{corollary}
\label{crossingslide}
There are isomorphisms of functors restricted to the subcategories:
$$ D^{\triangledown}(\oplus_{k=0}^n \gmod-A_{k,{\bf d}}) \rightarrow D^{\triangledown}(\oplus_{k=0}^n \gmod-A_{k,{\bf d'}}),$$
as follows
\begin{eqnarray}
\bigoplus_{k=0}^{|{\bf d}|}  \mathbb{L} \hat{T}_{w_{p,q}} \circ \mathbb{L} {}_k \hat{\iota}_{\bf d} \circ {}_k \hat{\pi}_{\bf d} &\cong&
\bigoplus_{k=0}^{|{\bf d'}|} \mathbb{L} {}_k \hat{\iota}_{\bf d'} \circ {}_k \hat{\pi}_{\bf d'} \circ  \mathbb{L} \hat{T}_{w_{p,q}},\\
\bigoplus_{k=0}^{|{\bf d}|}  \mathbb{R} \hat{J}_{w_{p,q}^{-1}} \circ \mathbb{L} {}_k \hat{\iota}_{\bf d} \circ {}_k \hat{\pi}_{\bf d} &\cong&
\bigoplus_{k=0}^{|{\bf d'}|} \mathbb{L} {}_k \hat{\iota}_{\bf d'} \circ {}_k \hat{\pi}_{\bf d'} \circ  \mathbb{R} \hat{J}_{w_{p,q}^{-1}}.
\end{eqnarray}
\end{corollary}

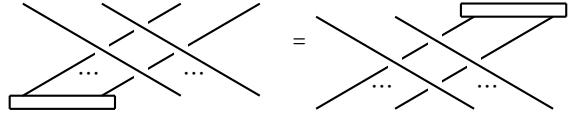
\begin{figure}
  \centering
\begin{equation*}
\begin{tikzpicture}
[scale=0.70]
\draw [shift={(-3,-2)}]  (3,0) -- (5,0)[thick];
\draw [shift={(-3,-2)}] (3,0) -- (3,.25)[thick];
\draw [shift={(-3,-2)}](5,0) -- (5,.25)[thick];
\draw [shift={(-3,-2)}] (3,.25) -- (5,.25)[thick];
\draw  [shift={(-2.5,-1.5)}] (4, .125) node{$ \cdots$};
\draw  [shift={(-.5,-1.5)}] (4, .125) node{$ \cdots$};
\draw (.25,-1.75) -- (3.25,0)[thick];
  \path [fill=white] (1.625,-1) rectangle (1.875,-.75);
\draw (1.75,-1.75) -- (4.75,0)[thick];
  \path [fill=white] (3.125,-1) rectangle (3.375,-.75);
  \path [fill=white] (2.375,-.3125) rectangle (2.625,-.5625);  
    \path [fill=white] (2.375,-1.1875) rectangle (2.625,-1.4375);
\draw (4.75,-1.75) -- (1.75,0)[thick];
\draw (3.25,-1.75) -- (.25,0)[thick];
  \draw  [shift={(3,-1.25)}] (2.5, .5) node{$ =$};
\end{tikzpicture}
\begin{tikzpicture}
[scale=.7]
\draw  [shift={(-2.5,-1.5)}] (4, .125) node{$ \cdots$};
\draw  [shift={(-.5,-1.5)}] (4, .125) node{$ \cdots$};
\draw [shift={(3,0)}](0,0) -- (2,0)[thick];
\draw [shift={(3,0)}](0,0) -- (0,.25)[thick];
\draw [shift={(3,0)}](2,0) -- (2, .25)[thick];
\draw [shift={(3,0)}](0, .25) -- (2, .25)[thick];
\draw (.25,-1.75) -- (3.25,0)[thick];
  \path [fill=white] (1.625,-1) rectangle (1.875,-.75);
\draw (1.75,-1.75) -- (4.75,0)[thick];
  \path [fill=white] (3.125,-1) rectangle (3.375,-.75);
  \path [fill=white] (2.375,-.3125) rectangle (2.625,-.5625);  
    \path [fill=white] (2.375,-1.1875) rectangle (2.625,-1.4375);
\draw (4.75,-1.75) -- (1.75,0)[thick];
\draw (3.25,-1.75) -- (.25,0)[thick];
\end{tikzpicture}
\end{equation*}
\caption{{\small Sliding past a crossing.}}
\label{crossslidefig}
\end{figure}

\begin{proof}
By Theorem ~\ref{maintwistingtheorem}, all of the Jones-Wenzl projection functors are isomorphic to the identity upon restriction to the projectively presented subcategory.
\end{proof}

The following result gives a categorification of the $R$-matrix in $\op{Rep}(\mathcal{U}_q(\mathfrak{sl}_2))$.
\begin{theorem}[Braiding]
\label{braid}
The functor $$ \mathbb{L} \hat{T} _{w_{p,q}} \colon D^{\triangledown}(\gmod-A_{k, {\bf d}}) \rightarrow  D^{\triangledown}(\gmod-A_{k, {\bf d'}}) $$ is an equivalence of categories with inverse functor $ \mathbb{R} \hat{J}_{w_{p,q}^{-1}}.$
\end{theorem}

\begin{proof}
On $ D^{<}(\mathcal{O}_k(\mathfrak{gl}_n)), $ the functors $ \mathbb{L}T_{w_{p,q}} $ and $ \mathbb{R}J_{w_{p,q}^{-1}} $ are inverse
    equivalences of categories. By Corollary ~\ref{maintwistingtheorem}, these functors restrict to
    equivalences on the subcategories. The graded version follows immediately.
\end{proof}

\subsection{Sliding projectors along a cup or cap}
Consider the cabled cap diagram $D_1$ and cup diagram $D_2$ displayed in Figure ~\ref{fig:coloredcapcup}. Let $ n_1 = d_1 + \cdots + d_{i-1}$ be the number of strands to the left of the cap and cup. Then we add $d_i$ nested caps (and cups) and denote $n_2 = n_1+ d_i$. Finally let $ n_3 = n_1+ 2d_i$. The diagrams $D_1^l$  and $D_1^r$ in Figure \ref{fig:capslide} differ from $D_1$ by an extra Jones-Wenzl projection associated with the composition
$ {\bf d_1} = (1^{n_1}, d_i, 1^{|{\bf d}|-n_2}) $ and
$ {\bf d_2} = (1^{n_2}, d_i, 1^{|{\bf d}|-n_3})$ respectively. Similar for the cup diagrams  $D_2^l$ and $D_2^r$ in Figure  \ref{fig:cupslide} in comparison with $D_2$. Let $F(D_i)$, $F(D_i^l)$ and $F(D_i^r)$ for $i=1,2$ be the functors associated with the respective diagrams. The equalities of intertwiners displayed in Figures ~\ref{fig:capslide} and ~\ref{fig:cupslide} lift to isomorphisms of functors as follows.

\begin{figure}[b]
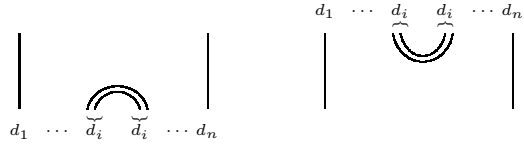

\begin{equation*}
 \xy (0,0);(0,10)*{}**\dir{-};
 (9,0)*{};(17,0)*{}**\crv{(10,4)&(16,4)}; 
  (10,0)*{};(16,0)*{}**\crv{(11,3)&(15,3)}; 
 (25,0)*{};(25,10)*{}**\dir{-};(25,10)*{};
(0,-3)*{\mbox{\tiny $d_1$}};(5,-3)*{\mbox{\tiny $\ldots$}};
(10,-1)*{\mbox{\tiny $\underbrace{}$}};
(10,-3)*{\mbox{\tiny $d_i$}};
(16,-1)*{\mbox{\tiny $\underbrace{}$}};
(16,-3)*{\mbox{\tiny $d_{i}$}};
(21,-3)*{\mbox{\tiny$\ldots$}};
(25,-3)*{\mbox{\tiny$d_n$}};
\endxy
\hspace{0.5in} \xy (0,0)*{};(0,10)*{}**\dir{-};
(10,10)*{};(16,10)*{}**\crv{(11,6)&(15,6)};
(9,10)*{};(17,10)*{}**\crv{(10,5)&(16,5)};
(10,10)*{};(16,10)*{};
(25,0)*{};(25,10)*{}**\dir{-};(25,10)*{}; (0,13)*{\mbox{\tiny $d_1$}};
(5,13)*{\mbox{\tiny $\ldots$}};
(10,13)*{\mbox{\tiny ${d_i}$}};
(10,11)*{\mbox{\tiny $\overbrace{}$}};
(16,13)*{\mbox{\tiny $d_{i}$}};
(16,11)*{\mbox{\tiny $\overbrace{}$}};
(21,13)*{\mbox{\tiny$\ldots$}};
(25,13)*{\mbox{\tiny$d_{n}$}};
\endxy
\end{equation*}
\caption{{\small Diagrams $D_1$ and $D_2$}}
\label{fig:coloredcapcup}
\end{figure}

\begin{figure}
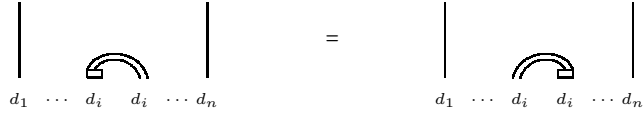

\begin{equation*}
 \xy (0,0);(0,10)*{}**\dir{-};
 (0,10); 
  (9,1)*{};(17,0)*{}**\crv{(10,4)&(16,4)}; 
  (10,1)*{};(16,0)*{}**\crv{(11,3)&(15,3)}; 
  (25,0)*{};(25,10)*{}**\dir{-};(25,10)*{};
(0,-3)*{\mbox{\tiny $d_1$}};(5,-3)*{\mbox{\tiny $\ldots$}};(10,-3)*{\mbox{\tiny $d_i$}};(16,-3)*{\mbox{\tiny
$d_{i}$}};(21,-3)*{\mbox{\tiny$\ldots$}};(25,-3)*{\mbox{\tiny$d_n$}};
(9,1);(11,1)*{}**\dir{-};
(9,0);(11,0)*{}**\dir{-};
(9,0);(9,1)*{}**\dir{-};
(11,0);(11,1)*{}**\dir{-};
\endxy
\hspace{.1in} \xy {(0,0)*{};(12.5,5)*{=}};
\endxy
\hspace{0.5in}
 \xy (0,0);(0,10)*{}**\dir{-};(0,10);
   (9,0)*{};(17,1)*{}**\crv{(10,4)&(16,4)}; 
  (10,0)*{};(16,1)*{}**\crv{(11,3)&(15,3)}; 
  (25,0)*{};(25,10)*{}**\dir{-};(25,10)*{};
(0,-3)*{\mbox{\tiny $d_1$}};(5,-3)*{\mbox{\tiny $\ldots$}};(10,-3)*{\mbox{\tiny $d_i$}};(16,-3)*{\mbox{\tiny
$d_{i}$}};(21,-3)*{\mbox{\tiny$\ldots$}};(25,-3)*{\mbox{\tiny$d_n$}};
(15,1);(17,1)*{}**\dir{-};
(15,0);(17,0)*{}**\dir{-};
(15,0);(15,1)*{}**\dir{-};
(17,0);(17,1)*{}**\dir{-};
\endxy
\end{equation*}
\caption{{\small Diagrams $D_1^l$ and $D_1^r$: Cap slide}}
\label{fig:capslide}
\end{figure}

\begin{figure}
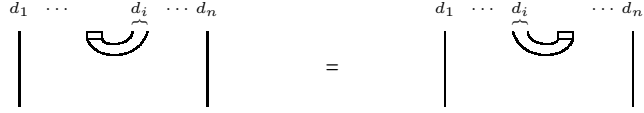

\begin{equation*}
 \xy (0,0)*{};(0,10)*{}**\dir{-};
 (9,9)*{};  (17,10)*{}**\crv{(10,6)&(16,6)};
  (11,9)*{};  (15,10)*{}**\crv{(11,8)&(15,8)};
(25,0)*{};(25,10)*{}**\dir{-};(25,10)*{}; 
(0,13)*{\mbox{\tiny $d_1$}};
(5,13)*{\mbox{\tiny $\ldots$}};
(16,13)*{\mbox{\tiny ${d_i}$}};
(16,11)*{\mbox{\tiny $\overbrace{}$}};
(21,13)*{\mbox{\tiny$\ldots$}};
(25,13)*{\mbox{\tiny$d_{n}$}};
(9,9);(9,10)*{}**\dir{-};
(9,9);(11,9)*{}**\dir{-};
(11,9);(11,10)*{}**\dir{-};
(9,10);(11,10)*{}**\dir{-};
\endxy
\hspace{.1in} \xy {(0,0)*{};(12.5,5)*{=}};
\endxy
\hspace{0.5in} \xy (0,0)*{};(0,10)*{}**\dir{-};
(9,10)*{};(17,9)*{}**\crv{(10,6)&(16,6)};
  (11,10)*{};  (15,9)*{}**\crv{(11,8)&(15,8)};
(25,0)*{};(25,10)*{}**\dir{-};(25,10)*{}; 
(0,13)*{\mbox{\tiny $d_1$}};
(5,13)*{\mbox{\tiny $\ldots$}};
(10,13)*{\mbox{\tiny ${d_i}$}};
(10,11)*{\mbox{\tiny $\overbrace{}$}};
(21,13)*{\mbox{\tiny$\ldots$}};
(25,13)*{\mbox{\tiny$d_{n}$}};
(15,9);(15,10)*{}**\dir{-};
(15,9);(17,9)*{}**\dir{-};
(17,9);(17,10)*{}**\dir{-};
(15,10);(17,10)*{}**\dir{-};
\endxy
\end{equation*}
\caption{{\small Diagrams $D_2^l$ and $D_2^r$: Cup slide}}
\label{fig:cupslide}
\end{figure}


\begin{theorem}[Cup and cap slides]
\label{cupslide}
There are isomorphisms of functors 
\begin{equation}
\label{RumRosinentee}
F(D_1^l)\cong F(D_1^r) \ , \hspace{.25in} F(D_2^l)\cong F(D_2^r) .
\end{equation}
\end{theorem}

\begin{proof}
By definition, \eqref{RumRosinentee} means there are isomorphisms:
\begin{enumerate}
\item
$\bigoplus_{k=0}^{|{\bf d}|} F(D_1) \circ \mathbb{L} ({}_k \hat{\iota}_{\bf d_1}) \circ {}_k \hat{\pi}_{\bf d_1} \cong \bigoplus_{k=0}^{|{\bf d}|} F(D_1) \circ \mathbb{L} ({}_k \hat{\iota}_{\bf d_2}) \circ {}_k \hat{\pi}_{\bf d_2}$,
\item $\bigoplus_{k=0}^{|{\bf d}|} \mathbb{L} ({}_k \hat{\iota}_{\bf d_1}) \circ {}_k \hat{\pi}_{\bf d_1} \circ F(D_2) \cong \bigoplus_{k=0}^{|{\bf d}|} \mathbb{L} ({}_k \hat{\iota}_{\bf d_2}) \circ {}_k \hat{\pi}_{\bf d_2} \circ F(D_2)$.
\end{enumerate}

Using adjointness properties it is enough to establish the second isomorphism. For that note that there is an adjunction morphism from the functor $G'$ attached to the diagram $D_2^{lr}$ which is $D_2$ but with two projectors, one on the left of the cups as in $D_2^l$ and one on the right as in $D_2^r$, to the functor associated with $D_2^l$ and $D_2^r$ respectively. We claim that in each case this is an isomorphism of functors. Composing the two isomorphisms provides then the isomorphism of functors we are looking for. Since the arguments in the two cases are completely analogous, we consider only the situation arising from $D_2^l$, {\it i.e.} the projector is on the left.

Let us first assume that there are no vertical strands in the diagram $D_2$. Then $F$ is a functor from a triangulated subcategory of the (bounded in one direction) derived category of finite-dimensional graded vector spaces. To prove the adjunction is an isomorphism of functors it is enough to check it on the one-dimensional vector space $\mC$ concentrated in homological degree zero. To see that this suffices note that this implies an isomorphism when applied to any bounded complex of graded vector spaces. It holds then also for any complex in the subcategory, since the subcomplex given by fixing an internal degree is always quasi-isomorphic to a bounded complex by definition of the  triangulated subcategory, \cite{AchS}.

Now observe that  $F(D_2)(\mC)$ is a simple object (obtained by applying inclusion functors to a simple object). By \eqref{basissimples} ${}_{d_1}\hat{\pi}_{(d_1, 1^{d_1})}F(D_2)(\mC)$ is isomorphic to the simple object $\hat{L}=\hat{L}(d_1,d_1|0,1|0,1|\cdots |0,1)$ which is the graded lift concentrated in degree zero of $L={L}(d_1,d_1|0,1|0,1|\cdots |0,1) \in\mathcal{O}_{d_1, (d_1,1^{d_1})}$. 

Let $\hat{P}^\bullet=(\cdots\rightarrow \hat{P}_{-2}\rightarrow \hat{P}_{-1}\rightarrow \hat{P}_0)$ be a minimal graded projective resolution of $\hat{L}$ and let it be a graded lift of a minimal projective resolution  ${P}^\bullet=(\cdots\rightarrow {P}_{-2}\rightarrow {P}_{-1}\rightarrow {P}_0)$ of $L$ in $\mathcal{O}_{d_1, (d_1,1^{d_1})}$. If now $Z_iP_j=0$ for any $d_1+1\leq i\leq 2d_1-1$ and $j\leq 0$, then $P_j\in\mathcal{O}_{d_1, (d_1, d_1)}$ for any $j\leq 0$. In particular applying  $\mathbb{L} ({}_{d_1} \hat{\iota}_{(d_1, d_1)}) \circ {}_{d_1} \hat{\pi}_{(d_1,d_1)}$ to $\hat{P}^\bullet$ does not change anything and the claimed isomorphism follows.  

Otherwise, there exist $i, j$ such that $Z_iP_j\not=0$. Pick $I:=\{i_1,i_2,\ldots, i_r\}\subseteq\{d_1+1, \ldots, 2d_1-1\}$ maximal such that there exists $j$ with  $Z_iP_j\not=0$ for any $i\in I$. Let $j_0$ be the maximal $j$ with this property. Note that by assumption $I\not=\emptyset$ and $j_0\not=0$. Consider now $Z_I:=Z_{i_1}\cdots Z_{i_r}$ with standard graded lift $\hat{Z}_I$. By construction $Z_IP_{j_0}\not=0$. To calculate $\mathbb{L}\hat Z_I\hat{L}$ we apply $\hat Z_I$ to $\hat{P}^\bullet$. The result is nonzero, since it has by construction nonzero homology in degree $j_0$. (Note that $\hat Z_I$ applied to the differential does not surject onto $\hat{P}_{j_0}$, since it did not surject in ${P}^\bullet$ by maximality.) But on the other hand, $\mathbb{L}\hat Z_I\hat{L}=0$ by \cite[Theorem~70]{FSS1}, see Remark~\ref{stillneeded}. Thus we have a contradiction and \eqref{RumRosinentee} follows in case the diagram contains no vertical lines. 

Now we add extra lines on the left and right. Again it is enough to show that the adjunction morphism gives an isomorphism when the functors are applied to simple modules. Applying the cup functors from $D_2$ maps a simple module to a simple module. Then the argument goes along the same lines as above.
\end{proof}

\section{Categorification of the coloured Reshetikhin-Turaev invariant}
\label{reidemeistersection}
\subsection{Main theorem}
Let $ E $ be an elementary, oriented, framed tangle diagram from $r$ ordered points to $ s $ ordered points such that each strand is labeled by a
natural number.  This naturally induces colours ${\bf d}=(d_1, \ldots, d_r) $ on the  $r$ points and colours ${\bf e}=(e_1, \ldots, e_s)$ on the $s$ points. We define a functor for the diagram $ E $
$$ \hat{\Phi}_{col}(E) \colon\quad D^{\triangledown}(\bigoplus_{k=0}^{|{\bf d}|} \gmod-A_{k, {\bf d}})
\rightarrow D^{\triangledown}( \bigoplus_{k=0}^{|{\bf e}|} \gmod-A_{k, {\bf e}}) $$ by
$$ \hat{\Phi}_{col}(E) = (\bigoplus_{k=0}^{|{\bf e}|} {}_k \hat{\pi}_{{\bf e}}) \circ
\hat{\Phi}(\op{cab}(D)) \circ (\bigoplus_{k=0}^{|{\bf d}|} \mathbb{L} ({}_k \hat{\iota}_{{\bf d}})) $$ where $ E $ is an oriented cabling of
$ E. $  Then for an arbitrary tangle $ T $ with diagram $ D = E_{\alpha_n} \circ \cdots \circ E_{\alpha_1}, $ define $
\hat{\Phi}_{col}(D)=\hat{\Phi}_{col}(E_{\alpha_n}) \circ \cdots \circ \hat{\Phi}_{col}(E_{\alpha_1}). $

\begin{theorem} 
\label{maintheorem}
Let $ D_1 $ and $ D_2 $ be two diagrams for an oriented, framed, coloured tangle $ T $ from points coloured by $ \bf d
$ to points coloured by $ \bf e$.  Then
$$ \hat {\Phi}_{col}(D_1)\langle 3\gamma(\op{cab}(D_1)) \rangle \cong
\hat{\Phi}_{col}(D_2)\langle 3\gamma(\op{cab}(D_2)) \rangle. $$
\end{theorem}

\begin{proof}
It suffices to show that the coloured Reidemeister moves appearing in ~\eqref{cr1}-\eqref{cr6} from the Appendix hold.
These functors are comprised of cup, cap, crossing, and Bernstein-Gelfand functors. The Bernstein-Gelfand-functors appearing in the interior of the diagrams appear always in pairs, forming categorified Jones-Wenzl projectors.
By Corollary ~\ref{crossingslide} and Theorem ~\ref{cupslide}, we may commute (up to isomorphism) all these categorified Jones-Wenzl projectors to the bottom where they act as identity functors by Theorem \ref{catJW}. Now only cup, cap, and crossing functors labeled by $1$ remain in the interior.  The result then follows from the invariance of the categorified uncoloured Reshetikhin-Turaev invariant in Theorem ~\ref{functorisos}.
\end{proof}
\begin{remark}{\rm 
With the categorification of finite tensor products of arbitrary irreducible finite-dimensional representation for $\mathfrak{sl}_k$ \cite{SaStr} and the results in \cite{MSslk}, our proofs show that Theorem~\ref{maintheorem} holds (with adapted grading shifts) for any $\mathfrak{sl}_k$, $k\geq 2$.}
\end{remark}

As an example, we illustrate the arguments of the proof of Theorem~\ref{maintheorem} more explicitly by explaining the Reidemeister move in \eqref{cr1}, in more detail.  For notational simplicity, we assume that there are no additional strands present. The left hand side of $\eqref{cr1}$ stands for the composition \eqref{cr1detailed} of elementary diagrams.
To each of the elementary diagrams we associated a functor, and let  $F=F_6F_5F_4F_3F_2F_1$ be their composition. Note that the inclusion and projection maps in the middle of the diagram all pair to Jones-Wenzl projectors. In fact, it is always the same projector in our example. Denote by $\hat{p}=\mathbb{L} {}_k \hat{\iota}_{(m,m,m)} \circ {}_k \hat{\pi}_{(m,m,m)}$ its categorification. We have  $F_i=  {}_k \hat{\pi}_{(m,m,m)} G_i\mathbb{L} {}_k \hat{\iota}_{(m,m,m)}$, where $G_i$
denotes the functor associated to the $m$-cabling of the elementary diagram $D_i$. Then we have
$$ F = \bigoplus_{k=0}^m
{}_k \hat{\pi}_{(m)} \circ
G_6 \circ
\hat{p} \circ
G_5 \circ
\hat{p} \circ
G_4 \circ
\hat{p} \circ
G_3 \circ
\hat{p} \circ
G_2 \circ
\hat{p} \circ
G_1 \circ
\mathbb{L} {}_{k} \hat{\iota}_{(m)}. $$
By Corollary ~\ref{crossingslide} and Theorem \ref{JWgenidempotent} we can remove the second and the fourth $\hat{p}$, since we can slide it through the crossing and then use the fact that it is an idempotent.
 \begin{minipage}{0.6\textwidth}
\end{minipage}

Another application of Theorems ~\ref{crossingslide} and ~\ref{cupslide}, gives
$$ F \cong \bigoplus_{k=0}^m
{}_k \hat{\pi}_{(m)} \circ
F_6 \circ
\mathbb{L} {}_k \hat{\iota}_{(m,1, \ldots, 1)} \circ {}_k \hat{\pi}_{(m,1, \ldots, 1)} \circ
F_5 \circ F_4 \circ F_3 \circ F_2 \circ F_1 \circ
\mathbb{L} {}_{k} \hat{\iota}_{(m)}. $$
This in turn is isomorphic to
$$ \bigoplus_{k=0}^m
{}_k \hat{\pi}_{(m)} \circ\mathbb{L} {}_k \hat{\iota}_{(m)} \circ {}_k \hat{\pi}_{(m)} \circ
F_6 \circ
F_5 \circ F_4 \circ F_3 \circ F_2 \circ F_1 \circ
\mathbb{L} {}_{k} \hat{\iota}_{(m)}. $$
Since $ {}_k \hat{\pi}_{(m)} \circ\mathbb{L} {}_k \hat{\iota}_{(m)} \circ {}_k \hat{\pi}_{(m)} \cong  {}_k \hat{\pi}_{(m)} $, we get
$ F \cong \bigoplus_k^m {}_k \hat{\pi}_{(m)} \circ F_6 \circ F_5 \circ F_4 \circ F_3 \circ F_2 \circ F_1 \circ \mathbb{L} {}_{k} \hat{\iota}_{(m)} $.
By Theorem ~\ref{catjones}, this is isomorphic to
$ \bigoplus_{k=0}^m {}_k \hat{\pi}_{(m)} \circ \mathbb{L} {}_{k} \hat{\iota}_{(m)} \cong \hat{\Id} $.

\subsection{Conjectures about the coloured unknot}
While it is easy to calculate the homology of the unknot coloured by the standard two-dimensional representation $V_1$, it is much more challenging to determine the homology of the unknot coloured by $V_n$ for $n>1$.  In the next section we will compute explicitly the homology of the unknot coloured by $V_2$ using the fully stratified structure \cite{BSsemi} of the category of Harish-Chandra bimodules.

Gorsky-Oblomkov-Rasmussen~\cite{GORS} gave a conjecture of the homology of the unknot coloured by $V_n$ coming from the study of rational DAHAs in a setting which is Koszul dual to our construction.

\begin{conjecture}
\cite{GOR}
The homology of the unknot coloured by $V_n$ is isomorphic to the homology of the differential bigraded algebra 
\begin{equation*}
B_n = (\mathbb{C}[u_1, \ldots, u_n] \otimes \Lambda[\zeta_1, \ldots, \zeta_n], d)
\end{equation*}
where 
\begin{equation*}
deg(u_k)=(2-2k,2k) \hspace{.5in}
deg(\zeta_k)=(1-2k,2+2k)
\end{equation*}
and
\begin{equation*}
d(u_k)=0 \hspace{.5in}
d(\zeta_k) = \sum_{i+j=k+1} u_i u_j.
\end{equation*}
\end{conjecture}
Some progress towards a proof of this conjecture has been made by Hogancamp ~\cite{Hog1}.
The homology of the unknot coloured by the three-dimensional representation calculated in \cite{CoKr} confirms the conjecture for $n=2$.  
Later on in Corollary \ref{extcalforunknot}, we reconfirm this conjecture for $n=2$.  
Also see \cite[Proposition 8.4]{Webcombined}.




\section{Examples}
\subsection{Categorified projector on $\mathcal{O}_{1}(\mathfrak{gl}_2)$}
Let $ Q $ denote the quiver in ~\eqref{quiversl2} with vertices $1$ and $2$.
A path (of length $ l>0$) is a sequence $ p = \alpha_1 \alpha_2 \cdots \alpha_l $ of arrows where the starting point of $ \alpha_i $ is the ending point of $ \alpha_{i+1} $ for $ i=1,\ldots,l-1$.
By $ \mathbb{C}Q $ we denote the path algebra of $ Q $. It has basis the set of all paths with additionally $(1)$ and $(2)$ the trivial paths of length $ 0 $ beginning at $ 1 $ and $ 2 $ respectively, and product given by concatenation.  For example, $(2|1|2) = (2|1)(1|2)$ is a basis element of $\mathbb{C}Q$ of degree two.
The path algebra is a graded algebra where the grading comes from the length of each path.
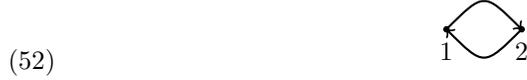
\begin{equation}
\label{quiversl2}
\begin{tikzpicture}
\filldraw[black](0,0) circle (1pt);
\filldraw[black](1,0) circle (1pt);
\draw (0,0) .. controls (.5,.5) .. (1,0)[->][thick];
\draw (0,0) .. controls (.5,-.5) .. (1,0)[<-][thick];
\draw (0,-.3) node{$1$};
\draw (1,-.3) node{$2$};
\end{tikzpicture}
\end{equation}

Set $ A $ to be the algebra $ \mathbb{C}Q $ modulo the two-sided ideal generated by $ (1|2|1) $.  By abuse of notation, we denote the image of an element $ p \in \mathbb{C}Q $ in the algebra $ A $ also by $ p$.
The algebra $ A $ inherits a grading from $ \mathbb{C}Q$ since the relation $ (1|2|1)=0 $ is homogenous.  Let $ A_j $ denote the degree $ j $ subspace of $ A $.  The degree zero part $ A_0 $ is a semi-simple algebra
spanned by $ (1) $ and $ (2)$.  The degree one subspace is spanned by $ (1|2) $ and $ (2|1) $.  The degree two subspace is spanned by $ (2|1|2) $ and $ A_j=0 $ for all $ j \geq 3$.  Let $ A_+ $ be the subspace of $ A $ whose homogenous elements are in positive degree.  The subspace $ A_+ $ is the radical of $ A $.

The graded category $ \mathcal{O}_{1}(\mathfrak{gl}_2) $ is equivalent to the category of finitely-generated, graded, right modules over the algebra $A$.
The projective modules $ (1)A $ and $ (2)A $ correspond to the dominant and anti-dominant projective modules respectively in category $ \mathcal{O}_1(\mathfrak{gl}_2)$.  The simple quotients of the latter two objects correspond to the one-dimensional right $A$-modules $ L(1) = (1)A/A_+ $ and $L(2) = (2)A/(2)A_+$.

Let $C=\End_A((2)A)$ be the endomorphism algebra of the anti-dominant projective module.  It is easy to compute that it's isomorphic to $\mathbb{C}[x]/(x^2)$.
Define functors
\begin{equation*}
\hat{\pi} \colon \gmod-A \rightarrow \gmod-C \hspace{.5in} \hat{\pi}(M)=M \otimes_A A(2)
\end{equation*}
and
\begin{equation*}
\hat{\iota} \colon D^\triangledown(\gmod-C) \rightarrow D^\triangledown(\gmod-A) \hspace{.5in}
\hat{\iota}(M)=M \otimes^{\mathbf{L}}_C (2)A.
\end{equation*}
The categorified Jones-Wenzl projector is then the composite
\begin{equation*}
\hat{p} \colon D^\triangledown(\gmod-A) \rightarrow D^\triangledown(\gmod-A) \hspace{.5in} 
\hat{p} =  \hat{\iota} \circ \hat{\pi}.
\end{equation*}
We now construct an explicit complex of $(A,A)$-bimodules which is quasi-isomorphic to the functor $\hat{p}$.
The first step is to resolve the $(A,C)$-bimodule $A(2)$ as a projective right $C$-module:
\begin{equation}
\label{projresofbigproj}
\cdots \longrightarrow A(2) \otimes C \langle 6 \rangle \stackrel{h_3}{\longrightarrow} A(2) \otimes C \langle 4 \rangle \stackrel{h_2}{\longrightarrow}  A(2) \otimes C \langle 2 \rangle \stackrel{h_1}{\longrightarrow}  A(2) \otimes C \stackrel{h_0}{\longrightarrow} A(2)
\end{equation}
where
\begin{equation*}
h_n((2) \otimes 1)=
\begin{cases}
(2) & \text{ if } n=0 \\
(2) \otimes x - (2|1|2) \otimes 1 & \text{ if } n=1,3,\ldots \\
(2) \otimes x  + (2|1|2) \otimes 1 & \text{ if } n=2,4,\ldots.
\end{cases}
\end{equation*}
Next, tensoring the complex
\begin{equation*}
\cdots \longrightarrow A(2) \otimes C \langle 6 \rangle \stackrel{h_3}{\longrightarrow} A(2) \otimes C \langle 4 \rangle \stackrel{h_2}{\longrightarrow}  A(2) \otimes C \langle 2 \rangle \stackrel{h_1}{\longrightarrow}  A(2) \otimes C 
\end{equation*}
on the right over $C$ with $(2)A$ we get that $\hat{p}$ is quasi-isomorphic to
\begin{equation}
\label{projresofbigproj2}
\cdots \longrightarrow A(2) \otimes (2)A \langle 6 \rangle \stackrel{f_3}{\longrightarrow} A(2) \otimes (2)A \langle 4 \rangle \stackrel{f_2}{\longrightarrow}  A(2) \otimes (2)A \langle 2 \rangle \stackrel{f_1}{\longrightarrow}  A(2) \otimes (2)A
\end{equation}

where
\begin{equation*}
f_n((2) \otimes (2))=
\begin{cases}
(2) \otimes (2|1|2) - (2|1|2) \otimes 1 & \text{ if } n=1,3,\ldots \\
(2) \otimes (2|1|2)  + (2|1|2) \otimes 1 & \text{ if } n=2,4,\ldots.
\end{cases}
\end{equation*}
Noticing that the bimodule $A(2) \otimes (2)A$ is isomorphic to the composition of projective functors
$\hat{\mathcal{E}}_0 \hat{\mathcal{F}}_1$,
the complex in ~\eqref{projresofbigproj2} could be understood as a complex of projective functors
\begin{equation*}
\cdots \longrightarrow \hat{\mathcal{E}}_0 \hat{\mathcal{F}}_1 \langle 6 \rangle \stackrel{f_3}{\longrightarrow} \hat{\mathcal{E}}_0 \hat{\mathcal{F}}_1 \langle 4 \rangle \stackrel{f_2}{\longrightarrow}  \hat{\mathcal{E}}_0 \hat{\mathcal{F}}_1 \langle 2 \rangle \stackrel{f_1}{\longrightarrow}  \hat{\mathcal{E}}_0 \hat{\mathcal{F}}_1.
\end{equation*}

\subsection{Categorified projectors on $\mathcal{O}_{1}(\mathfrak{gl}_3)$}
Let $ Q $ denote the quiver in ~\eqref{quiversl3} with vertices $1$, $2$, and $3$. 
Once again, by $ \mathbb{C}Q $ we denote the path algebra of $ Q $. 
Set $ A $ to be the algebra $ \mathbb{C}Q $ modulo the two-sided ideal generated by $ (1|2|1) $.
\begin{equation}
\label{quiversl3}
\begin{tikzpicture}
\filldraw[black](0,0) circle (1pt);
\filldraw[black](1,0) circle (1pt);
\draw (0,0) .. controls (.5,.5) .. (1,0)[->][thick];
\draw (0,0) .. controls (.5,-.5) .. (1,0)[<-][thick];
\draw (1,0) .. controls (1.5,.5) .. (2,0)[->][thick];
\draw (1,0) .. controls (1.5,-.5) .. (2,0)[<-][thick];
\draw (0,-.3) node{$1$};
\draw (1,-.3) node{$2$};
\draw (2,-.3) node{$3$};
\end{tikzpicture}
\end{equation}
The graded category $ {}^{\mathbb{Z}} \mathcal{O}_{1}(\mathfrak{gl}_3) $ is equivalent to the category of finitely-generated, graded, right modules over the algebra $A$.
The projective modules $ (1)A $ and $ (3)A $ correspond to the dominant and anti-dominant projective modules respectively in category $ \mathcal{O}_1(\mathfrak{gl}_3)$.  The simple quotients of the projective objects correspond to the one-dimensional right $A$-modules $ L(1) = (1)A(1)/(1)A_+ $, $L(2) = (2)A/(2)A_+$, and $L(3) = (3)A/(3)A_+$.

\subsubsection{The functor $\hat{p}_{(3)}$}
Let $C_{(3)}=\End_A((3)A)$ be the endomorphism algebra of the antidominant projective module.  It is easy to compute that it's isomorphic to $\mathbb{C}[x]/(x^3)$, (also see Proposition \ref{EndofAntiDom}).
Define functors
\begin{equation*}
\hat{\pi}_{(3)} \colon \gmod-A \rightarrow \gmod-C_{(3)} \hspace{.5in} \hat{\pi}_{(3)}(M)=M \otimes_A A(3)
\end{equation*}
and
\begin{equation*}
\hat{\iota}_{(3)} \colon D^\triangledown(\gmod-C_{(3)}) \rightarrow D^\triangledown(\gmod-A) \hspace{.5in}
\hat{\iota}_{(3)}(M)=M \otimes^{\mathbf{L}}_{C_{(3)}} (3)A.
\end{equation*}
The functor categorifying the Jones-Wenzl projector $V_1^{\otimes 3} \rightarrow V_3 \rightarrow V_1^{\otimes 3}$  is then the composite
\begin{equation*}
\hat{p}_{(3)} \colon D^\triangledown(\gmod-A) \rightarrow D^\triangledown(\gmod-A) \hspace{.5in} 
\hat{p}_{(3)} =  \hat{\iota}_{(3)} \circ \hat{\pi}_{(3)}.
\end{equation*}

We now construct an explicit complex of $(A,A)$-bimodules which is\\ quasi-isomorphic to the functor $\hat{p}_{(3)}$.
The first step is to resolve the $(A,C_{(3)})$-bimodule $A(3)$ as a projective right $C_{(3)}$-module:
\begin{equation*}
\cdots \longrightarrow A(3) \otimes C_{(3)} \langle 8 \rangle \stackrel{h_3}{\longrightarrow} A(3) \otimes C_{(3)} \langle 6 \rangle  \stackrel{h_2}{\longrightarrow}  A(3) \otimes C_{(3)}
\langle 2 \rangle \stackrel{h_1}{\longrightarrow}  A(3) \otimes C_{(3)} \stackrel{h_0}{\longrightarrow} A(3)
\end{equation*}
where
\begin{equation*}
h_n((3) \otimes 1)=
\begin{cases}
(3) & \text{ if } n=0 \\
(3) \otimes x - (3|2|3) \otimes x & \text{ if } n=1,3,\ldots \\
(3) \otimes x^2  + (3|2|3) \otimes x + (3|2|3|2|3) \otimes 1 & \text{ if } n=2,4,\ldots
\end{cases} .
\end{equation*}

Next, tensoring the complex
\begin{equation*}
\cdots \longrightarrow A(3) \otimes C_{(3)} \langle 8 \rangle \stackrel{h_3}{\longrightarrow} A(3) \otimes C_{(3)} \langle 6 \rangle \stackrel{h_2}{\longrightarrow}  A(3) \otimes C_{(3)} \langle 2 \rangle \stackrel{h_1}{\longrightarrow}  A(3) \otimes C_{(3)} 
\end{equation*}
on the right over $C_{(3)}$ with $(3)A$ we get that $\hat{p}_{(3)}$ is quasi-isomorphic to
\begin{equation}
\label{projresofbigproj4}
\cdots \rightarrow A(3) \otimes (3)A \langle 8 \rangle \stackrel{f_3}{\longrightarrow} A(3) \otimes (3)A \langle 6 \rangle \stackrel{f_2}{\longrightarrow}  A(3) \otimes (3)A \langle 2 \rangle \stackrel{f_1}{\longrightarrow}  A(3) \otimes (3)A
\end{equation}

where
\begin{equation*}
f_n((3) \otimes (3))=
\begin{cases}
(3) \otimes (3|2|3) - (3|2|3) \otimes (3) & \text{ for } n \text{ odd } \\
(3) \otimes (3|2|3|2|3)  + (3|2|3) \otimes (3|2|3) + (3|2|3|2|3) \otimes (3) & \text{ for } n \text{ even. }
\end{cases} 
\end{equation*}
Noticing that the bimodule $A(3) \otimes (3)A$ is isomorphic to the composition of projective functors
$\hat{\mathcal{E}}_0 \hat{\mathcal{F}}_1$
the complex in ~\eqref{projresofbigproj2} could be understood as a complex of projective functors
\begin{equation*}
\cdots \longrightarrow \hat{\mathcal{E}}_0 \hat{\mathcal{F}}_1 \langle 8 \rangle \stackrel{f_3}{\longrightarrow} \hat{\mathcal{E}}_0 \hat{\mathcal{F}}_1 \langle 6 \rangle \stackrel{f_2}{\longrightarrow}  \hat{\mathcal{E}}_0 \hat{\mathcal{F}}_1 \langle 2 \rangle \stackrel{f_1}{\longrightarrow}  \hat{\mathcal{E}}_0 \hat{\mathcal{F}}_1.
\end{equation*}

\subsubsection{The functor $\hat{p}_{(2,1)}$}
\label{p2,1}
Let $C_{(2,1)}=\End_A((2)A \oplus (3)A)$.
Define functors
\begin{equation*}
\hat{\pi}_{(2,1)} \colon \gmod-A \rightarrow \gmod-C_{(2,1)} \hspace{.5in} \hat{\pi}_{(2,1)}(M)=M \otimes_A (A(2) \oplus A(3))
\end{equation*}
and
\begin{equation*}
\hat{\iota}_{(2,1)} \colon D^\triangledown(\gmod-C_{(2,1)}) \rightarrow D^\triangledown(\gmod-A) \hspace{.3in}
\hat{\iota}_{(2,1)}(M)=M \otimes^{\mathbf{L}}_{C_{(2,1)}} ((2)A \oplus (3)A).
\end{equation*}
The functor categorifying the Jones-Wenzl projector $V_1^{\otimes 3} \rightarrow V_2 \otimes V_1 \rightarrow V_1^{\otimes 3}$  is then the composite
\begin{equation*}
\hat{p}_{(2,1)} \colon D^\triangledown(\gmod-A) \rightarrow D^\triangledown(\gmod-A) \hspace{.5in} 
\hat{p}_{(2,1)} =  \hat{\iota}_{(2,1)} \circ \hat{\pi}_{(2,1)}.
\end{equation*}

We now construct an explicit complex of $(A,A)$-bimodules which is \\ quasi-isomorphic to the functor $\hat{p}_{(2,1)}$.
The first step is to resolve the $(A,C_{(2,1)})$-bimodule $A(2) \oplus A(3)$ as a projective right $C_{(2,1)}$-module:

\begin{equation}
\label{projresofbigproj21}
\xymatrix{
A(2) \otimes (2)C_{(2,1)} \langle 4 \rangle \ar[r]^{\hspace{0in} h_2} 
&
{\begin{matrix}
A(2) \otimes (3) C_{(2,1)} \langle 1 \rangle \\ \oplus \\ A(3) \otimes (2) C_{(2,1)} \langle 1 \rangle
\end{matrix} \ar[r]^{\hspace{.15in} h_1}}
&
{\begin{matrix}
A(2) \otimes (2)C_{(2,1)} \\ \oplus  \\ A(3) \otimes (3)C_{(2,1)} 
\end{matrix} \ar[r]^{\hspace{.3in} h_0}}
&
{\begin{matrix} A(2) \\ \oplus \\ A(3) \end{matrix}}
\\
A(2) \otimes (2)C_{(2,1)} \langle 6 \rangle \ar[u]^{\hspace{.3in} h_3}
&
\cdots \ar[l]_{\hspace{.3in} h_4}
&
&
}
\end{equation}
where
\begin{align*}
&h_0((2) \otimes (2))=(2), &h_1((2) \otimes (3))=(2) \otimes (2|3) - (2|3) \otimes (3),  \\
&h_0((3) \otimes (3))=(3), &h_1((3) \otimes (2))=(3|2) \otimes (2) - (3) \otimes (3|2),
\end{align*}
\begin{equation*}
h_2((2) \otimes (2))= (2) \otimes (3|2|3|2) + (2|3|2) \otimes (3|2)
-(2|3) \otimes (2|3|2) - (2|3|2|3) \otimes (2),
\end{equation*}
\begin{equation*}
h_n((2) \otimes (2))=
\begin{cases}
(2) \otimes (2|3|2) - (2|3|2) \otimes (2) & \text{ if } n=3,5, \ldots \\
(2) \otimes (2|3|2) + (2|3|2) \otimes (2) & \text{ if } n=4,6, \ldots
\end{cases} .
\end{equation*}

Next, tensoring the complex
\begin{equation*}
\xymatrix{
A(2) \otimes (2)C_{(2,1)} \langle 4 \rangle \ar[r]^{h_2}  
&
{\begin{matrix}
A(2) \otimes (3) C_{(2,1)} \langle 1 \rangle \\ \oplus \\ A(3) \otimes (2) C_{(2,1)} \langle 1 \rangle
\end{matrix} \ar[r]^{h_1}} 
&
{\begin{matrix}
A(2) \otimes (2)C_{(2,1)} \\ \oplus  \\ A(3) \otimes (3)C_{(2,1)} 
\end{matrix}} \\
A(2) \otimes (2)C_{(2,1)} \langle 6 \rangle \ar[u]^{h_3}
& \cdots \ar[l]_{\hspace{.3in} h_4}
&
}
\end{equation*}
on the right over $C_{(2,1)}$ with $(2)A \oplus (3)A$ we get that $\hat{p}_{(2,1)}$ is quasi-isomorphic to
\begin{equation*}
\cdots \longrightarrow 
A(2) \otimes (2)A \langle 6 \rangle \stackrel{f_3}{\longrightarrow} 
A(2) \otimes (2)A \langle 4 \rangle  \stackrel{f_2}{\longrightarrow}  
\begin{matrix}
A(2) \otimes (3) A \langle 1 \rangle \\ \oplus \\ A(3) \otimes (2)A \langle 1 \rangle
\end{matrix} 
\stackrel{f_1}{\longrightarrow}  
\begin{matrix}
A(2) \otimes (2)A \\ \oplus  \\ A(3) \otimes (3)A
\end{matrix}
\end{equation*}
where
\begin{align*}
&f_1((2) \otimes (3))=(2) \otimes (2|3) - (2|3) \otimes (3)  \\
&f_1((3) \otimes (2))=(3|2) \otimes (2) - (3) \otimes (3|2)
\end{align*}
\begin{equation*}
f_2((2) \otimes (2))= (2) \otimes (3|2|3|2) + (2|3|2) \otimes (3|2)
-(2|3) \otimes (2|3|2) - (2|3|2|3) \otimes (2)
\end{equation*}

\begin{equation*}
f_n((2) \otimes (2))=
\begin{cases}
(2) \otimes (2|3|2) - (2|3|2) \otimes (2) & \text{ if } n=3,5, \ldots \\
(2) \otimes (2|3|2) + (2|3|2) \otimes (2) & \text{ if } n=4,6, \ldots 
\end{cases} .
\end{equation*}

\subsubsection{Functor $\hat{p}_{(1,2)}$}
Let $C_{(1,2)}=\End_A((1)A \oplus (3)A)$.
Define functors
\begin{equation*}
\hat{\pi}_{(1,2)} \colon \gmod-A \rightarrow \gmod-C_{(1,2)} \hspace{.5in} \hat{\pi}_{(1,2)}(M)=M \otimes_A (A(1) \oplus A(3))
\end{equation*}
and
\begin{equation*}
\hat{\iota}_{(1,2)} \colon D^\triangledown(\gmod-C_{(1,2)}) \rightarrow D^\triangledown(\gmod-A) \hspace{.3in}
\hat{\iota}_{(1,2)}(M)=M \otimes^{\mathbf{L}}_{C_{(1,2)}} ((1)A \oplus (3)A).
\end{equation*}
The functor categorifying the Jones-Wenzl projector $V_1^{\otimes 3} \rightarrow V_1 \otimes V_2 \rightarrow V_1^{\otimes 3}$  is then the composite
\begin{equation*}
\hat{p}_{(1,2)} \colon D^\triangledown(\gmod-A) \rightarrow D^\triangledown(\gmod-A) \hspace{.4in} 
\hat{p}_{(1,2)} =  \hat{\iota}_{(1,2)} \circ \hat{\pi}_{(1,2)}.
\end{equation*}

Just as in Section ~\ref{p2,1}, we construct an explicit complex of $(A,A)$-bimodules which is quasi-isomorphic to the functor $\hat{p}_{(1,2)}$:

\begin{equation*}
\cdots \longrightarrow 
\begin{matrix}
A(1) \otimes (1)A \langle 6 \rangle \\ \oplus \\ A(3) \otimes (3)A \langle 6 \rangle \\ \oplus \\ A(3) \otimes (1)A \langle 6 \rangle \\ \oplus \\ A(1) \otimes (3)A \langle 6 \rangle
\end{matrix} 
\stackrel{f_3}{\longrightarrow} 
\begin{matrix}
A(1) \otimes (1)A \langle 4 \rangle \\ \oplus \\ A(3) \otimes (3)A \langle 4 \rangle \\ \oplus \\ A(3) \otimes (1)A \langle 4 \rangle \\ \oplus \\ A(1) \otimes (3)A \langle 4 \rangle
\end{matrix}
\stackrel{f_2}{\longrightarrow}  
\begin{matrix}
A(3) \otimes (1) A \langle 2 \rangle \\ \oplus \\ A(1) \otimes (3)A \langle 2 \rangle \\ \oplus \\ A(3) \otimes (3)A \langle 2 \rangle
\end{matrix} 
\stackrel{f_1}{\longrightarrow}  
\begin{matrix}
A(1) \otimes (1)A \\ \oplus  \\ A(3) \otimes (3)A
\end{matrix}
\end{equation*}

where
\begin{align*}
&f_1((3) \otimes (1))=(3|2|1) \otimes (1) - (3) \otimes (3|2|1)  \\
&f_1((1) \otimes (3))=(1) \otimes (1|2|3) - (1|2|3) \otimes (3) \\
&f_1((3) \otimes (3))=(3|2|3) \otimes (3) - (3) \otimes (3|2|3) 
\end{align*}
\begin{align*}
&f_2((1) \otimes (1))= (1|2|3) \otimes (1) -(1) \otimes (3|2|1) \\
&f_2((3) \otimes (3))= (3) \otimes (1|2|3) -(3|2|1) \otimes (3)-(3|2|3) \otimes (3) - (3) \otimes (3|2|3) \\
&f_2((3) \otimes (1))= (3|2|3) \otimes (1) +(3) \otimes (3|2|1) \\
&f_2((1) \otimes (3))= (1) \otimes (3|2|3) -(1|2|3) \otimes (3) 
\end{align*}
and for $n \geq 3$
\begin{align*}
f_n((1) \otimes (1))= &(1|2|3) \otimes (1) + (1) \otimes (3|2|1) \\
f_n((3) \otimes (3))=&(-1)^{\lfloor \frac{n+4}{2} \rfloor } (3) \otimes (1|2|3) + (-1)^{\lfloor \frac{n+4}{2} \rfloor } (3|2|1) \otimes (3) +  (3) \otimes (3|2|3)\\ 
&+  (-1)^n    (3|2|3) \otimes (3) \\
f_n((3) \otimes (1))=& (3|2|3) \otimes (1) +  (-1)^{\lfloor \frac{n-1}{2} \rfloor}    (3) \otimes (3|2|1)   -(3|2|1) \otimes (1) \\
f_n((1) \otimes (3))= &(-1)^{n+1} (1) \otimes (3|2|3)    + (1) \otimes (1|2|3) + (-1)^{\lfloor \frac{n-1}{2} \rfloor}  (1|2|3) \otimes (3). 
\end{align*}
\begin{remark}{\rm 
By Theorem ~\ref{braid}, the categories $D^\triangledown(\gmod-C_{(2,1)})$ and $D^\triangledown(\gmod-C_{(1,2)})$ are equivalent.}
\end{remark}

\subsection{The unknot coloured by $V_2$} 
\begin{multicols}{2}
Consider the quiver $\Gamma$, where each unoriented edge represents two oriented edges in opposite directions. Then we may describe the graded category $ \mathcal{O}_{2}(\mathfrak{gl}_4)$ as the quotient of the path algebra of this quiver modulo the following relations:
\vfill\null
\columnbreak
\begin{minipage}[b]{4cm}
\begin{equation*}
\begin{tikzpicture}[scale=0.7]
\filldraw[black](4,2) circle (1pt);
\filldraw[black](1,2) circle (1pt);
\filldraw[black](2,2) circle (1pt);
\filldraw[black](3,2) circle (1pt);
\filldraw[black](2,3) circle (1pt);
\filldraw[black](2,1) circle (1pt);

\draw (4,2) -- (3,2)[thick];
\draw (1,2) -- (2,2)[thick];
\draw (2,2) -- (3,2)[thick];
\draw (1,2) -- (2,3)[thick];
\draw (2,3) -- (3,2)[thick];
\draw (1,2) -- (2,1)[thick];
\draw (2,1) -- (3,2)[thick];

\draw (4,1.7) node{$6$};
\draw (1,1.7) node{$2$};
\draw (2,1.7) node{$1$};
\draw (3,1.7) node{$5$};
\draw (2,0.7) node{$4$};
\draw (2,3.3) node{$3$};
\draw (0, 3.8) node{The quiver $\Gamma$:};
\end{tikzpicture}
\end{equation*}
\end{minipage}
\end{multicols}
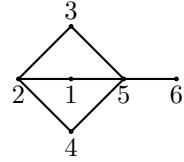
\vspace{-1.5cm}
\begin{multicols}{2}
\[
(1|2|1) = 0 \quad \quad \quad (1|5|1) = 0 
\]
\[
(6|5|1) = 0 \quad \quad \quad (1|5|6) = 0
\]
\[
(3|5|3) = 0 \quad \quad \quad (4|5|4) = 0 
\]
\[
(1|2|3) = (1|5|3) \quad \quad  (1|2|4) = (1|5|4) 
\]
\[
(4|2|3) = (4|5|3) \quad \quad  (3|2|4) = (3|5|4) 
\]
\vfill\null
\columnbreak
\[
(2|4|5) + (2|1|5)+(2|3|5)=0 
\]
\[
(5|4|2) + (5|3|2)+(5|1|2)=0  
\]
\[
(5|6|5) -(5|3|5)-(5|4|5)=0 
\]
\[
(3|2|1) = (3|5|1) \quad  \quad (4|2|1) = (4|5|1)
\]
\[
(2|3|2) = (2|1|2) \quad  \quad (2|1|2) = (2|4|2) .
\]
\\
\end{multicols}
\vspace{-1cm}
Denote this algebra by $A$.
Let $C=\End_A((1)A \oplus (5)A \oplus (6)A)$ whose indecomposable projective modules are $C(1)$, $C(5)$, and $C(6)$.  
The category of $C$-modules is fully stratified.
Define the standard modules by
\begin{equation*}
\Delta(i) = C(i) / C(<i)
\end{equation*}
where $ C(<i)$ is the image of all maps from $C(j)$ to $C(i)$ with $j<i$.
In particular $\Delta(1)=C(1)$.
We give a basis for these standard modules:
\begin{align*}
\Delta(1) = \{ & (1), (5|1), (5|3|2|1), (6|5|3|2|1) \} \\
\Delta(5) = \{ & 
(5), (5|4|5), (5|3|5), (5|6|5|6|5), (6|5), (6|5|4|5), (6|5|6|5), (6|5|3|2|3|5) \} \\
\Delta(6) = \{ & (6) \}.
\end{align*}

Now we find projective resolutions of the $\Delta(i)$ in terms of the $C(j)$.

\begin{equation*}
\Delta(1) \cong C(1).
\end{equation*}

\begin{equation*}
\Delta(5) \cong 
\xymatrix{
{\begin{matrix} C(1) \langle 1 \rangle \\ \oplus \\  C(1) \langle 3 \rangle   \end{matrix}} \ar[r]^{c}  &C(5)  
},
\quad \quad
c = 
\begin{pmatrix}
(1|5) & (1|2|3|5)
\end{pmatrix}.
\end{equation*}

\begin{equation*}
\Delta(6) \cong 
\xymatrix{
C(1) \langle 2 \rangle \ar[r]^{b} & C(5) \langle 1 \rangle \ar[r]^{c}  & C(6)  
}, \quad \quad
b = 
\begin{pmatrix}
(1|5)
\end{pmatrix}
\hspace{.2in}
c = 
\begin{pmatrix}
(5|6) 
\end{pmatrix} .
\end{equation*}

There are also proper standard modules $\bar{\Delta}(i)$ for $i=1,5,6$.
One finds
\begin{equation*}
\bar{\Delta}(1) = \Delta(1), \quad \quad
\bar{\Delta}(6) = \Delta(6), \quad \quad
\bar{\Delta}(5) = \Delta(5) / S
\end{equation*}
where $S$ is the submodule of $\Delta(5)$ generated by images of the radical of  $\End_C(\Delta(5))$.   It is actually easy to see that this submodule is the span of all elements except $(5)$ and $(6|5)$.
Thus 
\begin{equation*}
\bar{\Delta}(5)= \mathbb{C} \langle (5), (6|5) \rangle
\end{equation*}
where the only non-trivial action is $ (6|5).(5)=(6|5)$.
We may now find a filtration of $\Delta(5)$ whose subquotients are isomorphic to $\bar{\Delta}(5)$ up to shift.  Let
\begin{align*}
S_1 &= \mathbb{C} \langle (5|6|5|6|5), (6|5|3|2|3|5)   \rangle \\
S_2 &= \mathbb{C} \langle (5|6|5|6|5), (6|5|3|2|3|5), (5|4|5), (6|5|4|5)   \rangle \\
S_3 &= \mathbb{C} \langle (5|6|5|6|5), (6|5|3|2|3|5), (5|4|5), (6|5|4|5), (5|3|5), (6|5|6|5)   \rangle \\
S_4 &= \mathbb{C} \langle (5|6|5|6|5), (6|5|3|2|3|5), (5|4|5), (6|5|4|5), (5|3|5), (6|5|6|5), (5), (6|5)   \rangle .
\end{align*}
Then we obviously have $S_1 \subset S_2 \subset S_3 \subset S_4 = \Delta(5)$
and
\begin{equation*}
S_1  \cong \bar{\Delta}(5) \langle 4 \rangle, \quad
S_2 / S_1  \cong \bar{\Delta}(5) \langle 2 \rangle, \quad 
S_3 / S_2  \cong \bar{\Delta}(5) \langle 2 \rangle, \quad
S_4 / S_3  \cong \bar{\Delta}(5).
\end{equation*}

There is a resolution of the simple module $C$-module $L(1)$ (which is a quotient of $C(1)$ by proper standards:
\begin{equation*}
L^1 \cong
\xymatrix{
\bar{\Delta}(6) \langle 2 \rangle \ar[r]^{a} &
{\begin{matrix} \bar{\Delta}(5) \langle 1 \rangle \\ \oplus \\  \bar{\Delta}(5) \langle 3 \rangle   \end{matrix}} \ar[r]^{b}  & \bar{\Delta}(1)  
}
\hspace{.02in}\text{where}\hspace{.03in}
a = 
\begin{pmatrix}
(6|5) \\
0 
\end{pmatrix},
\hspace{.03in}
b = 
\begin{pmatrix}
(5|1) & (5|3|2|1)
\end{pmatrix} .
\end{equation*}

The proper standard module $\bar{\Delta}(5)$ is quasi-isomorphic to the complex
\begin{displaymath}
\xymatrix{\cdots \ar[r] & B_3  \ar[r]^{F_3}  & B_2 \ar[r]^{F_2} & B_1 \ar[r]^{F_1} & B_0}, \quad \quad B_j = \underbrace{\Delta(5) \langle 2j \rangle \oplus \cdots \oplus \Delta(5) \langle 2j \rangle}_{j+1},
\end{displaymath}
where $F_j$ is a matrix with entries give by:
\begin{equation*}
(F_j)_{kl} = 
\begin{cases}
(-1)^k(5|4|5)  &\text{ if }  k=l \text{ and } k=1, \ldots, j-1 \\
(5|3|5) &\text{ if }  l=k+1 \text{ and } k=1, \ldots, j-1.
\end{cases}
\end{equation*}
Now we construct a projective resolution of $L(1)$ in the category of $C$-modules.
\begin{equation*}
\xymatrix{
{\begin{matrix} q^7C(5) \\ \oplus  \\ q^8 C(1) \end{matrix}} \ar[d]^{\begin{pmatrix} (535)+(545)-\frac{1}{2}(515) \\ (1535) \end{pmatrix}} & & & D_5 \ar[lll]_{\begin{pmatrix} (545)-(535) \\ -(51) \end{pmatrix}} && D_6 \ar[ll] & \cdots \ar[l] \\
 q^5 C(5) \ar[rrr]_{\begin{pmatrix} \frac{1}{2} (5321) \\ (5456) - \frac{1}{2} (5656) \\ -\frac{1}{2}(51)  \end{pmatrix}}
& & &
 {\begin{matrix}
 q^2 C(1) \\
 \oplus \\
 q^2 C(6) \\
 \oplus \\
 q^4 C(1) \\
 \end{matrix}   } \ar[rr]_{\begin{pmatrix} (15)\\ (65) \\ (1235) \end{pmatrix}^T}
   & & qC(5) \ar[r]_{(51)} & C(1)
}
\end{equation*}
where
\begin{equation*}
D_{2n}=q^{4n-1}C(5) \oplus q^{4n} C(1) \hspace{.6in}
D_{2n+1}=q^{4n+1} C(5) \oplus q^{4n} C(1)
\end{equation*}
and
\begin{equation*}
\xymatrix{
 q^{4n+3} C(5)  \ar[rrr]^{(535)+(545)-\frac{1}{2}(515)}   \ar[rrrd]_{\hspace{.2in} -(5321)  } &&& q^{4n+1} C(5)   \ar[rrr]^{(545)-(535)} \ar[rrrd]_{\hspace{.2in} -(51)} &&& q^{4n-1} C(5)   \\ 
 q^{4n+4} C(1) \ar[rrru]^{ \hspace{.2in} (1535) }       &&& q^{4n} C(1) \ar[rrru]^{\hspace{.4in} (15)}	   &&& q^{4n} C(1) 
} .
\end{equation*}

\begin{corollary} \label{extcalforunknot}
In the category of $C$-modules, self extensions of $L(1)$ are given by:
\begin{align*}
& \Ext^0(L(1),L(1))  \cong \mathbb{C}, \hspace{.05in}
 \Ext^{-1}(L(1),L(1))  = 0 , \hspace{.05in}
 \Ext^{-2}(L(1),L(1))   \cong \mathbb{C} \langle -2 \rangle \oplus \mathbb{C} \langle -4 \rangle, \\
& \Ext^{-2n}(L(1),L(1))   \cong \mathbb{C} \langle -4n \rangle \hspace{.02in} \text{ if } \hspace{.03in} n \geq 2, \hspace{.03in}
 \Ext^{-(2n+1)}(L(1),L(1))   \cong \mathbb{C} \langle -4n \rangle \hspace{.02in} \text{ if } \hspace{.03in} n \geq 2.
\end{align*}
\end{corollary}
The Poincare series of this bigraded vector space is:
\begin{equation}
1+(q^{-2}+q^{-4})t^{-2} + \frac{q^{-8}t^{-4}(1+t^{-1})}{1-t^{-2}q^{-4}} .
\end{equation}
Shifting this space by $q^2 t^2$ yields the cohomology of the unknot coloured by $V_2$:
\begin{equation} \label{homofunknot}
q^2 t^2+(1+q^{-2}) + \frac{q^{-6}t^{-2}(1+t^{-1})}{1-t^{-2}q^{-4}} .
\end{equation}
Under the transformation $t \mapsto T$, $ q \mapsto T^{-1} Q^{-1}$, 
the series in \eqref{homofunknot} becomes
\begin{equation} \label{dualunknot}
\frac{Q^{-2} +1 - Q^4 T^2 + Q^6 T^3 }{1-T^2 Q^4} .
\end{equation}
This transformation comes from applying a Koszul duality functor to the homological and internal grading shifts respectively.
The transformed Poincare series in \eqref{dualunknot} is precisely the homology of the unknot coloured by the $3$-dimensional representation in \cite[Section 4.3.1]{CoKr}.
The fact that our calculation is related to the calculation in \cite{CoKr} by Koszul duality also follows from \cite{SS} where it was shown that the categorified projector considered in this paper is related to Cooper and Krushkal's via Koszul duality.
The series in \eqref{dualunknot} also agrees with \cite[Example 3.2]{GOR} up to an overall factor.  

\section{Appendix: The (coloured) Reidemeister moves}
\begin{minipage}{0.25\textwidth}
\begin{equation}
\label{cr1} 
\scalebox{0.7}{
\xy {(10,0)*{};(12.4,4)*{}**\dir{-}}; {\ar (13.6,6)*{};(16,10)*{}**\dir{-}}; {\ar (16,0)*{};(10,10)*{}**\dir{-}};
{(22,0)*{};(22,10)*{}**\dir{-}}; {(10,10)*{};(10,20)*{}**\dir{-}}; (16,10)*{};(22,10)*{}**\crv{(17,14)&(21,14)};
(16,0)*{};(22,0)*{}**\crv{(17,-4)&(21,-4)}; {(10,0)*{};(10,-10)*{}**\dir{-}}; (16,-10)*{};(22,-10)*{}**\crv{(17,-6)&(21,-6)};
{(22,-10)*{};(22,-20)*{}**\dir{-}}; (16,-20)*{};(22,-20)*{}**\crv{(17,-24)&(21,-24)}; {(10,-20)*{};(10,-30)*{}**\dir{-}};
{(10,-20)*{};(16,-10)*{}**\dir{-}}; {(16,-20)*{};(13.6,-16)*{}**\dir{-}}; {(10,-10)*{};(12.4,-14)*{}**\dir{-}};
{(12,-5)*{{m}}};
\endxy}
\scalebox{1}{
\hspace{.05in} \xy {(10,-3)*{=}}; {(16,-3)*{{{\scriptstyle{m}}}}}; 
\endxy}
\scalebox{0.7}{
 \xy {\ar (0,-30)*{};(0,20)*{}**\dir{-}};
\endxy}
\end{equation}
\end{minipage}
\begin{minipage}{0.07\textwidth}
and
\end{minipage}
\begin{minipage}{0.35\textwidth}
\begin{equation*}
\label{cr7} 
\scalebox{0.8}{
\xy {\ar (10,-10)*{};(10,0)*{}**\dir{-}}; (16,-10)*{};(22,-10)*{}**\crv{(17,-6)&(21,-6)}; {(22,-10)*{};(22,-20)*{}**\dir{-}};
(16,-20)*{};(22,-20)*{}**\crv{(17,-24)&(21,-24)}; {(10,-20)*{};(10,-30)*{}**\dir{-}}; {(10,-20)*{};(16,-10)*{}**\dir{-}};
{(16,-20)*{};(13.6,-16)*{}**\dir{-}}; {(10,-10)*{};(12.4,-14)*{}**\dir{-}}; {(28,-5)*{}};
\endxy}
\xy {(0,-11)*{=}};
\endxy
\scalebox{0.8}{
\hspace{0.05in} \xy {\ar (10,-30)*{};(10,0)*{}**\dir{-}};
\endxy}
\hspace{0.05in} \xy {(-10,-11)*{=}};
\endxy
\scalebox{0.8}{
\hspace{0.1in} \xy {\ar (10,-10)*{};(10,0)*{}**\dir{-}}; (16,-10)*{};(22,-10)*{}**\crv{(17,-6)&(21,-6)}; {(22,-10)*{};(22,-20)*{}**\dir{-}};
(16,-20)*{};(22,-20)*{}**\crv{(17,-24)&(21,-24)}; {(10,-20)*{};(10,-30)*{}**\dir{-}}; {(10,-10)*{};(16,-20)*{}**\dir{-}};
{(10,-20)*{};(12.4,-16)*{}**\dir{-}}; {(13.6,-14)*{};(16,-10)*{}**\dir{-}}; {(28,-5)*{}};
\endxy}
\end{equation*}
\end{minipage}
\begin{minipage}{0.28\textwidth}
\begin{equation}
\label{cr1detailed} \tag{$61^*$}
\scalebox{0.8}{
\xy {(10,0)*{};(12.4,4)*{}**\dir{-}}; {\ar (13.6,6)*{};(16,10)*{}**\dir{-}}; {\ar (16,0)*{};(10,10)*{}**\dir{-}};
{(22,0)*{};(22,10)*{}**\dir{-}}; {(10,10)*{};(10,20)*{}**\dir{-}}; (16,10)*{};(22,10)*{}**\crv{(17,14)&(21,14)};
(16,0)*{};(22,0)*{}**\crv{(17,-4)&(21,-4)}; {(10,0)*{};(10,-10)*{}**\dir{-}}; (16,-10)*{};(22,-10)*{}**\crv{(17,-6)&(21,-6)};
{(22,-10)*{};(22,-20)*{}**\dir{-}}; (16,-20)*{};(22,-20)*{}**\crv{(17,-24)&(21,-24)}; {(10,-20)*{};(10,-30)*{}**\dir{-}};
{(10,-20)*{};(16,-10)*{}**\dir{-}}; {(16,-20)*{};(13.6,-16)*{}**\dir{-}}; {(10,-10)*{};(12.4,-14)*{}**\dir{-}};
{(12,-25)*{\scriptstyle{m}}};
{(10,-20)*{};(22,-20)*{}**\dir{--}};
{(10,-10)*{};(22,-10)*{}**\dir{--}};
{(10,-5)*{};(22,-5)*{}**\dir{--}};
{(10,0)*{};(22,0)*{}**\dir{--}};
{(10,10)*{};(22,10)*{}**\dir{--}};
{(6,-25)*{\text{$D_1$}}};
{(6,-15)*{\text{$D_2$}}};
{(6,-8)*{\text{$D_3$}}};
{(6,-3)*{\text{$D_4$}}};
{(6,5)*{\text{$D_5$}}};
{(6,13)*{\text{$D_6$}}};
\endxy}
\hspace{.05in} 
\xy {(10,-5)*{=}};
\endxy
\scalebox{0.8}{
\xy {\ar (-30,-30)*{};(-30,20)*{}**\dir{-}}; {(-32,-25)*{\scriptstyle{m}}};
\endxy}
\end{equation}
\end{minipage}
\begin{equation}
\label{cr4} \xy {(10,0)*{};(12.4,4)*{}**\dir{-}}; { (13.6,6)*{};(16,10)*{}**\dir{-}}; { (16,0)*{};(10,10)*{}**\dir{-}}; {\ar
(22,10)*{};(22,0)*{}**\dir{-}}; (16,10)*{};(22,10)*{}**\crv{(17,14)&(21,14)}; {\ar (10,10)*{};(10,20)*{}**\dir{-}}; {(9,5)*{{m}}};
{(16,5)*{{n}}};
\endxy \hspace{.1in} \xy (10,0)*{}; (10,10)*{=};
\endxy \hspace{.1in} \xy {(16,0)*{};(22,10)*{}**\dir{-}}; {\ar (16,10)*{};(18.4,6)*{}**\dir{-}};
{(19.6,4)*{};(22,0)*{}**\dir{-}}; {(10,0)*{};(10,10)*{}**\dir{-}}; {\ar (22,10)*{};(22,20)*{}**\dir{-}};
(10,10)*{};(16,10)*{}**\crv{(11,14)&(15,14)}; {(8,5)*{{m}}}; {(16,5)*{{n}}}; \endxy \hspace{.5in} \xy
{(10,0)*{};(12.4,4)*{}**\dir{-}}; { (13.6,6)*{};(16,10)*{}**\dir{-}}; { (16,0)*{};(10,10)*{}**\dir{-}}; {\ar
(22,0)*{};(22,10)*{}**\dir{-}}; (16,10)*{};(22,10)*{}**\crv{(17,14)&(21,14)}; {\ar (10,10)*{};(10,20)*{}**\dir{-}}; {(9,5)*{{m}}};
{(16,5)*{{n}}};
\endxy
\hspace{.1in} \xy (10,0)*{}; (10,10)*{=};
\endxy
\hspace{.1in} \xy {(16,0)*{};(22,10)*{}**\dir{-}}; {(19.6,4)*{};(22,0)*{}**\dir{-}}; {(18.4,6)*{};(16,10)*{}**\dir{-}}; {\ar
(10,10)*{};(10,0)*{}**\dir{-}}; {\ar (22,10)*{};(22,20)*{}**\dir{-}}; (10,10)*{};(16,10)*{}**\crv{(11,14)&(15,14)}; {(8,5)*{{m}}};
{(16,5)*{{n}}};
\endxy
\end{equation}

\begin{equation}
\label{cr2} \xy {(10,0)*{};(12.4,4)*{}**\dir{-}}; { (13.6,6)*{};(16,10)*{}**\dir{-}}; { (16,0)*{};(10,10)*{}**\dir{-}}; {\ar
(10,10)*{};(16,20)*{}**\dir{-}}; {(16,10)*{};(13.6,14)*{}**\dir{-}}; {\ar (12.4,16)*{};(10,20)*{}**\dir{-}}; {(7,10)*{{m}}};
{(19,10)*{{n}}}; {(21,10)*{}};
\endxy
\hspace{0.1in} \xy {(10,0)*{}}; {(10,10)*{=}};
\endxy
\hspace{.1in} \xy {\ar (10,0)*{};(10,20)*{}**\dir{-}}; {(7,10)*{{m}}}; {\ar (16,0)*{};(16,20)*{}**\dir{-}}; {(19,10)*{{n}}};
\endxy
\hspace{.1in} \xy {(10,0)*{}}; {(10,10)*{=}};
\endxy
\hspace{.1in} \xy {(10,10)*{};(12.4,14)*{}**\dir{-}}; {\ar (13.6,16)*{};(16,20)*{}**\dir{-}}; {\ar (16,10)*{};(10,20)*{}**\dir{-}};
{(10,0)*{};(16,10)*{}**\dir{-}}; {(16,0)*{};(13.6,4)*{}**\dir{-}}; {(12.4,6)*{};(10,10)*{}**\dir{-}}; {(7,10)*{{m}}};
{(19,10)*{{n}}};
\endxy
\end{equation}

\begin{equation}
\label{cr3} \xy {(10,0)*{};(12.4,4)*{}**\dir{-}}; { (13.6,6)*{};(16,10)*{}**\dir{-}}; { (16,0)*{};(10,10)*{}**\dir{-}};
{(22,0)*{};(22,10)*{}**\dir{-}}; {(16,10)*{};(18.4,14)*{}**\dir{-}}; { (19.6,16)*{};(22,20)*{}**\dir{-}}; {
(22,10)*{};(16,20)*{}**\dir{-}}; {(10,10)*{};(10,20)*{}**\dir{-}}; {(10,20)*{};(12.4,24)*{}**\dir{-}}; {\ar
(13.6,26)*{};(16,30)*{}**\dir{-}}; {\ar (16,20)*{};(10,30)*{}**\dir{-}}; {\ar (22,20)*{};(22,30)*{}**\dir{-}}; {(8,18)*{{m}}};
{(15,18)*{{n}}}; {(24,18)*{{p}}};
\endxy
\hspace{0.1in} \xy {(10,0)*{}}; {(10,15)*{=}};
\endxy
\hspace{.1in} \xy {(16,0)*{};(18.4,4)*{}**\dir{-}}; { (19.6,6)*{};(22,10)*{}**\dir{-}}; { (22,0)*{};(16,10)*{}**\dir{-}};
{(10,0)*{};(10,10)*{}**\dir{-}}; {(16,20)*{};(18.4,24)*{}**\dir{-}}; {\ar (19.6,26)*{};(22,30)*{}**\dir{-}}; {\ar
(22,20)*{};(16,30)*{}**\dir{-}}; {\ar (10,20)*{};(10,30)*{}**\dir{-}}; {(10,10)*{};(12.4,14)*{}**\dir{-}}; {
(13.6,16)*{};(16,20)*{}**\dir{-}}; { (16,10)*{};(10,20)*{}**\dir{-}}; {(22,10)*{};(22,20)*{}**\dir{-}}; {(8,18)*{{{m}}}};
{(17,18)*{{n}}}; {(24,18)*{{p}}};
\endxy
\hspace{.5in} \xy {(10,0)*{};(16,10)*{}**\dir{-}}; {(16,0)*{};(13.6,4)*{}**\dir{-}}; {(12.4,6)*{};(10,10)*{}**\dir{-}};
{(22,0)*{};(22,10)*{}**\dir{-}}; {(16,10)*{};(22,20)*{}**\dir{-}}; {(22,10)*{};(19.6,14)*{}**\dir{-}}; {(18.4,16)*{};(16,20)*{}**\dir{-}};
{(10,10)*{};(10,20)*{}**\dir{-}}; {\ar (10,20)*{};(16,30)*{}**\dir{-}}; {(16,20)*{};(13.6,24)*{}**\dir{-}}; {\ar
(12.4,26)*{};(10,30)*{}**\dir{-}}; {\ar (22,20)*{};(22,30)*{}**\dir{-}}; {(8,18)*{{m}}}; {(15,18)*{{n}}}; {(24,18)*{{p}}};
\endxy
\hspace{0.1in} \xy {(10,0)*{}}; {(10,15)*{=}};
\endxy
\hspace{.1in} \xy {(16,0)*{};(22,10)*{}**\dir{-}}; {(22,0)*{};(19.6,4)*{}**\dir{-}}; {(18.4,6)*{};(16,10)*{}**\dir{-}};
{(10,0)*{};(10,10)*{}**\dir{-}}; {\ar (16,20)*{};(22,30)*{}**\dir{-}}; {(22,20)*{};(19.6,24)*{}**\dir{-}}; {\ar
(18.4,26)*{};(16,30)*{}**\dir{-}}; {\ar (10,20)*{};(10,30)*{}**\dir{-}}; {(10,10)*{};(16,20)*{}**\dir{-}}; {
(16,10)*{};(13.6,14)*{}**\dir{-}}; { (12.4,16)*{}; (10,20)*{}**\dir{-}}; {(22,10)*{};(22,20)*{}**\dir{-}}; {(8,18)*{{{m}}}};
{(17,18)*{{n}}}; {(24,18)*{{p}}};
\endxy
\end{equation}

\begin{minipage}{1.0\textwidth}

\begin{equation}
\label{cr5} \xy {\ar(22,0)*{};(22,10)*{}**\dir{-}}; (16,0)*{};(22,0)*{}**\crv{(17,-4)&(21,-4)};
{(10,0)*{};(10,-10)*{}**\dir{-}}; (10,0)*{};(16,0)*{}**\crv{(11,4)&(15,4)}; {(12,-5)*{{m}}};
\endxy
\hspace{.05in} \xy {(10,0)*{=}};
\endxy
\hspace{0.05in} \xy {\ar (10,-10)*{};(10,10)*{}**\dir{-}}; {(12,-5)*{{m}}};
\endxy
\hspace{.0in} \xy {(10,0)*{=}};
\endxy
\hspace{0.05in} \xy {\ar(10,0)*{};(10,10)*{}**\dir{-}}; (10,0)*{};(16,0)*{}**\crv{(11,-4)&(15,-4)};
{(22,0)*{};(22,-10)*{}**\dir{-}}; (16,0)*{};(22,0)*{}**\crv{(17,4)&(21,4)}; {(20,-5)*{{m}}};
\endxy
\hspace{.35in} \xy {(22,10)*{};(22,0)*{}**\dir{-}}; (16,0)*{};(22,0)*{}**\crv{(17,-4)&(21,-4)};
{\ar(10,0)*{};(10,-10)*{}**\dir{-}}; (10,0)*{};(16,0)*{}**\crv{(11,4)&(15,4)}; {(12,-5)*{{m}}};
\endxy
\hspace{.05in} \xy {(10,0)*{=}};
\endxy
\hspace{0.05in} \xy {\ar (10,10)*{};(10,-10)*{}**\dir{-}}; {(12,-5)*{{m}}};
\endxy
\hspace{.0in} \xy {(10,0)*{=}};
\endxy
\hspace{0.05in} \xy {(10,10)*{};(10,0)*{}**\dir{-}}; (10,0)*{};(16,0)*{}**\crv{(11,-4)&(15,-4)}; {\ar
(22,0)*{};(22,-10)*{}**\dir{-}}; (16,0)*{};(22,0)*{}**\crv{(17,4)&(21,4)}; {(20,-5)*{{m}}};
\endxy
\end{equation}
\end{minipage}

\begin{equation}
\label{cr6} \xy {(10,-30)*{};(10,-20)*{}**\dir{-}}; {\ar (16,-20)*{};(16,-30)*{}**\dir{-}};
(22,-20)*{};(40,-20)*{}**\crv{(23,-29)&(39,-29)}; (28,-20)*{};(34,-20)*{}**\crv{(29,-24)&(33,-24)}; {(10,-20)*{};(10,-10)*{}**\dir{-}};
{(16,-20)*{};(16,-10)*{}**\dir{-}}; {(34,-20)*{};(34,-10)*{}**\dir{-}}; {(40,-20)*{};(40,-10)*{}**\dir{-}};
{(22,-10)*{};(28,-20)*{}**\dir{-}}; {(22,-20)*{};(24.4,-16)*{}**\dir{-}}; {(25.6,-14)*{};(28,-10)*{}**\dir{-}};
{(34,-10)*{};(34,0)*{}**\dir{-}}; {(40,-10)*{};(40,0)*{}**\dir{-}}; (10,-10)*{};(28,-10)*{}**\crv{(11,-1)&(27,-1)};
(16,-10)*{};(22,-10)*{}**\crv{(17,-6)&(21,-6)}; (10,10)*{};(28,10)*{}**\crv{(11,1)&(27,1)}; (16,10)*{};(22,10)*{}**\crv{(17,6)&(21,6)};
{(34,0)*{};(34,10)*{}**\dir{-}}; {(40,0)*{};(40,10)*{}**\dir{-}}; {(10,10)*{};(10,20)*{}**\dir{-}}; {(16,10)*{};(16,20)*{}**\dir{-}};
{(22,10)*{};(28,20)*{}**\dir{-}}; {(28,10)*{};(25.6,14)*{}**\dir{-}}; {(22,20)*{};(24.4,16)*{}**\dir{-}};
{(34,10)*{};(34,20)*{}**\dir{-}}; {(40,10)*{};(40,20)*{}**\dir{-}}; {\ar (10,20)*{};(10,30)*{}**\dir{-}};
{(16,30)*{};(16,20)*{}**\dir{-}}; (22,20)*{};(40,20)*{}**\crv{(23,29)&(39,29)}; (28,20)*{};(34,20)*{}**\crv{(29,24)&(33,24)};
{(12,-25)*{{n}}}; {(18,-25)*{{m}}};
\endxy
\hspace{.1in} \xy {(10,0)*{=}};\endxy \hspace{.1in} \xy {\ar (10,-30)*{};(10,30)*{}**\dir{-}}; {\ar (16,30)*{};(16,-30)*{}**\dir{-}};
{(12,-25)*{{n}}}; {(18,-25)*{{m}}};\endxy \hspace{.1in} \xy {(10,0)*{=}};\endxy \hspace{.1in} \xy
{(10,-30)*{};(10,-20)*{}**\dir{-}}; {\ar (16,-20)*{};(16,-30)*{}**\dir{-}}; (22,-20)*{};(40,-20)*{}**\crv{(23,-29)&(39,-29)};
(28,-20)*{};(34,-20)*{}**\crv{(29,-24)&(33,-24)}; {(10,-20)*{};(10,-10)*{}**\dir{-}}; {(16,-20)*{};(16,-10)*{}**\dir{-}};
{(34,-20)*{};(34,-10)*{}**\dir{-}}; {(40,-20)*{};(40,-10)*{}**\dir{-}}; {(25.6,-16)*{};(28,-20)*{}**\dir{-}};
{(22,-10)*{};(24.4,-14)*{}**\dir{-}}; {(22,-20)*{};(28,-10)*{}**\dir{-}}; {(34,-10)*{};(34,0)*{}**\dir{-}};
{(40,-10)*{};(40,0)*{}**\dir{-}}; (10,-10)*{};(28,-10)*{}**\crv{(11,-1)&(27,-1)}; (16,-10)*{};(22,-10)*{}**\crv{(17,-6)&(21,-6)};
(10,10)*{};(28,10)*{}**\crv{(11,1)&(27,1)}; (16,10)*{};(22,10)*{}**\crv{(17,6)&(21,6)}; {(34,0)*{};(34,10)*{}**\dir{-}};
{(40,0)*{};(40,10)*{}**\dir{-}}; {(22,10)*{};(24.4,14)*{}**\dir{-}}; {(28,20)*{};(25.6,16)*{}**\dir{-}}; {(16,10)*{};(16,20)*{}**\dir{-}};
{(28,10)*{};(22,20)*{}**\dir{-}}; {(34,10)*{};(34,20)*{}**\dir{-}}; {(40,10)*{};(40,20)*{}**\dir{-}}; {\ar
(10,20)*{};(10,30)*{}**\dir{-}}; {(10,10)*{};(10,20)*{}**\dir{-}}; {(16,20)*{};(16,30)*{}**\dir{-}};
(22,20)*{};(40,20)*{}**\crv{(23,29)&(39,29)}; (28,20)*{};(34,20)*{}**\crv{(29,24)&(33,24)}; {(12,-25)*{{n}}}; {(18,-25)*{{m}}};
\endxy
\end{equation}

\vspace{1cm}

C.S.\\ Department of Mathematics, Endenicher Allee 60, 53115 Bonn  (Germany).\\ email: \email{stroppel@math.uni-bonn.de}

J.S.\\ Department of Mathematics, 1650 Bedford Avenue, 11225 Brooklyn (US). \\ email: \email{jsussan@mec.cuny.edu}  \\
Mathematics Program, The Graduate Center, CUNY, 365 Fifth Avenue,\\ 10016 New York (US). \\ email:  \email{jsussan@gc.cuny.edu}

\end{document}